\documentclass[11pt, reqno]{amsart}
\usepackage{fullpage}
\usepackage{comment}

\usepackage{amssymb,amsmath,amsthm,amscd,mathrsfs,graphicx, color}
\usepackage[cmtip,all,matrix,arrow,tips,curve]{xy}
 \usepackage{hyperref}
\usepackage[usenames,dvipsnames]{xcolor}
\usepackage{xypic}
\hypersetup{colorlinks=true,citecolor=ForestGreen,linkcolor=Maroon,urlcolor=NavyBlue}
\usepackage{mathtools}
\usepackage{mathpazo}

\usepackage[normalem]{ulem}

\usepackage{cleveref}

\usepackage{enumitem}

\numberwithin{equation}{section}
\newtheorem{teo}{Theorem}[section]
\newtheorem{pro}[teo]{Proposition}
\newtheorem{lem}[teo]{Lemma}
\newtheorem{cor}[teo]{Corollary}

\newtheorem{wrn}[teo]{Warning}
\newtheorem{condition}[teo]{Condition}

\newtheorem{teoalpha}{Theorem}

\theoremstyle{definition}

\theoremstyle{remark}
\newtheorem{rem}[teo]{Remark}

\newenvironment{alphabetize}{\begin{enumerate}[label=(\alph*)]
}{\end{enumerate}}

\newif\ifHideFoot
\HideFootfalse   \ifHideFoot

\newcommand{\Yano}[1]{}
\newcommand{\Jeff}[1]{}
\newcommand{\Charles}[1]{}

\else 

\newcommand{\marg}[1]{\normalsize{{
\color{red}\footnote{{\color{blue}#1}}}{\marginpar[\vskip
-.25cm{\color{red}\hfill\tiny\thefootnote$\implies$}]{\vskip
-.2cm{\color{red}$\impliedby$\tiny\thefootnote}}}}}
\newcommand{\Yano}[1]{\marg{(Yano) #1}}
\newcommand{\Jeff}[1]{\marg{(Jeff) #1}}
\newcommand{\Charles}[1]{\marg{(Charles) #1}}

\fi

\global\let\hom\undefined
\DeclareMathOperator{\hom}{Hom}

\newcommand{\til}[1]{{\widetilde{#1}}}
\def\et{{\rm \acute et}}
\def\ord{{\rm ord}}
\def\mmu{{\pmb\mu}}

\def\inv{^{-1}}

\def\cx{{\mathbb C}}
\def\ff{{\mathbb F}}

\def\rat{{\mathbb Q}}
\def\integ{{\mathbb Z}}

\def\iso{\cong}

\renewcommand{\bar}[1]{{\overline{#1}}}

\DeclareMathOperator{\aut}{Aut}
\DeclareMathOperator{\Inv}{Inv}

\DeclareMathOperator{\End}{End}

\DeclareMathOperator{\gal}{Gal}

\DeclareMathOperator{\spec}{Spec}
\DeclareMathOperator{\pic}{Pic}

\DeclareMathOperator{\A}{A}
\DeclareMathOperator{\chow}{CH}
\DeclareMathOperator{\tr}{tr}
\DeclareMathOperator{\K}{K}

\DeclareMathOperator{\coker}{coker}

\DeclareMathOperator{\absfrob}{fr}
\DeclareMathOperator{\im}{im}

\def\iff{\quad\Longleftrightarrow\quad}

\def\nori{\mathrm{Nori}}
\def\ab{\mathrm{ab}}
\def\et{\mathrm{et}}

\def\abvar{\mathrm{ab}}
\def\abelian{\mathrm{abn}}
\def\finflat{\mathrm{ff}}
\def\nonabelian{\mathrm{n-ab}}

\newcommand{\st}[1]{\left\{#1\right\}}
\newcommand{\abs}[1]{{\left|#1\right|}}
\newcommand{\powser}[1]{[\![#1]\!]}

\newcommand{\invlim}[1]{\lim_{\stackrel{\leftarrow}{#1}}}

\newcommand{\rest}[1]{|_{#1}}

 \mathtoolsset{centercolon}

\title[Prym varieties]{Putting the p back in Prym}

\author{Jeffrey D. Achter}
\address{Colorado State University, Department of Mathematics,
Fort Collins, CO 80523,
USA}
\email{j.achter@colostate.edu}

\author{Sebastian Casalaina-Martin }
\address{University of Colorado, Department of Mathematics, 
Boulder, CO 80309, USA }
\email{casa@math.colorado.edu}

\thanks{
  The first- and second-named authors were partially supported by respective Simons Foundation grants 637075 and 581058.}

\date{\today}

\begin{document}

\maketitle

\begin{abstract}
After Jacobians of curves, 
Prym varieties are perhaps the next most studied abelian varieties.  They turn out to be quite useful in a number of contexts.  For technical reasons, there does not appear to be any systematic treatment of Prym varieties in characteristic 2, and due to our recent interest in this topic, the purpose of this paper is to fill in that gap.  Our main result is a classification of branched covers of curves in characteristic 2 that give rise to Prym varieties.  We are also interested in the case of Prym varieties in the relative setting, and so we develop that theory here as well, including an extension of Welters' Criterion. 
\end{abstract}

\section*{Introduction}

Consider an \'etale double cover of smooth complex projective curves $f\colon C\to C'$.  It turns out
that the principal polarization of $\pic^0(C)$ defines a
canonical principal polarization on the Prym variety $P(C/C')$, a
sub-abelian variety of $\pic^0(C)$ complementary to
$f^*\pic^0(C')$. This phenomenon, discovered in the late 19th century
by complex geometers,  was recast in purely
algebro-geometric language by Mumford in \cite{mumford74}; see \cite{farkas-prymsurvey} for a
  social and mathematical history of these ideas.

After Jacobians of curves themselves, Pryms are one of the most
accessible classes of principally polarized abelian
varieties. It is  difficult to improve on the assessment by
  Welters \cite{welters87}: \emph{Jacobi varieties are certainly the best understood
    principally polarized abelian varieties (PPAV).  In this sense
    they are followed immediately by Prym varieties.  Beyond these two
    classes, however, darkness predominates generally in the geometry
    of PPAV.}    In
particular, for small values of $g$, the Prym construction lets one
understand the geometry $\mathcal A_g$ in terms of curves.  The
Torelli map identifies $\mathcal M_g$ with an open subspace of
$\mathcal A_g$ only when $g \le 3$; but for $g \in \st{4,5}$, one at
least knows that a dense subset of principally polarized abelian varieties are 
Prym varieties.  For details of this and similar geometric insights afforded by Pryms, we refer the reader to \cite{beauville-prymsurvey} and \cite{farkas-prymsurvey}.

On one hand, the literature is rich with arithmetic applications of
the theory of Prym varieties.  For example, in \cite{beauvilleschottky},
Beauville constructs a moduli stack of Prym covers over $\integ[1/2]$, together with a so-called Prym map to the moduli stack of principally polarized abelian schemes,
and uses Prym varieties in  \cite{beauville77}  to study the Chow groups and rationality of fibrations in quadrics over algebraically
closed fields whose characteristic is not two.  Faber and van der Geer
\cite{fabervandergeer04} build on Beauville's work to analyze complete families of
curves, again away from characteristic two.  Over finite fields of odd
characteristic, point counts sometimes are able to distinguish Prym
varieties from arbitrary abelian varieties (e.g., \cite{perret06}).

On the other hand, 
while it is natural, of course, to hope to study these
phenomena in characteristic two as well, 
 the foundational literature is notably quiet on the
question of constructing Pryms in general contexts.  For example,
Mumford's original algebraic treatment \cite{mumford74} works over an
algebraically closed field $k$ of characteristic other than two. Beauville works in the same setting in \cite{beauville77}  and \cite{beauvilleschottky}, although in the former, he does allude to results over an algebraically closed field of characteristic $2$.   
Tyurin \cite{tyurin72} retreats to the case where $k = \cx$, while
Welters \cite{welters87} allows an arbitrary algebraically closed
field of characteristic zero.  The most self-contained exposition the
authors are aware of, namely the treatment by Birkenhake and Lange
\cite{BL}, is resolutely over the complex numbers.

Speaking from personal experience, it seems to the authors that the
further the base field is from the complex numbers (for example, if
the field is not algebraically closed, or not even perfect, or has
characteristic two), the less certain one might be about the status of
certain constructions and assertions regarding Prym varieties.
Our intention in this paper is to develop a comprehensive theory of
Prym, and Prym--Tyurin, varieties that works over an arbitrary base.
Even in the special case where the base scheme is $S =
\spec \bar\ff_2 $ our results, while perhaps not surprising, do not
appear in the literature.

For better or worse, our main result shows that there is
nothing new under the sun, even if the landscape is a 
scheme $S$ on which $2$ is not invertible.

\begin{teoalpha}
  Let $S$ be a connected scheme.  Let $f\colon C \to C'$ be a
  finite $S$-morphism of degree $d$ of smooth proper curves over $S$ with
  geometrically connected fibers of respective genera  $g > g' \ge 1$, and suppose that $f$ is fiberwise separable. 
  \begin{enumerate}
    \item There exists a canonical complementary abelian scheme
      $P(C/C')$ for $f^*\pic^0_{C'/S}$ in $\pic^0_{C/S}$ of relative
      dimension $h = g-g'$ over $S$.
    \item The principal polarization $\Theta_{C/S}$ of $\pic^0_{C/S}$ restricts to
        a multiple $e\cdot \Xi$ of a principal polarization $\Xi$ of $P(C/C')$ if and
        only if, for some geometric point $s$ of $ S$, the morphism $f_s$ has one of the
        following types:
\begin{enumerate}        
        \item $d=2$ and $f_s$ is \'etale, in which case $e=2$ and
          $h=g'-1$;
          
\item   $d=2$ and the ramification divisor of $f_s$ has degree $2$, in
  which case $e=2$ and $h= g'$;

\item $d=3$, $f_s$ is \'etale and noncyclic, and $g'=2$, in which case $e=3$ and $h=2$;
    
  \item $g = 2$ and $g'=1$, in which case $e=\deg f_s/\deg f_s^\abelian$ and $h=g'=1$, where $f_s^\abelian$ is the maximal abelian subcover of $f_s$ 
 (see \eqref{E:MaxAbCovFact} and  \Cref{L:newBLprop11.4.3}).
       \end{enumerate}
\end{enumerate}
\end{teoalpha}

This is proven in \Cref{T:classifyprym} and
\Cref{T:classifyprymS}. The abelian scheme $P(C/C')$ of part (1) is constructed as follows.
The morphism $f:C \to C'$ induces a morphism of abelian schemes
$f^*:\pic^0_{C'/S} \to \pic^0_{C/S}$, whose image $Y$ is an abelian scheme.  The principal polarization
of $\pic^0_{C/S}$ restricts to a polarization on $Y$, whose exponent we denote by $e$ (\S \ref{S:DgExpTp}).  Using the
polarization on $\pic^0_{C/S}$ we define a norm endomorphism  $N_Y:
\pic^0_{C/S} \to \operatorname{Pic}^0_{C/S}$ with image $Y$ (\S \ref{S:normendomorphism}), and define $P(C/C')$ to be the image
$\operatorname{Im}([e]_X - N_Y)$.  In this context, we say that
$P(C/C')$ is a Prym scheme of exponent $e$ (embedded in
$\pic^0_{C/S}$).  In the situation of part (2), when the induced
polarization on $P(C/C')$ is the $e^{th}$
multiple of a principal polarization $\xi$, we call the principally
polarized abelian scheme $(P(C/C'),\xi)$ a Prym--Tyurin Prym  scheme of exponent
$e$.
    In much of the
literature, as in the beginning of this introduction, the phrase
``Prym variety'' often refers to case 2(a).

We also recover an arithmetic version of a criterion of Welters; here we work over a field $K$.

\begin{teoalpha}
Let $C$ be a smooth pointed proper curve over a field $K$, which we take to be embedded via the Abel--Jacobi map in its principally polarized Jacobian $(\operatorname{Pic}^0_{C/K},\Theta_{C/K})$, let $Z\hookrightarrow \operatorname{Pic}^0_{C/K}$ be a sub-abelian variety of dimension $g_Z$ with principal polarization $\Xi$, and let $\beta$ be the canonical composition 
$$\xymatrix{
\beta: C\ar[r]& \operatorname{Pic}^0_{C/K} \ar[r]^\sim &
\widehat{\operatorname{Pic}^0_{C/K}} \ar[r] &  \widehat Z  \ar[r]^\sim&
   Z.}
$$
Then $\Theta_{C/K}|_Z=e\cdot \Xi$ if and only if there is a numerical (or equivalently,
  homological) equivalence
  \[
\beta_*[C_{\bar K}] \equiv e
    \frac{[\Xi]^{g_Z-1}}{(g_Z-1)!}
  \]
  of 1-cycles on $Z_{\bar K}$.
\end{teoalpha}

In fact, our result does not require $C$ to be pointed, and is actually
phrased in terms of a morphim $\beta:C \to T$ to a torsor under $Z$;
see Theorem \ref{T:Welters} for the complete statement.

The reader simply curious about what sorts of Prym--Tyurin and Prym
constructions are valid over arbitary fields is invited to directly
turn to the results of  \S \ref{S:prymfield}.  Here is how we work
up to these statements.
In \S \ref{S:rudiments}, we start by collecting some basic facts on
abelian schemes, their polarizations and their endomorphisms.
 
In \S \ref{S:morenotes}, we define the norm endomorphism mentioned
above, and use this to get a general theory of complements for
polarized abelian schemes.  Our treatment here follows
\cite[Ch.~5]{BL}, except that at each stage we make sure our
constructions are valid over in the relative setting.
     
In \S \ref{S:abelmap}, we collect some results on the Abel map
$\alpha_{C/S}: C \to
\pic^{(1)}_{C/S}$ associated to a smooth proper curve, and the
compatibility of $\alpha_{C/S}$ and polarizations under finite morphisms $C
\to C'$.   Since we work in a setting where $C$ need not admit a
section, $\pic^{(1)}_{C/S}$ might be a nontrivial torsor under
$\pic^{0}_{C/S}$, and so our results are necessarily expressed in
terms of maps between \emph{torsors} under Jacobian varieties, rather
than the Jacobians themselves.

In \S \ref{S:prymfield}, we follow the development in
\cite[Ch.~12]{BL} of the theory of Prym--Tyurin and Prym varieties, as
well as Welters' Criterion.  One of the tools used in
\emph{op.~cit.}~is a construction due to Matsusaka \cite{matsusaka59}
which, given cycles $\alpha$ and $\beta$ of complementary dimension on
an abelian variety $X$, produces an endomorphism
$\delta(\alpha,\beta)\in \operatorname{End}(X)$.  However, Matsusaka's
work is built upon Weil's foundations for algebraic geometry, which
now seem ill-suited to analyzing objects over a field that is not 
algebraically closed.  Consequently, we take a detour in \S
\ref{S:cyclefield} to recast this work using a modern reformulation of
Samuel's notion of a regular homomorphism (\cite{ACMVfunctorial}).  In \S \ref{S:prymscheme},
we consider the existence of Prym--Tyurin and Prym schemes relative to
an arbitrary base $S$.  Finally, some sample applications are sketched
out in \S \ref{S:applications}.

After we distributed a first draft of this paper, 
Tudor Ciurca kindly shared with us his recent work on Prym varieties and cubic threefolds over
$\integ$ \cite{ciurcaprym24}.  As part of this work Ciurca gives a proof
-- completely independent from ours -- 
that if $C \to C'$ is an \'etale cover of curves over an arbitrary
field $K$, then the induced morphism on Jacobians $\pic^0_{C/K} \to
\pic^0_{C'/K}$ is smooth; he uses this to deduce the existence of a Prym
variety $P(C/C')$ for an \'etale double cover of curves in
characteristic two.  We encourage the interested reader to consult
Ciurca's work directly.

\subsection*{Acknowledgments} We thank Rachel Pries for a helpful
discussion of wild ramification, and Patrick Brosnan for alerting us
to a serious oversight in a first draft of this paper.

 \section{Preliminaries on abelian schemes}\label{S:rudiments}
In this section we review some standard results about abelian schemes.  The main purpose is to fix notation, as well as to collect some results for which we could not find a reference.

We follow the conventions in \cite[Ch.~6]{GIT}.  In particular, we
work over a connected locally Noetherian scheme $S$.  An \emph{abelian scheme}
over $S$, or an \emph{abelian $S$-scheme}, is a smooth proper group
scheme $X \to S$ with geometrically connected fibers.  Recall that a
smooth proper group scheme $X \to S$ with connected fibers has
geometrically connected fibers \cite[Rem.~1.2(c)]{faltingschai}, and
so is an abelian scheme.  We refer to abelian schemes over fields as
abelian varieties, and recall that every abelian variety is projective;
 in fact, more generally,  if $S$ is geometrically
unibranch (e.g., normal), then every abelian scheme over $S$ is
projective over $S$ (see \cite[p.3, (c)]{faltingschai}).
    On the other hand, we recall that, for instance,  an abelian scheme over an Artinian ring need not be projective \cite[XII.4, p.189]{raynaud_faisceaux}.

A \emph{morphism of abelian schemes} over $S$ is a morphism $X \to Y$ of $S$-schemes which carries the identity section of $X$ to that of $Y$.  It follows from rigidity  \cite[Prop.~6.1]{GIT}  that such a morphism is automatically a homomorphism of group schemes  \cite[Cor.~6.4]{GIT} .
Recall that given  a morphism  $f\colon X \to Y$ of abelian schemes over $S$, the \emph{kernel} $\ker(f)$ is defined to be the pull-back  of $f$ along  the $0$-section of $Y$.    This is naturally a sub-$S$-group scheme of $X$, and by construction, its formation is compatible with base change.

For an abelian scheme $X$ over $S$ and an integer $e$, we denote by
$[e]_X:X\to X$ the multiplication by $e$ map, and by $X[e]$ its
kernel.  We denote by $\widehat X:=\operatorname{Pic}^0_{X/S}$ the
dual abelian $S$-scheme (\cite[Cor.~6.8]{GIT} when $X/S$ is
projective, and \cite[p.3]{faltingschai} in general), and for an
$S$-homomorphism $f\colon X\to Y$ of abelian $S$ schemes, we denote by
$\widehat f\colon \widehat Y\to \widehat X$ the $S$-homomorphism induced by pull-back of line bundles.  We denote by $\pi:X\to S$ the structure map to $S$,  $\mu:X\times_S X\to X$ the group law, $p_i:X\times _SX \to X$, $i=1,2$ the projections, and by $\epsilon:S\to X$ the zero section.

\subsection{Polarizations of abelian schemes}\label{S:Pol}
Given an abelian scheme $X$ over $S$ and a line bundle $L$ on $X$, one obtains a morphism
\begin{equation}\label{E:phi_L}
\xymatrix{\phi_L:X\ar[r]& \widehat X,}
\end{equation}
whose definition we recall on points.  Given 
a morphism  $a:T\to X$ of $S$-schemes, i.e., an element $a\in X(T)$, denote by $\pi_T:X_T=T\times_SX \to T$ the pull-back of $X/S$ along $T\to S$, by $L_T$ the pull-back of $L$ along $X_T\to X$, and we denote the translation by $a$ map by  $$\xymatrix{t_a: X_T=T\times_TX_T\ar[r]^-{a\times \operatorname{Id}}&  X_T\times_TX_T \ar[r]^-{\mu_T}& X_T.}$$
 Then we have 
\begin{equation}\label{E:phi_L2}
  \xymatrix@R=.5em{
\phi_L(T):X(T)\ar[r]&\widehat X(T)\\
a\ar@{|->}[r]& t_a^*L_T\otimes L_T^{-1}\otimes \pi_T^*a^*L_T^{-1}\otimes \pi_T^*\epsilon_T^*L_T.
}
\end{equation}
  By construction, the formation of $\phi_L$ is compatible with base change.
In the usual way, the morphism of schemes $\phi_L$ in \eqref{E:phi_L} is identified with $\phi_L(X)(\operatorname{Id}_X)$ in \eqref{E:phi_L2}
   and, viewing $\widehat X$ as a moduli space of line bundles on $X$, we have that $\phi_L$ is the morphism associated to the line bundle 
$$\mu^*L\otimes p_1^*L^{-1}\otimes p_2^*L^{-1},
 $$
viewing this line bundle on $X\times_S X$ as a family of line bundles on $X$ parameterized by the first copy of $X$.  Moreover, if $S$ is the spectrum of an algebraically closed field $k$, then for $a\in X(k)$, one has  $\phi_L(a)=t_a^*L\otimes L^{-1}$.   Note that, \emph{a priori}, the definition of $\phi_L(T)$ above only gives a map to  $\operatorname{Pic}_{X/S}(T)$; to check that the image is in $\operatorname{Pic}^0_{X/S}(T)=\widehat X(T)$ it suffices to check on geometric fibers, and this is standard (e.g., \cite[p.79]{mumfordAV}).  

We use the notation 
 $$\K(L) := \ker \phi_L.$$
 We have $\K(L)=X$ if and only if $L\in \widehat X(S)$, i.e., the map $S\to \operatorname{Pic}_{X/S}$ induced by $L$ has image in the connected component of the identity  $\widehat X= \operatorname{Pic}^0_{X/S}$. Again, it suffices to check this on geometric fibers, which is standard (e.g., \cite[Def.~p.74 and Diag.~p.80]{mumfordAV}).  
 At the opposite extreme, if $\pi_*L\ne 0$ then  $\K(L)$ is finite over $S$ if and only if $L$
is ample; since $\phi_L$ is stable under base change by construction, and $X/S$ is proper, it suffices to show that on the geometric  fibers $X_s$ we have that  $H^0(X_s,L_s)\ne 0$ and $\K(L_s)$ is quasi-finite if and only if $L_s$ is ample, which is a standard result (e.g., \cite[p.60, Appl.~1]{mumfordAV}).

Given an abelian scheme $X$ over $S$, a \emph{polarization} on $X$ is a homomorphism
$$
\xymatrix{
\lambda : X\ar[r]& \widehat X}
$$
such that for every geometric point $s$ of $S$, there is an ample line
bundle $L_s$ on $X_s$ such that the morphism
$\lambda_s=\phi_{L_s}:X_s\to \widehat X_s$.  This is equivalent to the
existence of an \'etale cover $S'\to S$ such that there exists a
relatively ample line bundle $L'$ on $X_{S'}=S'\times_S X $ such that
$\lambda_{S'}=\phi_{L'}:X_{S'}\to \widehat X_{S'}$
\cite[p.4]{faltingschai}. Note in particular that: (1) every polarized
abelian scheme is \'etale locally projective, (2) if $X$ is projective
and $L$ is relatively ample, then $\phi_L$ is a polarization, and (3)
every polarization on an abelian variety $X$ over an
algebraically
 closed field is of the form $\phi_L$ for an ample line bundle $L$ on
$X$.  Finally, given a polarization $\lambda$, there is a moduli space parameterizing 
bundles $L$ on $X$ for which $\phi_L=\lambda$; this admits a coarse moduli space which, if nonempty, defines a closed subscheme of $\operatorname{Pic}_{X/S}$  that is a 
torsor under $\widehat X$ \cite[p.4]{faltingschai}.

A polarization is necessarily finite and surjective \cite[p.120]{GIT},
as well as faithfully flat, being a surjective morphism of abelian
schemes \cite[Lem.~6.12]{GIT}.  Note that this implies that
$\ker \lambda$ is a finite flat commutative group scheme over $S$.  A
polarization is called principal if it is an isomorphism.  We
say that a line bundle $L$ (resp.~Cartier divisor $D$) on $X$ that is
ample over $S$ induces the polarization $\lambda$ if
$\lambda =\phi_{L}$ (resp.~$\lambda = \phi_{\mathcal O_{X}(D)}$).

\subsubsection{Degree, exponent, and type of a
  polarization} \label{S:DgExpTp} A polarized abelian scheme
$(X,\lambda)$ over $S$ has an associated \emph{degree} $d$.  Recall
that this is defined so that the degree of the finite flat morphism
$\lambda$ is $d^2$, or equivalently, that the order of the finite flat
commutative group scheme $\ker \lambda$ over $S$ is $d^2$, and that
after an \'etale cover $S'\to S$ such that $\lambda =\phi_L$ for some
ample line bundle $L$ on $X_{S'}$, one has that $\pi_{S'*}L$ is a
vector bundle of rank $d$ on $S$, and the finite flat group scheme $\K(L)$
has order $d^2$ over $S'$ \cite[p.4]{faltingschai},
\cite[p.150]{mumfordAV}.  The degree of a polarization is stable under
base change.  We denote by $\mathcal A_{g,d}$ the separated algebraic
stack of finite type over $\operatorname{Spec}\mathbb Z$ whose objects
over $S$ are polarized abelian schemes $(A,\lambda)$ of dimension $g$
and degree $d$ over $S$ (see e.g., \cite[p.486]{djAV}).

For a finite flat commutative group scheme $G$ of order $n$ over $S$,
one has that $G$ is contained in $G[n]$, the kernel of the
multiplication by $n$ map on $G$ \cite[Thm.~(Deligne), p.4]{TateOort}.
One defines the \emph{exponent} of $G$ to be the smallest positive
integer $e$ such that $G \subseteq G[e]$.  The exponent of a polarized
abelian scheme $(X,\lambda)$ is defined to be the exponent of the
finite flat commutative group scheme $\ker (\lambda)$.  Since the
multiplication map is proper, and therefore closed, by considering
generic fibers it is easy to see that the exponent is lower
semi-continuous on fibers, and stable under faithfully flat base
change.

In order to describe the behavior of the exponent under base change in
more detail, it is helpful to discuss the type of a polarization.  We
refer to \cite{djAV} for references.  First recall that for a
polarized abelian scheme $(A,\lambda)$ of dimension $g$ and degree $d$
over $S$ there is a nondegenerate alternating pairing (e.g., \cite[p.5]{faltingschai}, \cite[p.487]{djAV}):
 $$
\xymatrix{e^\lambda :\ker (\lambda)\times_S\ker (\lambda)\ar[r]& \mathbb G_{m,S}.}
$$
For every sequence of positive integers $\delta = (\delta_1,\dots,\delta_g)$ with $\delta_1\mid \delta_2\cdots \mid \delta_g$ and $\prod_{i=1}^g\delta_i=d$, define a scheme over $\operatorname{Spec}\mathbb Z$:
$$
\K(\delta):=\prod_{i=1}^g \mu_{\delta_i,\mathbb Z}\times_{\mathbb Z}(\mathbb Z/\delta_i\mathbb Z)
$$
and the nondegenerate alternating pairing
$$
e_\delta:\K(\delta)\times_{\mathbb Z}\K(\delta)\longrightarrow \mathbb G_{m,\mathbb Z}
$$
$$
((\zeta_i,n_i)_{i=1,\dots,g}, (\eta_i,m_i)_{i=1,\dots,g})\mapsto \prod_{i=1}^g\zeta_i^{m_i}\cdot \eta_i^{-n_i}.
$$
We denote by $\mathcal A_{g,\delta}$ the separated algebraic stack of finite type over $\operatorname{Spec}\mathbb Z$ whose objects over $S$ are polarized abelian schemes $(A,\lambda)$ of dimension $g$ and degree $d$ over $S$ such that there exists a faithfully flat base change $S'\to S$ such that $(\ker (\lambda)_{S'},e^\lambda_{S'})\cong (\K(\delta)_{S'},e_{\delta,S'})$; it is equivalent to require simply that $\ker (\lambda)_{S'}\cong \K(\delta)_{S'}$ (\cite[Rem.~1.4(2)]{djAV}).  We say that such a polarization $\lambda$ is of \emph{type} $\delta$.  For a polarization of type $\delta=(\delta_1,\dots,\delta_g)$, the exponent is $\delta_g$.  

The stacks $\mathcal A_{g,\delta}$ are locally closed substacks of
$\mathcal A_{g,d}$ \cite[Prop.~1.5]{djAV}, and the irreducible
components of $\mathcal A_{g,d}$ are the closures
$(\mathcal A_{g,\delta})^c$ of the $\mathcal A_{g,\delta}$ in
$\mathcal A_{g,d}$ \cite[Thm.~1.12(1)]{djAV}.  If $k$ is an
algebraically closed field of characteristic $p=\operatorname{char}(k)$ and $p\nmid d$, then the $\mathcal A_{g,\delta,k}$ are
closed in $\mathcal A_{g,d,k}$, and
$\mathcal A_{g,d,k}$ is the disjoint union of the
$\mathcal A_{g,\delta,k}$; if
$\operatorname{char}(k)=0$, this is classical, and for
$\operatorname{char}(k)=p>0$, this is \cite[Thm.~1.12(2),
Rem.~1.13(2)]{djAV}.  If $p\mid d$, then one sees from the definition of $\K(\delta)$
that each $\mathcal A_{g,\delta,k}$ avoids the $p$-rank zero locus of 
$\mathcal A_{g,d,k}$, and thus the union of the various
$\mathcal A_{g,d,k}$ cannot be all of $\mathcal A_{g,d,k}$.  (See,
e.g., \Cref{R:defprank} for a reminder on the notion of $p$-rank.)

Moreover, if $p\mid d$, but $p^2\nmid d$, then the
$(\mathcal A_{g,\delta,k})^c$ are still disjoint, but if $p^2\mid d$,
then the $(\mathcal A_{g,\delta,k})^c$ are not disjoint
\cite[Rem.~1.13(2)]{djAV}.  In particular, there are examples due to
Norman (\cite{norman78}; see also \cite[Exa.~3.5.1]{djAV}, and Remark
\ref{R:normanEx} below) showing that
$(\mathcal A_{2,(p,p),\bar\ff_p})^c$
and
$(\mathcal A_{2,(1,p^2),\bar\ff_p})^c$
have a nontrivial intersection.

In summary, in terms of exponents, this implies that in characteristic
$0$, or in positive characteristic $p$ if $p\nmid d$, the exponent is
stable under base change. (One can also see this directly, since then
the kernel of the polarization is an \'etale group scheme over $S$.)
However, in positive characteristic, if $p^2\mid d$, then there are examples where the exponent is \emph{not} stable under base change. 
The following lemma shows, however, that in the situation we frequently will be in, namely the case where our polarization is a multiple of a principal polarization, the exponent is in fact stable under base change:

\begin{lem}[Exponent under base change]\label{L:ExpBC}
Let $(X,\lambda)$ be a polarized abelian scheme of dimension $g$ and degree $d$ over $S$.  The exponent of $\lambda$ is stable under base change if one of the following conditions hold:
\begin{alphabetize}
\item The associated morphism $S\to \mathcal A_{g,d}$ factors through
  the natural inclusion $\mathcal A_{g,\delta}\subseteq \mathcal
  A_{g,d}$ for some polarization type $\delta$; e.g., $S$ is a scheme over a field $K$ of characteristic $p$ and $p\nmid d$.
\item We have $\lambda=e\xi$ for a positive integer $e$ and a principal polarization $\xi$ on $X$.
\end{alphabetize}
\end{lem}

\begin{proof}
  Case (a) is clear as the exponent of $\K(\delta)$ is $\delta_g$, and
  thus the exponent of $(X,\lambda)$ is $\delta_g$ after any base
  change.  For case (b) we have $\ker \lambda = X[e]$.  Clearly the
  exponent is at most $e$, and since $\deg [n]_X=n^{2g}$ (e.g.,
  \cite[p.65]{mumfordAV}), the exponent of $X[e]$ is not less than
  $e$.
\end{proof}

\begin{rem}\label{R:normanEx}
Condition (b) does not imply condition (a). 
Indeed, Norman's example \cite{norman78} showing that $(\mathcal
A_{2,(p,p),\ff_p})^c$ and $(\mathcal A_{2,(1,p^2),\ff_p})^c$ are not
disjoint is obtained in the following way.  Let $(E,\lambda)/k$ be a
(principally polarized)
supersingular elliptic curve over an algebraically closed field of
characteristic $p>0$.  Then $(E\times_kE, p\cdot(\lambda\times
\lambda))$ deforms into both $\mathcal A_{2,(p,p),k}$ and $\mathcal A_{2,(1,p^2),k}$.
\end{rem}

\subsection{Short exact sequences of abelian schemes}

The following seems to be well-known, but we are not aware of a reference:

\begin{lem}\label{L:dualexact}
  If $0 \to X \stackrel{\iota} \to Y \to Z \to 0$ is a short exact sequence
  of abelian schemes over $S$, then so is $$\xymatrix{0 \ar[r]&
    \widehat Z \ar[r]&  \widehat Y \ar[r]^{\widehat\iota} &  \widehat X \ar[r]& 0}.$$
\end{lem}

\begin{proof}
It suffices to verify that the sequence is exact on fibers over
geometric points, so we assume that $S$
is the spectrum of an algebraically closed field.  While this case appears to be well-known, for lack of a better reference we include a proof,  following Conrad on 
\href{https://mathoverflow.net/questions/37536/quotient-of-abelian-variety-by-an-abelian-subvariety}{{\tt MathOverflow (37536)}}:
  Let $N$ be any positive integer. 
  Since the multiplication by $N$ map $[N]$ is
surjective, by the snake
lemma we have an exact sequence of group schemes $0 \to X[N] \stackrel{\iota[N]}\to Y[N]
\to Z[N] \to 0$.  Taking Cartier duals, and
recalling that $\operatorname{Ext}^1(G,\mathbb G_m) = 0$ for a finite commutative
group scheme $G$ over an algebraically closed field, we obtain an exact sequence  of finite group schemes $0 \to Z[N]^\vee \to Y[N]^\vee
\to X[N]^\vee \to 0$.  Using the canonical identifications $\widehat X[N]
\iso X[N]^\vee$, etc.~(e.g., \cite[Cor.~1.3(ii)]{odaderham}), we find that  we have an exact sequence $0 \to \widehat Z[N] \to \widehat
Y[N] \stackrel{\hat \iota [N]}{\to} \widehat X[N] \to 0$.  In particular, for each $N$ we have
$\ker(\widehat \iota)[N] = \widehat Z[N]$, and thus the canonical morphism  $\widehat Z\to \ker \widehat \iota$ obtained via pull-back of line bundles is an isomorphism, i.e., $
\widehat Z= \ker\widehat \iota$, as the union of the $N$-torsion  in a sub-group scheme of an abelian variety is dense in the sub-group scheme.  A similar argument using the surjection $Y[N] \stackrel{\hat \iota [N]}{\twoheadrightarrow} \widehat X[N]$ implies that the image of $\widehat \iota$ is equal to $X$.  
\end{proof}

One application of this is to provide a condition for a morphism of abelian schemes to be a closed embedding:

\begin{cor}
  \label{C:hatf}
  Let $f:X \to Y$ be a morphism of abelian schemes over $S$.  Then the following are equivalent:
  \begin{alphabetize}
  \item $f$ is a closed embedding;
  \item $\widehat f$ is surjective, and $\ker(\widehat f)$ is an abelian scheme. \label{I:C:hatf-b}
  \end{alphabetize} 
\end{cor}

\begin{proof}
  If $f$ is a closed embedding, then we obtain a short exact sequence of abelian schemes over $S$, $0\to X\stackrel{f}{\to} Y\to X/Y\to 0$.  Dualizing this short exact sequence,   Lemma \ref{L:dualexact} implies that $\widehat f$ is surjective and $\ker(\widehat f) \iso \widehat {X/Y}$ is an abelian scheme over $S$.  Conversely, if (b) holds, then we have an exact sequence of abelian varieties $0 \to \ker(\widehat f) \to \widehat Y \stackrel{\widehat f}{\to}\widehat X \to 0$; now invoke Lemma \ref{L:dualexact} again.
\end{proof}

In \S \ref{SubS:KerAbSch}, we review some criteria via Nori fundamental groups for checking the condition in \Cref{C:hatf}\ref{I:C:hatf-b}, that the kernel of a morphism of abelian schemes be an abelian scheme.

\subsection{Isogenies and endomorphisms of abelian schemes}\label{S:isog}
Given abelian schemes $X/S$ and $Y/S$, an \emph{isogeny} $f\colon X\to Y$ is a finite surjective homomorphism of abelian $S$-schemes.  The kernel $\ker (f)$ is a finite flat commutative $S$-group scheme, being the kernel of a surjective morphism of abelian
schemes \cite[Lem.~6.12]{GIT},  and therefore has an \emph{exponent}, $e=e(f)$, i.e., 
the smallest integer $e$ such that $\ker(f) \subseteq \ker([e]_X)$ (see \S \ref{S:DgExpTp}).

\begin{rem}[Exponent of isogenies and base change]
Since the multiplication map is proper, and therefore closed, by considering generic fibers it is easy to see that the exponent of an isogeny is lower semi-continuous on fibers, and stable under faithfully flat base change. 
However, the exponent of an isogeny need not be stable under base change, even if the isogeny is a polarization (see \S \ref{S:DgExpTp}).  Nevertheless, our primary interest will be exponents of polarizations, and there are cases of interest where the exponent is stable under base change (see 
\Cref{L:ExpBC}). 
\end{rem}

A basic tool we will use is:

\begin{pro}\label{T:factor}
Let $X$, $Y$, and $Y'$ be abelian $S$-group schemes,
 and suppose that we are given $S$-homomorphisms $\phi\colon X\to Y$ and $\phi'\colon X\to Y'$. If $\phi$ is surjective, then there is an $S$-homomorphism $\psi\colon Y\to Y'$ making the following diagram commute
$$
\xymatrix{
X\ar[r]^{\phi'} \ar@{->>}[d]_{\phi}& Y'\\
Y \ar@{-->}[ru]_\psi& 
}
$$
if and only if
$
\ker \phi\subseteq \ker \phi'
$. 
Moreover, in this situation, $\psi$ is unique.  
\end{pro}

\begin{proof}
 Since $\phi$ is surjective it is flat \cite[Lem.~6.12]{GIT}, and so $\ker(\phi)$ is flat over $S$.  
 The proposition then  follows from \cite[Exp.~V, Cor.~10.1.3]{sga3-1}, observing that $Y=X/\ker (\phi)$ and represents the fppf quotient sheaf. 
  \end{proof}

\begin{cor}[{\cite[Prop.~1.2.6]{BL}}]\label{P:BL1.2.6}
    
 Let $f\colon X\to Y$ be an isogeny of abelian schemes over $S$.  If $\ker(f)\subseteq X[N]$ for some positive integer $N$,  then there is an isogeny $g_N\colon Y\to X$, unique up to isomorphisms,  such that 
$g_Nf=[N]_{X}$ and $fg_N=[N]_{Y}$.  The formation of $g_N$ is
compatible with base change.
\end{cor}

\begin{proof}    For existence, one applies \Cref{T:factor} to  $f$ and $[N]_X$ to see that  $[N]_X$ factors through $f$ to give $g_N$ such that $g_Nf=[N]_X$. 
\Cref{T:factor} also gives the uniqueness statement for $g_N$.
That the formation of $g_N$ is compatible with base change follows from the uniqueness of $g_N$.

To show that $fg_N=[N]_Y$, one argues as follows.  One can check from the construction that $\ker g_N\subseteq Y[N]$, and so by the same argument, $[N]_Y$ factors through $g_N$ to give $f_N$ such that $f_Ng_N=[N]_Y$.  Then we have $f_N[N]_X= f_Ng_Nf= [N]_Yf=f[N]_X$.  Since $[N]_X$ is surjective, we can conclude that $f_N=f$, completing the proof.  
 \end{proof}

\begin{rem}\label{R:gPrim} If the exponent of the isogeny $f\colon X\to Y$ is $e(f)$, then 
the homomorphism $g_{e(f)}$ is primitive. (A morphism $a\colon X \to Y$ is called
primitive if there is no morphism $b\colon X \to Y$ and integer $n\ge 2$
such that $a = nb$.)  Indeed, if there were an integer $n\ge 2$ such that $h=\frac{g_{e(f)}}{n}\in \operatorname{Hom}(Y,X)$, then we would have $hf= [e(f)/n]_X$, and so $\ker f\subseteq \ker [e(f)/n]_X$, contradicting the definition of the exponent of $f$.  Recall that while the exponent $e(f)$ of $f$ is stable under \emph{faithfully flat} base change, the exponent $e(f)$  is not necessarily stable under \emph{arbitrary} base change (\S \ref{S:DgExpTp}).  In particular, if $S'\to S $ is a morphism of schemes and $f_{S'}:X_{S'}\to Y_{S'}$ is the morphism obtained by base change, it is possible that $(g_{e(f)})_{S'}$ may not equal $g_{e(f_{S'})}$.  
 \end{rem}

\begin{rem}
The exponent of $g_e$ must divide the exponent of $f$ by construction.  But considering the case $f=[n]_X\colon X\to X$ for any natural number $n>1$, which obviously has exponent $e=n$ and $g_e=\operatorname{Id}_X$,  one sees that the exponent of $g_e$ (equal to $1$ in this case) may be strictly smaller than the exponent of $f$.
\end{rem}

 Recall that for abelian $S$-schemes $X$ and $Y$, we have that $\hom(X,Y)$ ($:=\operatorname{Hom}_S(X,Y)$)  is a
free $\integ$-module of finite rank; indeed, by rigidity \cite[Ch.~6, \S 1, Prop.~6.1, p.115]{GIT}, we have an injection $\hom(X,Y)\hookrightarrow \hom(X_s,Y_s)$ for any geometric point $s$ in $S$, and then one concludes using the case of abelian varieties 
(e.g., \cite[\S 19, Thm.~3, p.176]{mumfordAV}). 
  We set $\hom_\rat(X,Y) = \hom(X,Y)\otimes_\integ\rat$.

\begin{cor}[{\cite[Cor.~1.2.7]{BL}}]\label{C:BL1.2.7} 
In the notation of  \Cref{P:BL1.2.6}:
\begin{alphabetize}
\item Isogenies define an equivalence relation for 
 abelian $S$-schemes.

\item  \label{C:BL1.2.7b}  An element of $\operatorname{End}(X/S)$ is an isogeny if and only if it is invertible in $\operatorname{End}_{\mathbb Q}(X/S)$. 
\end{alphabetize}
\end{cor}

\begin{proof}
This is clear from what is above.
\end{proof}

In light of \Cref{C:BL1.2.7}\ref{C:BL1.2.7b}, 
for an isogeny $f\colon X\to Y$ of abelian schemes over $S$ with $\ker (f)\subseteq X[N]$, if $g_N\colon Y\to X$ is the isogeny from \Cref{P:BL1.2.6}, we use the notation
\begin{equation}
f^{-1}:= \frac{g_N}{N}\in \operatorname{Hom}_{\mathbb Q}(Y,X);
\end{equation}
i.e., $f^{-1}$ is the inverse in $\operatorname{Hom}_{\mathbb Q}(Y,X)$ of $f$.   Note that this definition of $f^{-1}$ is independent of the choice of $N$ such that $\ker (f)\subseteq X[N]$.  Indeed, if $\ker (f)\subseteq X[M]$, then $(Mg_N)f =[NM]_X= (Ng_M)f$ and $f(Mg_N)=[NM]_Y=f(Ng_M)$, implying by uniqueness of $g_{NM}$ that  $Mg_N=g_{NM}= Ng_M$, so that dividing by $NM$ we have  $\frac{g_N}{N}=\frac{g_{NM}}{NM}= \frac{g_M}{M}$.  
Since the formation of $g_N$ is stable under base change, so is the formation of $f^{-1}$.

  Every $h\in \operatorname{End}_{\mathbb Q}(X/S)$ can be written as
 $h=rf$ for some $r\in \mathbb Q$ and some $f\in
 \operatorname{End}(X/S)$, and since $\widehat h:=r\widehat f \in \operatorname{End}_{\mathbb Q}(X/S)$ does
 not depend on the choice of $r$ and $f$ such that $h=rf$, the assignment $h\mapsto \widehat h$ defines an involution
 on $\operatorname{End}_{\mathbb Q}(X/S)$.   The involution is stable under base change.  
   
From this one obtains the Rosati involution for $h\in \operatorname{End}_{\mathbb Q}(X/S)$ associated to a polarization $\lambda$:
$$
h^{(\dagger)}:=\lambda^{-1} \widehat h \lambda \in \operatorname{End}_{\mathbb Q}(X/S).
$$
If $\lambda$ is principal, then $\operatorname{End}(X/S)$ is stable
under the Rosati involution; and in all cases, formation of 
the Rosati involution is compatible with base change.

\subsection{Multiples of principal polarizations}\label{S:eXi-X[e]}
One of our main focuses is studying abelian schemes together with polarizations that are multiples of a principle polarization.   We will frequently use the following elementary lemma, whose proof we include for lack of a reference:

\begin{lem}\label{L:eXi-X[e]}
Let $(X,\lambda)$ be a polarized abelian scheme over $S$ and let $e$ be a positive integer.  The following are equivalent:
\begin{alphabetize}
\item \label{L:eXi-X[e]Gl}  $\lambda$ is equal to $e$ times a principal polarization $\xi$ on $X$.
\item \label{L:eXi-X[e]Fib} For every point $s$ of $S$ we have $\lambda_s$ is $e$ times a principal polarization on $X_s$.

\item \label{L:eXi-X[e]GFib} For every  geometric point $\bar s$ of  $S$ we have $\lambda_{\bar s}$ is $e$ times a principal polarization on $X_{\bar s}$.

\item \label{L:eXi-X[e]Ker} $\ker \lambda = X[e]$.
\end{alphabetize}
\end{lem}

\begin{proof}
The implications \ref{L:eXi-X[e]Gl} $\implies$ \ref{L:eXi-X[e]Fib} $\implies$ \ref{L:eXi-X[e]GFib} are immediate (as is \ref{L:eXi-X[e]Gl} $\implies$ \ref{L:eXi-X[e]Ker}).   

Next let us show \ref{L:eXi-X[e]GFib} $\implies$ \ref{L:eXi-X[e]Ker}.
First, we claim that $e\ker \lambda =0$.  We just need to check that
the zero section $\epsilon\colon S\to e\ker \lambda$ and the projection
$\pi|_{e\ker \lambda}\colon  e\ker \lambda\to S$ are inverses to one
another.  In fact, it suffices to show that $\pi|_{e\ker \lambda}$ is
a closed immersion.  This can be checked on every fiber
(\cite[17.2.6]{EGAIV4} gives the equivalence of monomorphisms with
maps with fibers that are empty or isomorphisms, and
\cite[18.12.6]{EGAIV4} gives that proper monomorphisms are closed
immersions).   Since being a closed immersion satisfies faithfully flat descent (e.g., \cite[Prop.~14.53]{GW}), it in fact suffices to check after base change to the geometric fiber.
   Therefore we have that $e\ker \phi =0$, and we can conclude that  $\ker \lambda \subseteq X[e]$.   The degree of $\ker \lambda$ over $S$ can be checked on fibers, and therefore, since both $\ker \lambda$ and $X[e]$ are finite flat groups schemes of the same degree, the inclusion is an isomorphism.

We now show \ref{L:eXi-X[e]Ker} $\implies$ \ref{L:eXi-X[e]Gl}.  Assuming \ref{L:eXi-X[e]Ker}, then \Cref{T:factor} implies that $\lambda$ factors as $\lambda = \psi\circ  [e]$ for a unique  morphism  of abelian $S$-schemes $\psi\colon X\to \widehat X$.  We just need to show that $\psi$ is a principal polarization, and  for degree reasons, it suffices to show that $\psi$ is a polarization. From the definition, we must show that for each $s$ in $S$ there is an ample line bundle $M_{\bar s}$ on the geometric fiber $X_{\bar s}$ such that $\psi_{\bar s}=\phi_{M_{\bar s}}$.  
 By definition, there exists an ample line bundle $L_{\bar s}$ on $X_{\bar s}$ such that $\lambda_{\bar s}= \phi_{L_{\bar s}}$.  From \cite[\S 23, Thm.~3, p.231]{mumfordAV},  the containment  $\K(L_{\bar s})\supseteq X_{\bar s}[e]$ implies that there exists a line bundle $M_{\bar s}$ such that $M_{\bar s}^{\otimes e}\cong L_{\bar s}$.  We then have that $\lambda_{\bar s} = \phi_{L_{\bar s}}= \phi_{M^{\otimes e}_{\bar s}}=  \phi_{M_{\bar s}}\circ [e]$.
 Since $\psi$ is unique, and stable under base change, we have that $\psi_{\bar s}=\phi_{M_{\bar s}}$, and we are done. 
\end{proof}

\begin{rem}\label{R:CEx-1FbSuf} 
Given a polarized abelian scheme $(X,\lambda)$, for degree reasons one may conclude that if for one geometric point $\bar s$ of $S$ the fiber $\lambda_{\bar s}$ is a principal polarization, then $\lambda$ is a principal polarization. 
In characteristic $0$,  if for one geometric fiber $\lambda_{\bar s}$ is $e$ times a  principal polarization, then $\lambda$ is $e$ times a principal polarization; this follows, for instance,  from the consideration of the type of the polarization (see \S \ref{S:DgExpTp}).
 However, if we do not restrict to characteristic $0$, Norman's example (see
\Cref{R:normanEx}) shows that it is possible  for one fiber
$\lambda_s$ to be  $e$ times a  principal polarization ($e \ge 2$),
while  $\lambda$ is \emph{not} $e$ times a principal
polarization.
\end{rem}

\subsection{Kernels, abelian schemes, and the geometric Nori fundamental group}\label{SubS:KerAbSch}

In order to apply \Cref{C:hatf}\ref{I:C:hatf-b}, it will be convenient  to have a criterion for when the kernel of a surjective morphism
of abelian schemes $f:X\to Y$ is actually an abelian scheme.  
As motivation, consider the case of complex abelian varieties.  In this case having the kernel be an abelian scheme is equivalent to the kernel being  connected; said in another way, the question is whether $f$ factors through any isogeny over $Y$.  
For complex abelian varieties this question can 
can be understood purely
in terms of its effect on the period lattices of the abelian
varieties, or, equivalently, on their fundamental groups via the
theory of covering spaces; i.e., the kernel of $f$ is an abelian
scheme if and only if  $f_*(\pi_1(X,0_X))=\pi_1(Y,0_Y)$.
 
  In positive
characteristic, one can study the effect of a morphism of abelian
varieties on their \'etale fundamental groups, but this misses
phenomena related to inseparability; e.g., over an algebraically closed field of characteristic $p>0$, one may have $\ker f$ being connected but non-reduced, so that the kernel would fail to be an abelian scheme, but the fundamental group would not detect this.  Working instead with Nori's
fundamental group scheme allows one to recover the connection between fundamental groups and kernels.
Recall that for an abelian variety $X/K$ over a field $K$, its geometric \'etale fundamental group may be be
computed as $\pi_1^\et(X_{\bar K},0) = \invlim N X[N](\bar K)$, while its geometric Nori fundamental group \emph{scheme} is
$\pi_1^\nori(X_{\bar K}) = \invlim N X[N]$ \cite{nori83}.  In characteristic zero,
these notions coincide; in positive characteristic, the geometric Nori fundamental
group scheme is a strictly richer object.  Note that if $f\colon Y \to
X$ is an isogeny of abelian varieties, then $Y$ is a torsor over $X$
under the finite group scheme $\ker(f)$; and that $f$ is separable if
and only if $\ker(f)$ is \'etale.

  More generally,  recall that if $(X,P)/K$ is a proper connected reduced pointed scheme, then
the Nori fundamental group scheme $\pi_1^\nori(X,P)$ is a profinite
group scheme over $K$ which classifies torsors over $X$ (pointed over
$P$) under finite group schemes; such a torsor $(\widetilde
X,\widetilde P)$ under a group scheme $G$ corresponds to 
homomorphisms $\pi_1^\nori(\widetilde X, \widetilde P) \hookrightarrow
    \pi_1^\nori(X,P) \twoheadrightarrow G$ \cite[Thm.~4]{garuti09} \cite{nori83}.
   Similarly, there is a profinite group
scheme $\pi_1^\nori(X,P)^\abelian$ which classifies torsors under
finite \emph{commutative} group schemes \cite[\S
3]{antei11}.
 There is a surjection
$\pi_1^\nori(X,P) \twoheadrightarrow \pi_1^\nori(X,P)^\abelian$, and a cover $(\widetilde X,\widetilde P)\to (X,P)$
is abelian, i.e., is a torsor under a finite commutative group
scheme $G$,  if and only if the corresponding classifying map $\pi_1^\nori(X,P) \twoheadrightarrow G$
factors through $\pi_1^\nori(X,P)^\abelian$.  In a given setting (such
as for abelian varieties, or pointed curves), if the choice of
basepoint is clear from context, we will often write $\pi_1^\nori(X)$
for $\pi_1^\nori(X,P)$.

Recall
 that if $f\colon X \to Y$ is a morphism of abelian varieties over a
field, then $\ker(f)$ admits a maximal subabelian variety
$\ker(f)^\ab$ (\Cref{L:maxab}).

\begin{lem}
\label{L:nori}
Let $f:X \to Y$ be a surjective morphism of abelian schemes over
$S$.  Then $\ker(f)$ is an abelian scheme over $S$ if and only if for each geometric point $s$ of $S$, the induced
map $$\pi_1(f_s)\colon \pi_1^\nori(X_s) \to \pi_1^\nori(Y_s)$$ is
surjective. 
\end{lem}
 
\begin{proof}
It suffices to prove this for the geometric fibers, so we may and do assume that $S=\operatorname{Spec} k$ for an algebraically closed field $k$.  

If $\ker(f)$ is an abelian variety, then we have a short exact
sequence  of abelian varieties $$\xymatrix{0\ar[r] &\ker (f)\ar[r]&
  X\ar[r]^f & Y\ar[r]& 0,}$$ and in the proof of Lemma \ref{L:dualexact}, we saw that this induces a short exact sequence of $k$-group schemes  
$$\xymatrix{0\ar[r]& \ker (f)[N]\ar[r]& X[N]\ar[r]^{f[N]}& Y[N]\ar[r]& 0}$$ for each $N\ge 2$.  Since $N$-torsion in an abelian variety forms a surjective system, we see that in the limit, the exact sequence stays exact, and we obtain a surjection $\pi_1(f)\colon \pi_1^\nori(X) \to \pi_1^\nori(Y)$.

Conversely, suppose that $\ker(f)$ is not an abelian variety.   We want to show that  $\pi_1(f)$ is not surjective.  
Initially, suppose $f$ is an isogeny.    Then we may assume that $\ker(f)$ is nontrivial, and suppose it has exponent $N>1$.  From the exact
sequence of group schemes 
\[\xymatrix{
0 \ar[r]&\ker (f)  \ar[r] & X \ar[r] & Y\ar[r] & 0
}
\]
and the snake lemma applied to the multiplication-by-$N$ maps, we
deduce that there is an exact sequence of group schemes 
\begin{equation}
  \label{E:gpschemeN}
\xymatrix{
0 \ar[r] &   \ker (f) \ar[r] &X[N]  \ar[r]& Y[N] \ar[r] & \ker(f)/[N]\ker(f) \iso \ker(f) \ar[r]
&  0,
}
\end{equation}
where we are using that the multiplication by $N$ maps are surjective for abelian varieties, and that we have chosen $N$ so that $\ker(f)[N] = \ker (f)$.  (It is important here that we are using the group schemes, and not just the $k$-points, since otherwise the argument would fail if $f$ were a purely inseparable isogeny.)  
Therefore, $\pi_1(f)\colon\pi_1^\nori(X) \to \pi_1^\nori(Y)$ is not
surjective.

For the general case, let $Z = \ker(f)^\abvar \subsetneq \ker(f)$ be the maximal abelian
subvariety (\cref{L:maxab}). By factoring $f$ as 
$X \to X/Z \to Y$, we find that $X/Z\to Y$ is a necessarily non-trivial isogeny, and therefore, from the previous paragraph, we have that 
$\pi_1^\nori(X/Z) \to \pi_1^\nori(Y)$ is not surjective, so that 
$f_*\colon\pi_1^\nori(X) \to \pi_1^\nori(X/Z) \to \pi_1^\nori(Y)$ fails to be 
surjective, as well.
\end{proof}

More generally we have the following result, which is well-documented
in the case where $f$ is an isogeny.

\begin{lem}
   \label{L:kerhatf}
  Let $f\colon X \to Y$ be a surjective morphism of abelian schemes over
  $S$.  Then for each geometric point  $s$ of $S$, we have  canonical isomorphisms
     $$\coker(\pi_1^{\operatorname{Nori}}(f_s)) \iso \ker(f_s)/\ker(f_s)^\abvar \ \ \text{ and } \ \ \ker(\widehat
  f_s) \iso (\ker(f_s)/\ker(f_s)^\abvar)^\vee \iso \coker(\pi_1^{\operatorname{Nori}}(f_s))^\vee.$$
\end{lem}

\begin{proof}
  It clearly suffices to prove this when $S = \spec(k)$ is the
  spectrum of an algebraically closed field.
     Let $\bar X =
X/\ker(f)^\abvar$, and factor $f$ as a composition of surjections $f:X \stackrel{\phi}{\to}
\bar X \stackrel{\bar f}{\to} Y$. 

First observe that from \Cref{L:nori}, since by construction the kernel of $\phi$ is an abelian scheme,  $\pi_1^\nori(\phi)$ is surjective, so that  $\coker(\pi_1^{\operatorname{Nori}}(f))= \coker(\pi_1^{\operatorname{Nori}}(\bar f))$. 
Let $N$ be any integer which annihilates the finite flat
group scheme $\ker(\bar f) = \ker(f)^{\finflat}$; i.e., $\ker(\bar f) \subseteq \bar X[N]$.  As in \eqref{E:gpschemeN}, we obtain an
exact sequence of finite flat group schemes
\begin{equation}
  \label{E:barf}
  \xymatrix{0 \ar[r] & \ker(\bar f) \ar[r] & \bar X[N] \ar[r]^{\bar f[N]} & Y[N]
    \ar[r] & \ker(\bar f) \ar[r] & 0.
  }
\end{equation}
 so that for sufficiently
divisible $N$, we have that $\coker(\pi_1^\nori (\bar f)) = \coker(\bar X[N] \to Y[N]) \iso
\ker(\bar f)$.  To prove the first claim of the lemma, it now suffices to recall that $\ker(\bar f) = \ker(f)
/ \ker(f)^\abvar$.

Now, to prove the second claim of the lemma,  take Cartier duals in \eqref{E:barf} and recall
\cite[Cor.~1.3(ii)]{odaderham} that this yields an exact sequence
\[
  \xymatrix{0 \ar[r] & \ker(\bar f)^\vee \ar[r] & \widehat Y[N]
    \ar[r]^{\widehat{\bar f}[N]} & \widehat{\bar X}[N]
    \ar[r] & \ker(\bar f)^\vee \ar[r] & 0.
  }
\]
Since this is true for all $N$ divisible by the exponent of $\ker(\bar
f)$, we find that $\ker(\widehat{\bar f})$ is canonically isomorphic
to $\ker(\bar f)^\vee$.  Because $\widehat f$ factors as $\widehat f: \widehat Y
\stackrel{\widehat{\bar f}}{\to} \widehat{\bar X} \stackrel{\widehat \phi}{\hookrightarrow} \widehat
X$ (with the inclusion coming by virtue of \Cref{L:dualexact}), we conclude that $\ker(\widehat f) = \ker (\widehat{\bar f}) \iso \ker(\bar f)^\vee$.
\end{proof}

\subsection{Images and preimages of abelian schemes}
\label{S:imab}

Consider a morphism $f:X \to Y$ of abelian schemes over $S$.  We will
tentatively say that $f$ is \emph{epi-abelian} if the schematic image
of $f$ is again an abelian scheme.  Certain arguments we develop here
are more transparent under the following hypothesis on a scheme $S$.
\begin{condition}
  \label{alwaysepi}
  Any morphism of $S$-abelian schemes is epi-abelian.
\end{condition}

On one hand, we have:
\begin{lem}[{\cite[Thm.~A]{ACMimab}}]\label{L:ACMimab}
  Suppose that $S$ is either:
  \begin{alphabetize}
  \item the spectrum of a field, or
    \item the spectrum of an unramified DVR of mixed characteristic,
      or
    \item of characteristic zero.
    \end{alphabetize}
Then $S$ satisfies \Cref{alwaysepi}.    
\end{lem}
On the other hand we note that, in contrast, both  $\spec \ff_p
\powser T$ and $\spec \integ_{(p)}[\zeta_p]$ \emph{fail} to satisfy
\Cref{alwaysepi} \cite[\S 4]{ACMimab}.

Fortunately, it turns out that for the construction of Prym schemes,
we can often circumvent such charming pathologies.
   
Let $X/S$ be a group scheme over $S$.  A maximal abelian subscheme
of $X$ is an abelian subscheme $X^\ab$, maximal among all abelian
subschemes of $X$, whose construction is compatible with base change $T
\to S$.

We will use the following criterion to assert the existence of a
maximal abelian subscheme in the context of constructing complementary
abelian subschemes of a separably polarized abelian scheme.

\begin{lem}
  \label{L:maxab}
  Let $\lambda:X \to Y$ be a surjection of abelian schemes over
  $S$, and let $Z \hookrightarrow Y$ be an abelian subscheme.  Let
  $\tilde Z = \lambda\inv(Z) = Z \times_Y X$.  If $\lambda$ is smooth
  or if $S$ is the spectrum of a field, then $\tilde{Z}$ admits
  a maximal abelian subscheme, $\tilde Z ^\ab$.
\end{lem}

\begin{proof}
Call a morphism $f:W \to S$ cohomologically flat if it is
cohomologically flat in dimension zero, i.e., if formation of
$f_*\mathcal O_X$ is compatible with base change.  If $f$ is proper,
flat and smoooth, then it is cohomologically flat as well
\cite[Prop.~7.8.6]{egaIII2}.  If $W$ is a proper,
flat finitely presented $S$-group scheme, then $W$ admits a maximal
abelian subscheme if and only if it is cohomologically flat
\cite[Prop.~2.16]{brochard21}.

With these preliminaries dispatched, the proof goes quickly.  The case
where $S$ is the spectrum of a field is immediate -- indeed, any 
scheme over a field is cohomologically flat -- so suppose 
$\lambda$ is smooth.  
It is also proper and flat, thus cohomologically flat.
Pulling back $Z$ by $\lambda$ shows that $\tilde Z \to Z$ is also
proper, flat and smooth (and cohomologically flat).  Since $Z \to S$
is also proper, flat and smooth, the structural morphism $\tilde Z \to
Z \to S$ is proper, flat and smooth, thus cohomologically flat.
In particular, $\tilde Z$ admits a maximal abelian subscheme.
\end{proof}

\begin{rem}\label{R:LAS}
If $X$ is an abelian variety, and $G\subseteq X$ is a closed algebraic subgroup, then $G^{\abvar}=(G^\circ)_{\operatorname{red}}$, the reduction of the connected component of the identity.  Indeed, over a perfect field, the reduction of an algebraic group is an algebraic group, while over an arbitrary field, the same holds for proper algebraic groups \cite[Exp.~VI, Prop.~3.1]{FGA}.
\end{rem}

\subsection{Picard and Albanese schemes}
\label{S:picalbbackground}
We now turn our attention to Picard and Albanese schemes.  

\subsubsection{Picard schemes for families of smooth projective varieties with vanishing $H^2$}\label{S:PicVanH2}
Assume  that $f\colon C\to S$ is a smooth and projective  morphism with geometrically integral fibers, and $R^2f_*\mathcal O_C=0$; e.g., $C/S$ is a smooth projective curve over $S$.
 We use \cite{kleimanPic} as a reference for the following standard results on Picard schemes.  First,
  using  that $C/S$ is projective and flat  with geometrically integral fibers,
  we
  can invoke  \cite[Thm.~9.4.8]{kleimanPic} to conclude the \'etale
  sheafification
  of the Picard functor is representable by a separated scheme
  $\operatorname{Pic}_{C/S}$.
  Then using the additional assumption that $R^2f_*\mathcal O_C=0$, we can
  invoke
  \cite[Prop.~9.5.19]{kleimanPic} to conclude that $\operatorname{Pic}_{C/S}$ is
  smooth.  Denoting by $\operatorname{Pic}^0_{C/S}$ the connected component of
  the
  identity, we can use \cite[Thm.~9.5.4]{kleimanPic} to conclude that the fiber
  of
  $\operatorname{Pic}^0_{C/S}$ over every point is projective.  It then follows
  from \cite[Prop.~9.5.20]{kleimanPic} that $\operatorname{Pic}^0_{C/S}$ is
  proper
  over $S$.  In other words, it is an abelian scheme over $S$. By construction it is stable under base change.

  In fact, even if $C \to S$ is merely proper (but not projective;
  this can happen if the fibers of $C$ have genus one), it is
  stil true that $\pic^0_{C/S}$ is an abelian scheme.  Indeed, after
  an \'etale base change $T \to S$, $C_T \to T$ is projective, and
  thus $\pic^0_{C_T/T}$ is an abelian scheme.  A descent argument (see Cesnavicius on \href{https://mathoverflow.net/q/204182}{{\tt MathOverflow (204182)}})
 then shows that $\pic^0_{C_T/T}$ descends to
  $S$ as an abelian scheme.

\subsubsection{Picard schemes for torsors under abelian schemes}\label{S:PicTors}
Let $T/S$ be a torsor under an abelian scheme $X$ over $S$.  The same argument as  above
implies the existence of the Picard scheme
$\operatorname{Pic}_{T/S}$ representing  the \'etale
  sheafification
  of the Picard functor.  The fact that in this situation the
connected component $\operatorname{Pic}^0_{T/S}$ is an abelian scheme
is for instance \cite[Prop.~2.1.5]{olssonAV}.  In fact, given any
$S$-scheme $S'\to S$, the set of isomorphisms of $X_{S'}$-torsors
$\iota\colon  X_{S'}\to T_{S'}$ is canonically in bijection with the set
$T(S')$.  Choosing an \'etale cover $S'/S$ over which  $T_{S'}$ is isomorphic to $X_{S'}$  
 we obtain a non-empty set $T(S')$, and by descent, any such isomorphism $\iota$ then induces an isomorphism $\iota_*\colon \widehat X\to \operatorname{Pic}^0_{T/S}$.  This construction gives a  canonical morphism $\widehat X\times_S T\to \operatorname{Pic}^0_{T/S}$, $(L, \iota)\mapsto \iota_*L$, which factors uniquely through the first  projection 
$
\xymatrix{
\widehat X\times_S T \ar[r]^<>(0.5){pr_1} & \widehat  X \ar[r]^<>(0.5)\zeta & \operatorname{Pic}^0_{T/S}
}
$  \cite[Prop.~2.1.5]{olssonAV}, 
giving a canonical isomorphism:
\begin{equation}\label{E:PicTors}
\xymatrix{
\zeta\colon \widehat X \ar[r]& \operatorname{Pic}^0_{T/S}}.
\end{equation}

Note that given a Poincar\'e line bundle over $\operatorname{Pic}^0_{T/S}\times_S T$, we can view it alternatively as a family of line bundles over $\operatorname{Pic}^0_{T/S}= \widehat X$ parameterized by $T$, and we obtain a morphism $T\to \operatorname{Pic}_{\widehat X/S}$, which identifies $T$ with a component of $\operatorname{Pic}_{\widehat X/S}$.  Indeed, it suffices to check this on geometric fibers, in which case the torsor is trivial, and the assertion comes down to the fact that $\widehat{\widehat{X}}=X$.  Note that if the universal object for $\operatorname{Pic}^0_{T/S}$ is not a line bundle, so that there is not a Poincar\'e line bundle, one can make an \'etale base change, and then use descent, to obtain the same result.  In other words, a torsor under $X$ over $S$ can be viewed as a component of $\operatorname{Pic}_{\widehat X/S}$.

  \subsubsection{Albanese schemes for familes with $\operatorname{Pic}^0$ an abelian scheme} \label{S:AlbTors}
  We now assume that $f\colon C\to S$ is a smooth projective  morphism with
  geometrically integral fibers, or that it is a smooth proper
  morphism whose fibers are geometrically integral curves; either implies that the \'etale sheafification of the Picard functor
  is representable by a scheme $\operatorname{Pic}_{C/S}$ (see the discussion in  \S
  \ref{S:PicVanH2}), and we
  assume further that the connected component
  $\operatorname{Pic}^0_{C/S}$ is an abelian scheme (e.g., this holds in the
  situation of \S \ref{S:PicVanH2} or \ref{S:PicTors}).
     
  Given an abelian scheme $X/S$,  then via the definition of the Picard scheme,  
 there is a bijection between:
\begin{itemize}
\item $S$-morphisms $C\to T$ where $T/S$ is some torsor under $X$.
\item $S$-morphisms $ \widehat X \to \pic_{C/S}$.
\end{itemize}
Indeed, given a Poincar\'e line bundle on $\operatorname{Pic}_{C/S}\times_S C$, a morphism $\widehat X\to \operatorname{Pic}_{C/S}$ corresponds to a line bundle $L$ on the product $\widehat X\times_S C$.  Now viewing $L$ instead as a family of line bundles on $\widehat X$ parameterized by $C$, we have a morphism $C\to \operatorname{Pic}_{\widehat X/S}$.  Looking at the component containing the image of $C$ under this map, we get a map $C\to T$ for some torsor $T$ under $\operatorname{Pic}^0_{\widehat X/S}=X$. Again, if the universal object for $\operatorname{Pic}_{C/S}$ is not a line bundle, so that there is not a Poincar\'e line bundle, one can obtain the same result by taking an \'etale base change, and then using descent.   As every torsor under $X$ over $S$ can be viewed as a component of $\operatorname{Pic}_{\widehat X/S}$ (\S \ref{S:PicTors}), one can reverse the above construction, giving the asserted equivalence.

    From this equivalence and the identity map $\widehat {\widehat {\operatorname{Pic}^0_{C/S}}} = \operatorname{Pic}^0_{C/S} \to {\operatorname{Pic}^0_{C/S}}
    \subseteq \operatorname{Pic}_{C/S}$, one obtains a canonical morphism 
    $$\operatorname{alb}^1\colon C\to \operatorname{Alb}^1_{C/S}$$ 
    where $\operatorname{Alb}^1_{C/S}$ is a torsor under $\operatorname{Alb}^0_{C/S}:=\widehat {\operatorname{Pic}^0_{C/S}}$, and moreover,          
    one obtains 
     \cite[Exp.~VI, Thm.~3.3(iii)]{FGA} that $\operatorname{alb}^1\colon C\to \operatorname{Alb}^1_{C/S}$   is an Albanese morphism,  
        meaning that it 
     is initial for $S$-morphisms of $C$ into torsors under abelian $S$-schemes.  A section of $C/S$ gives an identification $\operatorname{Alb}^1_{C/S}\cong \operatorname{Alb}^0_{C/S}$ so that the composition of $\operatorname{alb}^1$ with this isomorphism gives an $S$-morphism $$\operatorname{alb}^0\colon C\to \operatorname{Alb}^0_{C/S}$$ that is initial for $S$-morphisms of $C$ into abelian $S$-schemes taking the section of $C/S$ to the zero section.  These constructions are stable under base change.

\subsubsection{Albanese schemes for curves}\label{S:AlbCurve} 
By a \emph{curve} we will mean a geometrically integral separated scheme of dimension one over a field, and by a curve over $S$ we will mean a flat morphism $C\to S$ such that every fiber is a curve.

Suppose $f:C \to S$ is a smooth proper  curve over $S$,
and  let $\operatorname{Div}_{C/S}$ be the functor of relative
effective divisors on $C$ \cite[Def.~9.3.6]{kleimanPic}; there is a
natural transformation of functors $C \to \operatorname{Div}_{C/S}$,
where $\sigma \in C_T(T)$ gets sent to the corresponding relative
effective divisor on $C_T$.  Composing this with the Abel map, a
natural transformation  $\operatorname{Div}_{C/S} \to \pic_{C/S}$
\cite[Def.~9.4.6]{kleimanPic}, yields a morphism of functors, and thus
of schemes, $\alpha: C \to \pic_{C/S}$.  Let $\pic^{(1)}_{C/S}$ be the
component of $\pic_{C/S}$ which contains the image of $\alpha$; it is
a torsor under $\pic^0_{C/S}$.
  We denote the induced $S$-morphism $$\xymatrix{\alpha^{(1)}\colon C\ar[r] & \operatorname{Pic}^{(1)}_{C/S}},$$
 and this agrees  with the  Albanese torsor discussed in \S \ref{S:AlbTors}.
 Moreover, a section $P\colon S\to C$ of $C/S$ gives an identification $\operatorname{Pic}^{(1)}_{C/S}\cong \operatorname{Pic}^0_{C/S}$ so that the composition of $\alpha^{(1)}$ with this isomorphism gives an $S$-morphism 
 $$
 \xymatrix{\alpha_P\colon C\ar[r]& \operatorname{Pic}^0_{C/S}}
 $$
 which agrees with the pointed Albanese defined above
   (see also 
\cite[Prop.~III.6.1]{milneAV} in the case where $S$ is the spectrum of
a field).  When $S=\operatorname{Spec}k$ for an algebraically closed field $k$, then $\alpha^{(1)}$ is given by $Q\in C(k)\mapsto \mathcal O_C(Q)$, the identification $\operatorname{Pic}^{(1)}_{C/k}\cong \operatorname{Pic}^0_{C/k}$ is given by $\mathcal L\mapsto \mathcal L(-P)$, and $\alpha_P$ is given by $Q\in C(k)\mapsto \mathcal O_C(Q-P)$.

\section{Norm endomorphisms, complements, and Prym schemes}
\label{S:morenotes}

The goal of this section is to review the notion of  projectors for sub-abelian schemes of an abelian scheme.   More precisely, given an inclusion  $Y \hookrightarrow X$ of abelian schemes over $S$, 
one would like to have an endomorphism $\epsilon_Y$ of $X$ such that $\operatorname{Im}(\epsilon_Y)= Y$ and $\epsilon_Y^2=\epsilon_Y$.  The norm 
endomorphism $N_Y$ of $Y$, constructed via an auxiliary polarization $\lambda$ of $X$, achieves this, up to the exponent $e$ of the polarization restricted to $Y$, in the sense that $\operatorname{Im}N_Y=Y$ and 
$N_Y^2=eN_Y$.    Consequently, one may take $\epsilon_Y:=\frac{1}{e}N_Y\in \operatorname{End}_{\mathbb Q}(X/S)$   
and obtain a projector in the rational endomorphism ring.  
If the exponent $e$ of the polarization restricted to $Y$ is stable under base change, e.g., the restricted polarization is a positive integer multiple of a principal polarization (see \Cref{L:ExpBC}), then so is the norm map $N_Y$ and the projector $\epsilon_Y$. 
We review this here, as we are not aware of a reference for this in
the relative setting.

To carry out the full construction, we will need to make at least one
of the following assumptions:

\begin{condition}
  \label{comphyp}
The abelian scheme $X/S$ is equipped with a polarization $\lambda:X
\to \hat X$ and at least one of the following holds:
\begin{alphabetize}
\item $S$ satisfies \Cref{alwaysepi};
 \item $\deg(\lambda)$ is invertible on $S$.
\end{alphabetize}
\end{condition}
Note that if $\lambda$ is principal then $(X,\lambda)/S$ automatically
satisfies \Cref{comphyp} without additional hypothesis.

\subsection{Norm endomorphisms and complements}
\label{S:normendomorphism}
Let $\iota= \iota_Y\colon Y \hookrightarrow X$ be an inclusion of
abelian schemes over $S$, and assume that $X$ admits a polarization
$\lambda\colon X \to \widehat X$. \emph{Starting in \Cref{L:kerNab},
  and until the end of this section, we will assume that
  \Cref{comphyp} holds.}

Let 
\begin{equation}\label{E:eY}
e = e_Y = e_{Y\subseteq X, \lambda}
\end{equation}
 be the exponent (\S \ref{S:DgExpTp}) of the polarization $\iota_Y^*\lambda:= \widehat \iota \circ \lambda \circ \iota$, 
       the pullback of $\lambda$ by $\iota_Y$:
\[
\xymatrix{
  Y \ar[r]^{\iota^*\lambda} \ar[d]^{\iota}& \widehat Y \\
  X \ar[r]^\lambda & \widehat X\ar[u]^{\widehat\iota}
}
\]
Define the norm map $N_Y = N_{Y\subseteq X, \lambda}: X\to X$ as the composition
\begin{equation}\label{E:NmY}
\xymatrix@C=3.5em{
Y \ar[r]^{\iota} &  X \ar[r]^\lambda \ar@/^5ex/^{N_Y}[rrrr]\ar@/_3ex/_{M_Y}[rrr]& \widehat X
  \ar[r]^{\widehat\iota} & \widehat Y \ar[r]^<>(0.5){e_Y\cdot  (\iota^*\lambda)\inv} & Y \ar@^{(->}[r]^\iota& X}
\end{equation}
We similarly define the map $M_Y:X\to Y$ as indicated in the diagram.

\begin{rem}[Norm maps and base change]
The norm map $N_Y$ is stable under faithfully flat base change (since the exponent $e_Y$ is), and 
if the exponent $e_Y$ is stable under arbitrary base extension, e.g.,
$\iota_Y^*\lambda$ is a positive integer multiple of a principal
polarization  (see \Cref{L:ExpBC}), then so is the norm map $N_Y$.
All of the constructions and results that follow in \S
\ref{S:normendomorphism} have the same base change properties.  
\end{rem}

\begin{lem}\label{BL:L531}
  We have
  \begin{alphabetize}
  \item  \label{BL:L531a} $M_Y \iota = N_Y|_Y= [e_Y]_Y$.
  \item \label{BL:L531b} $N_Y^2 = e_Y N_Y$.
     \end{alphabetize}
\end{lem}

\begin{proof} Assertion (b) is a straight forward computation from the definitions (see \cite[Lem.~5.3.1]{BL}, where the same argument works verbatim); 
(a) then follows  (see e.g., \cite[p.~125, eq.~(1)]{BL}).
\end{proof}

\begin{lem}\label{L:newZvsWTZ}
We have
  \[
    [e_Y](\ker N_Y) \subseteq \im([e_Y]-N_Y) \subseteq \ker N_Y.
  \]
\end{lem}

\begin{proof}
  On one hand, since $N_Y|_{\ker(N_Y)} = [0]$, we have
\[
    [e_Y](\ker N_Y) = ([e_Y]-N_Y)(\ker N_Y) \subseteq ([e_Y]-N_Y)(X) =
    \im([e_Y]-N_Y).
  \]
  On the other hand, using \Cref{BL:L531}(b), we see that
  \[
    N_Y([e_Y]-N_Y) = N_Y[e_Y] - N_Y^2 = e_YN_Y - e_YN_Y = 0,
  \]
  and so $N_Y(\im([e_Y]-N_Y)) = (0)$.
\end{proof}

To proceed, we place a modest condition on the polarized abelian
scheme $(X,\lambda)$.

\begin{lem}
  \label{L:kerNab}
If $(X,\lambda)/S$ satisfies \Cref{comphyp}, then the group scheme $\ker(N_Y)$ admits a maximal abelian subscheme
  $\ker(N_Y)^\ab$, and $\dim \ker(N_Y)^\ab = \dim \ker(N_Y)$.
  \end{lem}

\begin{proof}
  The exact sequence of abelian schemes
  \[
  \xymatrix{
    0 \ar[r] & Y \ar[r]^\iota & X \ar[r] & X/Y \ar[r] & 0}
  \]
  yields a dual exact sequence
  \[
  \xymatrix{
    0 \ar[r] & \widehat{X/Y} \ar[r] & \widehat X \ar[r] & \widehat Y
    \ar[r] & 0.}
  \]
  Let $Z = \lambda\inv(\widehat{X/Y})^\abvar$ (Lemma \ref{L:maxab}).  Then $Z \subseteq
  \ker(N_Y)$, and $\dim Z = \dim \lambda\inv(\widehat{X/Y}) = \dim
  \ker(N_Y)$.  (The last equality holds because $e_Y
  (\iota^*\lambda)\inv$ is an isogeny.)
               \end{proof}

For reasons we explain below in \S\ref{S:compl}, we call the abelian
scheme
\begin{equation}
  \label{E:defcomplement}
Z := \ker(N_Y)^\abvar = \lambda\inv(\widehat{X/Y})^\abvar
\end{equation}
constructed in Lemma \ref{L:kerNab} \emph{the complement of $Y$ in $X$
(with respect to $\lambda$)}.

Alternatively, we could define $Z$ using the following equivalent
formulation.
\begin{lem}
  \label{L:Zisimage}
  We have
  \[
    \ker(N_Y)^{\abvar} = \im([e_Y]_X - N_Y).
  \]
  In particular, $N_Y$ and $[e_Y]_X - N_Y$ are epi-abelian.
\end{lem}

\begin{proof}
  Using \Cref{L:newZvsWTZ}, we see that
  \[
    [e]Z = Z \subseteq \im([e]-N_Y) \subseteq \ker(N_Y).
  \]
  Since $\dim_SZ = \dim_S\ker(N_Y)$, the result follows from \cite[Thm.~C]{ACMimab}.
\end{proof}

\begin{rem}
When $S=\operatorname{Spec}K$ for a field $K$, we have $Z=  ((\ker N_Y)^\circ)_{\operatorname{red}}$ (\Cref{R:LAS}), and in particular, over a field of characteristic $0$ we have $Z=(\ker N_Y)^\circ$. 
\end{rem}

\begin{lem}\label{L:KYcapY}
  We have $(\ker N_Y) \times_X Y = Y[e]$.
   \end{lem}
\begin{proof}
  We have $(\ker N_Y) \times_X Y = \ker(N_Y\rest Y) = Y[e]$ (\Cref{BL:L531}\ref{BL:L531a}).
\end{proof}

\subsubsection{Complements are complementary}\label{S:compl}
We now establish a few results justifying the terminology that $Z$ is complementary to $Y$, and showing that $Y$ is the complement of $Z$.

\begin{lem}\label{L:compl}
  The sub-abelian schemes $Y$ and $Z$ are complementary          in the sense that the addition map 
 \begin{equation}\label{E:muDef}
     \mu := \iota_Y +
  \iota_Z: Y\times_S Z \longrightarrow X
\end{equation}
  is an isogeny of abelian
  schemes over $S$.
  
  The kernel of $\mu$ is  
  \begin{equation}\label{E:kerMu}
  \ker(\mu) = Y\times_XZ \subseteq
  Y\times_X (\ker N_Y)= Y[e],
\end{equation}
so that $\deg \mu  \le e^{2\dim Y}$, and moreover, 
  equality holds in \eqref{E:kerMu}, as well as for the degree inequality, if $\ker N_Y$ is an abelian $S$-scheme; i.e., if
  the natural inclusion $Z\subseteq \ker N_Y$
  is an equality. 
   \end{lem}

\begin{proof}
  Since $\ker(\mu) = Y\times_XZ \subseteq
  Y\times_X(\ker N_Y)= Y[e]$ (\Cref{L:KYcapY}) has relative dimension zero over $S$,
it follows that   $\mu$ is an isogeny onto its image.  Considering the
definition of the norm map \eqref{E:NmY}, and \Cref{BL:L531}\ref{BL:L531b}, it
follows that $\operatorname{Im}N_Y=Y$, and therefore $\dim_S Y+\dim_S (\ker N_Y)=\dim_S X$.  Thus, as $\dim_S(Y) + \dim_S(Z) = \dim_S(Y)+\dim_S(\ker N_Y) =
  \dim_S(X)$, the image of $\mu$ is all of $X$.
  The assertion on degrees now follows from the containment $\ker(\mu)
  \subseteq Y[e]$. Finally, if $\ker N_Y$ is an abelian scheme, then we have $Z=\ker N_Y$ (\Cref{L:Zisimage}), so that $Y\times_XZ= Y\times_X(\ker N_Y)= Y\times_X (\ker N_Y)$, which equals $Y[e]$ by \Cref{L:KYcapY}.
\end{proof}

Now consider the norm map $N_Z = N_{Z\subseteq X,\lambda}$ for the complement $Z$ of $Y$ in $X$.  For later reference, we collect a few results about the interactions of $N_Y$ and $N_Z$:

\begin{lem}\label{L:BL125(23)}
  We have 
  \begin{alphabetize}
  \item $N_Y|_Y=[e_Y]$ and $N_Z|_Z=[e_Z]$.
  
    \item  $N_Y^{(\dagger)} = N_Y$, and $N_Z^{(\dagger)} = N_Z$, where $(\dagger)$ is the Rosati involution induced by $\lambda$.
    
\item   $N_Z N_Y=0$ and $N_YN_Z = 0$, so that $N_Y|_Z=0$ and $N_Z|_Y=0$.

\item $e_YN_Z+e_ZN_Y=[e_Ye_Z]_X$.
\end{alphabetize}

\end{lem}

\begin{proof}  (a) is just a restatement of \Cref{BL:L531}(a). 
(b) is a straight forward computation from the definitions (see \cite[Lem.~5.3.1]{BL}, where the same argument works verbatim).
 For (c), 
 since the image of $N_Z$ is $Z$, and $Z$ is contained in the kernel of $N_Y$, we have $N_Y N_Z = 0$.  For the other claim, use the fact that $(N_Z N_Y)^{(\dagger)} = N_Y^{(\dagger)} N_Z^{(\dagger)} = N_Y N_Z$.
 For (d), this follows immediately on restriction to $Y$ and $Z$ from (a) and (c).  Then, using the fact that the addition map $\mu\colon Y\times_SZ\to X$ is surjective, (d) follows for all of $X$.    (See also  \cite[p.~125, eqs.~(1)--(4)]{BL} for $S = \spec \cx$.)
\end{proof}

\begin{lem}
  The image of $[e_Z]-N_Z$ is $Y$.
\end{lem}

\begin{proof}
  The image of $[e_Z]-N_Z$ is the largest abelian scheme contained in the kernel of $N_Z$ (apply \Cref{L:newZvsWTZ} and \Cref{L:Zisimage} to $Z\subseteq X$); now use a dimension count.
\end{proof}

Thus, $(Y_0,Y_1) :=   (Y,Z)$ form a complementary pair in a symmetric
sense:
\begin{align} \label{E:Sym1}
N_{Y_i}(Y_i) &= Y_i,\\
([e_{Y_i}]-N_{Y_i})(Y_i) &= Y_{1-i}. \label{E:Sym2}
\end{align}

 \begin{rem}
The map
\[
  \xymatrix{
(N_Y,N_Z): X\ar[r] &  Y\times_S Z}\]
is an isogeny.  In fact, this follows from \Cref{L:BL125(23)}(d),
which says  that $(e_YN_Y,e_ZN_Z)$ is an isogeny inverse of $\mu$, in
the sense that $\mu\circ (e_YN_Y,e_ZN_Z)=[e_Ye_Z]_X$, or
alternatively, that $$\mu^{-1}=\frac{1}{e_Ye_Z} (e_YN_Y,e_ZN_Z)\in
\operatorname{Hom}_{\mathbb Q}(X,Y\times_S Z).$$ (If $S = \spec \cx$,
this is \cite[Thm.~5.3.5]{BL}.)
\end{rem}

\subsection{Poincar\'e reducibility}
 We can use these ideas to prove Poincar\'e reducibility for
 polarizable abelian schemes (abelian schemes that admit a
 polarization).      The starting point is:

\begin{teo}\label{T:C536}
For a polarized abelian scheme $(X,\lambda)$ over $S$, and an abelian subscheme $Y\subseteq X$ with complement $Z\subseteq X$, 
the addition map
\[\xymatrix{
\mu\colon (Y,\iota_Y^*\lambda)\times (Z,\iota_Z^*\lambda)\ar[r]& (X,\lambda)}\]
is an isogeny of polarized abelian schemes.
\end{teo}

\begin{proof}
  The proof of \cite[Cor.~5.3.6]{BL} holds in this situation, using  \Cref{L:BL125(23)}.
      \end{proof}

From this one obtains: 

\begin{cor}\label{C:PoincRed}
 Suppose that  $S$ satisfies \Cref{alwaysepi}.  Given a polarizable abelian scheme $X$ over $S$, there is an isogeny
\begin{equation}\label{E:PoincRed}
\xymatrix{X_1^{n_1}\times_S\cdots \times_SX_r^{n_r} \ar[r]& X}
\end{equation}
with simple polarizable abelian schemes $X_\nu$ over $S$ not isogenous to each other.  Moreover, the polarizable abelian schemes $X_\nu$ and the integers $n_\nu$ are uniquely determined up to isogenies and permutations.  Finally, if $X$ is projective over $S$, so are the $X_{\nu}/S$.  
\end{cor}

\begin{proof}  (The cases $S = \spec \cx$, $\spec \bar K$ and $\spec
  K$ may be found, respectively, in \cite[Thm.~5.7]{BL}, \cite[\S 19
    Thm.~1]{mumfordAV}, \cite[Cor.~3.20]{conradtrace}.)
  We may and do fix a polarization on $X$. Utilizing \Cref{T:C536}, 
one obtains such a decomposition by induction on dimension (since a
non-simple abelian scheme admits a nontrivial sub-abelian scheme, which in turn
admits a complement), and replacing isogenous abelian schemes with copies of one another.  

The uniqueness can be shown as follows.  Given another such isogeny $Y_1^{m_1}\times_S\cdots \times_SY_t^{m_t}\to X$,
one obtains an isogeny $X\to Y_1^{m_1}\times_S\cdots \times_SY_t^{m_t} $ from \Cref{P:BL1.2.6};  then by considering the composition $X_1^{n_1}\times_S\cdots \times_SX_r^{n_r} \to X \to Y_1^{m_1}\times_S\cdots \times_SY_t^{m_t}$ 
and the 
projections onto each component, one concludes uniqueness, as a nonzero map between simple abelian schemes is an isogeny.
Note that this also shows that the decomposition is independent of the choice of polarization of $X$.

Finally, if $L$ on $X$ is relatively ample, then the pull-backs of $L$ via \eqref{E:PoincRed} are also relatively ample. 
\end{proof}

\begin{rem} Note that the decomposition \eqref{E:PoincRed} in \Cref{C:PoincRed} is \emph{not} stable under base change.  For instance one can construct families of principally polarized abelian varieties over $\mathbb C$ (e.g., Jacobians of curves) whose very general fibers are simple,  but which degenerate to products of abelian varieties.  Such a family must be simple, since the very general fiber is simple, but the special fiber is  not simple.   
In fact, in contrast to other base change properties we have been
discussing, the decomposition in \Cref{C:PoincRed} is not even stable
under faithfully flat base change.  Indeed, there are examples of
abelian varieties that are simple, but not geometrically simple.  For
instance, let $L/K$ be a nontrivial finite separable extension of
fields, and let $X$ be a 
geometrically simple abelian variety over a field $L$ which admits
no model over $L$ (e.g., let $K=\rat$, let $L = \rat(j)$ for some
irrational algebraic number $j$, and let
$X/L$ be an elliptic curve with $j$-invariant $j$).  Let $Y=\operatorname{R}_{L/K}X$ be the Weil restriction of
$X$ to $K$.  Then $Y$ is simple, but not geometrically
simple. For example, if $L/K$ is Galois and $K$ is finitely
generated, then $H^1(Y_{\bar K},\rat_\ell)$ is the induced
representation $\operatorname{Ind}_{\gal(L)}^{\gal(K)}(H^1(X_{\bar
  L},\rat_\ell))$, and thus irreducible as a representation of $\gal(K)$, while $Y_L \iso \prod_{\sigma
  \in \gal(L/K)}X^\sigma$; see, e.g., \cite[\S 1]{milne72}.   \end{rem}

We also note the following consequence of \Cref{T:C536}; recall that a polarized abelian scheme is said to be \emph{indecomposable} if it is not isomorphic as a polarized abelian scheme to a product of polarized abelian schemes.  

\begin{cor}\label{C:e=1}
If  $e_Y=1$, then $(X,\lambda)\cong (Y,\iota_Y^*\lambda)\times_S (Z,\iota_Z^*\lambda)$.  In particular, if $(X,\lambda)$ is indecomposable, then $e_Y=1$  if and only  if $Y=0$ or $Y=X$; if $Y=X$, then $\lambda$ is principal.  
\end{cor}

\begin{proof}
This is a direct consequence of \Cref{T:C536} and \Cref{L:compl}.
\end{proof}

\subsection{Complements inside prinicpally polarized abelian varieties}
\label{SS:complementppav}

In the special, but extremely useful, case where $\lambda$ is a principal polarization, we can say more.  Since $\lambda$ is an isomorphism, the complement $Z$ of $Y$ is simply
\begin{equation}\label{E:BL1213-1}
Z=
\lambda^{-1}(\ker \widehat{\iota_Y})\cong \widehat{X/Y}.
\end{equation}
In fact, there is a short exact sequence
\begin{equation}\label{E:BL1213-2}
\xymatrix{
0\ar[r]& Y \ar[r]^{\iota_Y}& X \ar[r]^{\widehat \iota_Z \lambda }& \widehat Z \ar[r]& 0
}
\end{equation}
and under the isomorphism $\widehat Z\cong X/Y$, the composition $\iota_Z^*\lambda: Z\stackrel{\iota_Z}{\to}X \stackrel{\widehat {\iota_Z}\lambda}{\to} \widehat Z $ is identified with the composition $Z\stackrel{\iota_Z}{\to}X \to X/Y$. 

Moreover, we can characterize the kernel of the induced polarization on $Y$.

\begin{cor}\label{P:BL-1214}
  \label{L:kerpullbackpolarization}
 If $\lambda$ is principal, then $\K(\iota_Y^*\lambda) = Y\times_XZ \subseteq Y$.
\end{cor}

\begin{proof} (See {\cite[Cor.~12.1.4]{BL}} for the case $S = \spec \cx$.)
We have $\K(\iota_Y^*\lambda):= \ker \iota_Y^*\lambda=\ker (\widehat \iota_Y~\lambda \iota_Y)= \iota_Y^{-1}\lambda^{-1}(\ker \widehat \iota_Y)=\iota_Y^{-1}Z:= Y\times_XZ$.
\end{proof}

\begin{rem}\label{R:Ki*Theta}
Note by symmetry (e.g., \eqref{E:Sym2}) that if $\lambda$ is principal we also have $$\operatorname{K}(\iota_Z^*\lambda)=Z\times_XY  = Z \times_X Y \subseteq Z,
$$ 
so that  $$\operatorname{K}(\iota_Z^*\lambda)\cong \operatorname{K}(\iota_Y^*\lambda).$$ In particular, the exponents of the kernels of $\iota_Y^*\lambda $ and $\iota_Z^*\lambda$ are the same; i.e., $$e_Y=e_Z.$$  
In addition, if follows that if either of $\iota_Z^*\lambda$ or
$\iota_Y^*\lambda$ has a type (see \S \ref{S:DgExpTp}), then so does
the other.   In fact, without loss of generality, if $\dim Z\le \dim Y$,
and $Z$ has type $\delta = (\delta_1,\dots,\delta_{\dim Z})$, then
$Y$ has type  $(1,\dots,1,\delta_1,\dots,\delta_{\dim Z})$.  (See
{\cite[Cor.~12.1.5]{BL}} for the case $S = \spec \cx$.)
\end{rem}

\begin{cor}\label{C:e=1-ppav}
If $\lambda$ is principal and  $e_Y=1$  (or equivalently,  $e_Z=1$), then  $\iota_Y^*\lambda$ and $\iota_Z^*\lambda$ are principal polarizations and  $(X,\lambda)\cong (Y,\iota_Y^*\lambda)\times_S (Z,\iota_Z^*\lambda)$.  In particular, if $(X,\lambda)$ is indecomposable, then $e_Y=e_Z=1$  if and only  if  $Y=X$ or $Z=X$.
\end{cor}

\begin{proof}
This follows immediately from \Cref{C:e=1} and \Cref{R:Ki*Theta}.
\end{proof}

\subsection{Symmetric idempotents, projectors, and abelian subschemes}

We now show that norm endomorphisms and abelian subschemes are equivalent.
 
           We start with the observation that if we set $\epsilon_Y= \frac{1}{e_Y}N_Y\in \operatorname{End}_{\mathbb Q}(X)$, then as a consequence of \Cref{BL:L531}(b) and \Cref{L:BL125(23)}(b), we have
$$
\epsilon_Y^{(\dagger)}= \epsilon_Y \ \ \ \epsilon_Y^2= \epsilon_Y.
$$
We call any endomorphism $\epsilon\in \operatorname{Hom}_{\mathbb
  Q}(X)$ such that $\epsilon^{(\dagger)}=\epsilon$ and
$\epsilon^2=\epsilon$ a \emph{symmetric idempotent}, or a symmetric
projector.   For a symmetric idempotent $\epsilon$ there is a positive
integer $n$ such that $n\epsilon \in \operatorname{End}(X)$, and we
define $$X^\epsilon:= \operatorname{Im}(n\epsilon)$$ to be the image
of $n\epsilon$; this is independent of the choice of $n$.  This
construction is only useful to us if the image is an abelian scheme,
in which case we will say that $\epsilon$ is a \emph{symmetric
  epi-abelian idempotent}.  (The reader should not be too alarmed at this apparent restriction.  On one hand, in many natural settings, an idempotent, like any morphism of abelian schemes, is epi-abelian (\Cref{L:ACMimab}); on the other hand, the authors have never encountered a symmetric idempotent which is not epi-abelian.)

We therefore have maps of sets 
$$
\xymatrix{
\{\text{Abelian subschemes of } X\}
\ar@/^1pc/[r]^<>(0.5){\epsilon_{(-)}}&  \{\text{Symmetric epi-abelian idempotents in } \operatorname{End}_{\mathbb Q}(X)\} \ar@/^1pc/[l]^<>(0.5){X^{(-)}}
}
$$

\begin{teo}\label{BL:T532}
The maps $X^{(-)}\circ \epsilon_{(-)}$ and $\epsilon_{(-)}\circ X^{(-)}$ are the identity maps.  
 \end{teo}

\begin{proof}
The fact that $X^{(-)}\circ \epsilon_{(-)}$ is the identity map comes down to the fact that by definition, given an abelian subscheme $Y$ of $X$, the image of $N_Y$ is $Y$ (see \eqref{E:NmY} and \Cref{BL:L531}(a)).  
Note that this shows also that $X^{(-)}$ is surjective.  

Therefore, for the converse, it suffices to show that $X^{(-)}$ is
injective;  i.e., given two symmetric epi-abelian idempotents $\epsilon_1$ and $\epsilon_2$ such that $X^{\epsilon_1}=X^{\epsilon_2}$, one must show $\epsilon_1=\epsilon_2$.  By rigidity, it suffices to show this when restricted to geometric fibers.
The rest of the argument is identical to \cite[Thm.~5.3.2]{BL}
  as  the Rosati involution is stable under base change, so that the proof reduces to the positive definiteness of the trace pairing on geometric fibers \cite[\S 21, Thm.~1]{mumfordAV}.
\end{proof}

\begin{cor}\label{BL:C533}
  If $f \in
  \operatorname{End}(X)$ and $Y = \operatorname{Im}(f)$ is an abelian subscheme, then $f = N_Y$
    if and only if $f^2 = e_Yf$ and $f^{(\dagger)} = f$. 
   \end{cor}
  
  \begin{proof} The forward implication is  \Cref{BL:L531}\ref{BL:L531b} and \Cref{L:BL125(23)}.  For the converse, 
we have $\epsilon:= \frac{f}{e}$ is a symmetric epi-abelian idempotent, and then
one employs \Cref{BL:T532}.   (When $S = \spec \cx$, this is {\cite[Cor.~5.3.3, Crit.~5.3.4]{BL}}.)
\end{proof}

\begin{cor}\label{C:eps_Y}
There is an involution on the set of symmetric epi-abelian idempotents given by $\epsilon\mapsto (1-\epsilon)$.  
 If $(Y_0,Y_1)$ are complementary abelian subschemes, then $\epsilon_{Y_i}=1-\epsilon_{Y_{1-i}}$. 
\end{cor}

\begin{proof}
The first assertion is clear.  The second follows from \Cref{BL:T532} and \eqref{E:Sym2}.  
\end{proof}

\begin{cor}\label{BL:C534}
If $\lambda$ is principal, then $N_Y$ is primitive.
\end{cor}

\begin{proof}
Due to \Cref{T:factor}, we only have to show that $\ker N_Y\nsupseteq
X[n]$ for any integer $n\ge 2$. 
This follows from the definition \eqref{E:NmY}, since $\lambda$ is
assumed to be an isomorphism, $\widehat \iota$ has connected kernel
(\Cref{L:dualexact}), and  $e_Y(\iota_Y^*\lambda)^{-1}$ is primitive
(\Cref{R:gPrim}).    (If $S = \spec \cx$, this is \cite[Crit.~5.3.4]{BL}.)
\end{proof}

 \subsection{Complements inside principally polarized abelian
  varieties: automorphisms}\label{S:ppavAut} 
We return now to the case that  $\lambda$ is a principal polarization,
and we assume that the principally polarized abelian $S$-scheme
$(X,\lambda)$ admits an automorphism $\sigma$ over $S$  of finite
order $n$.  In particular,  $$\widehat \sigma \lambda = \lambda \sigma \ \
\ \text{ and } \ \ \ \sigma^n=1.$$
To make progress, suppose that the endomorphism
\[
  \tau := 1 + \sigma + \sigma^2 + \dots + \sigma^{n-1} \in \End(X)
\]
is epi-abelian.  (This holds in the key case of interest where $X
=\pic^0_{C/S}$ is the Picard scheme of a smooth projective curve and
$\sigma$ is induced by an automorphism of $C$, as in \Cref{R:C/sigma},
as well as over any base scheme $S$ which satisfies \Cref{alwaysepi}.)  Now consider the abelian subscheme
\[
  Y := \im(\tau)\subseteq X.
\]
For context, define the $\sigma$-fixed subgroup scheme
$$X^\sigma:= \ker (1-\sigma),$$
and observe that $(1-\sigma)(1+\sigma +\sigma^2+\dots+\sigma^{n-1})=0$, so that $Y\subseteq X^\sigma$.
At the same time, since on $X^\sigma$ we have that $\sigma =1$, it follows that on $X^\sigma$ we have 
 $\tau|_{X^\sigma} =[n]_{X^\sigma}$.  Therefore,  
$[n] X^\sigma = \tau  X^\sigma \subseteq \tau X = Y$.  In other words,
$$
[n]X^\sigma \subseteq Y\subseteq X^\sigma.
$$
As in the proof of Lemma \ref{L:Zisimage}, we conclude that
$Y=(X^\sigma)^{\abvar}$. 
         
Next observe that since $(1+\sigma+\sigma^2+\cdots+\sigma^{n-1})^2= n(1+\sigma+\sigma^2+\cdots +\sigma^{n-1})$, and  $\sigma^{(\dagger)}=\lambda^{-1}\widehat \sigma \lambda = \lambda ^{-1}\lambda \sigma = \sigma$
we have that 
$$
\frac{1}{n}\left(1+\sigma+\sigma^2+\cdots+\sigma^{n-1}\right) = \frac
1n \tau =\epsilon_Y
$$
is the symmetric idempotent associated to $Y$ (\Cref{BL:T532}).

Now setting $Z$ to be the complement of $Y$ in $X$, we have from \Cref{C:eps_Y} that $\epsilon_Z=1-\epsilon_Y$.  Consequently, $Z$ is the image of 
\begin{equation}\label{E:1-sigma}
n\epsilon_Z=(n-1)-\sigma-\sigma^2-\cdots-\sigma^{n-1}.
\end{equation}

Since $n\epsilon_Y$ is an endomorphism of $X$ and  $N_Y$ is primitive (\Cref{BL:C534}),  we have $e_Y\mid n$, and if $e_Y=1$ then both $Y$ and $Z$ are principally polarized, and 
$(X,\lambda)\cong (Y,\iota_Y^*\lambda)\times_S (Z,\iota_Z^*\lambda)$ as principally polarized abelian schemes (\Cref{C:e=1-ppav}).
As a special case, if $n$ is prime, $(X,\lambda)$ is indecomposable as a principally polarized abelian scheme, and $Y\ne 0$, then $e_Y=n=e_Z$; 
moreover, in this case,  if $\iota_Y^*\lambda$ has a type $\delta$, then $\delta = (1,\dots,1,n,\dots,n)$, and the type of $\iota_Z^*\lambda$ is $(1,\dots,1,n,\dots,n)$, where the number of times $n$ appears in each type is the same (\Cref{R:Ki*Theta}).

\subsection{Definitions of Prym schemes}\label{S:PrymDef}
For clarity, we review several definitions of Prym schemes related to various usages in the literature, where there is some ambiguity, depending on the context.  This section serves to fix the terminology we will be using.  
Roughly speaking, for varieties over a field $K$, an abelian subvariety $Z\subseteq \operatorname{Pic}^0_{C/K}$ of a Jacobian of a curve $C/K$ will be called a Prym--Tyurin variety if  the restriction to $Z$ of the canonical polarization on the Jacobian is a multiple of a principal polarization on $Z$, and will be called a Prym variety if there is a finite cover of curves $f:C\to C'$ so that $Z$ is the complement of $f^*\operatorname{Pic}^0_{C'/K}$, whether or not the restriction of the canonical polarization to $Z$  is a multiple of a principal polarization.

\subsubsection{Prym--Tyurin scheme}\label{S:PTsch}

 An \emph{embedded Prym--Tyurin scheme of exponent $e$} over $S$ is a
 tuple $$(Z,\xi,C,\iota_Z)$$ such that $(Z,\xi)$  is a principally
 polarized abelian scheme over $S$, $C$ is   a smooth proper curve over $S$,  and  $$\xymatrix{\iota_Z\colon Z\ar@{^(->}[r]& \operatorname{Pic}^0_{C/S}}$$ 
is a closed 
subabelian $S$-scheme such that 
$$
\iota_Z^*\lambda_C = e\xi,
$$
where $\lambda_C$ is the canonical principal polarization on $\operatorname{Pic}^0_{C/S}$.

\begin{rem}
We note that $e_Z=e$; i.e.,  the exponent $e_Z$ of $Z$ as a subvariety of  $\operatorname{Pic}^0_{C/S}$ with respect to $\lambda_C$, in the sense of \S \ref{S:normendomorphism},  is $e$, explaining the terminology.  Indeed, the exponent $e_Z$ was defined as the exponent of the isogeny $\iota_Z^*\lambda_C$  (\S \ref{S:isog}).
    \end{rem}

\begin{rem}\label{R:Z[e]-PT}
Given a smooth proper  curve $C/S$ and an abelian subscheme $\iota_Z\colon Z\hookrightarrow \operatorname{Pic}^0_{C/S}$ over $S$, we have from \Cref{L:eXi-X[e]}  that $(Z,\frac{1}{e}\iota_Z^*\lambda_C, C,\iota_Z)$ is an embedded Prym--Tyurin scheme of exponent $e$ over $S$ if and only if   $$\K(\iota_Z^*\lambda_C) = Z[e] :=  \ker [e]_Z,$$
if and only if for every point $s$ of $S$ we have $(Z_s,\frac{1}{e}\iota_{Z_s}^*\lambda_{C_s}, C_s,\iota_{Z_s})$ is an embedded Prym--Tyurin variety of exponent $e$.   
   \end{rem}

A \emph{Prym--Tyurin scheme of exponent $e$} over $S$ 
 is a principally polarized abelian scheme $(Z,\xi)$ over $S$ such that there exists  a tuple $(Z,\xi, C,\iota_Z)$ that is an embedded Prym--Tyurin scheme of exponent $e$.
  In this situation, we also say that $(Z,\xi)$ is a Prym--Tyurin scheme for the curve $C/S$ and that $e$ is the exponent of the Prym--Tyurin scheme $(Z,
\xi)$.

\subsubsection{Generalized Prym--Tyurin scheme}
 An \emph{embedded generalized Prym--Tyurin
  scheme  of generalized exponent $m$} is a tuple $$(Z,\xi,C,\iota_Z,\sigma)$$ where $(Z,\xi,C,\iota_Z)$ is an embedded Prym--Tyurin scheme of some exponent $e$, and 
$\sigma\in \operatorname{End}(\operatorname{Pic}^0_{C/S})$ is an
 epi-abelian endomorphism, symmetric with respect to the Rosati involution given by the canonical polarization $\lambda_C$ on $\operatorname{Pic}^0_{C/S}$, such that 
  \begin{equation}
  \label{E:blochmurre}
\sigma^2+(m-2)\sigma-(m-1)= (1-\sigma )((1-\sigma) -m)=0,
\end{equation}
and $\iota_Z(Z)=\operatorname{Im}(1-\sigma)$.  (In particular,
$1-\sigma$ is epi-abelian.)
For simplicity of notation, we will identify $Z$ with $\iota_Z(Z)$ going forward.  

\begin{wrn}
Note that the terminology in the definition for \emph{generalized} Prym--Tyurin schemes may be confusing, in that one need not have the \emph{generalized} exponent $m$ equal to the exponent $e$ (see \Cref{L:GenPT=PT}).
\end{wrn}

  \begin{rem}
  When $S=\operatorname{Spec}K$ for a field $K$, this definition is due to 
Bloch--Murre  \cite{BlMurre} and Kanev \cite{kanev}.   Note that over a field there is a surjection (discussed in \S \ref{S:regular})
   $\operatorname{CH}^1(C\times_KC)\to \operatorname{End}(\operatorname{Pic}^0_{C/K})$ given by taking a correspondence to the associated endomorphism.  Bloch--Murre and Kanev give their definition in terms of correspondences, but we prefer to work directly with the associated endomorphism. 
\end{rem}

It is worth noting that \eqref{E:blochmurre} is equivalent to $(1-\sigma)^2=m(1-\sigma)$, so that by virtue of  \Cref{BL:T532},
 the condition \eqref{E:blochmurre} is equivalent to the statement
 that  $\frac{1}{m}(1-\sigma)\in \operatorname{End}_{\mathbb
   Q}(\operatorname{Pic}^0_{C/S})$ is the symmetric epi-abelian idempotent $\epsilon_Z=\frac{1}{e}N_Z$ associated to $Z$.

\begin{lem}[Prym--Tyurin is generalized Prym--Tyurin]\label{L:GenPT=PT}
   Given an embedded Prym--Tyurin scheme $(Z,\xi,C, \iota_Z)$ of exponent $e$, then for every positive integer $a$, setting $\sigma = 1-a N_Z\in \operatorname{End}(\operatorname{Pic}^0_{C/K})$, we have that $(Z,\xi,C,\iota_Z,\sigma)$ is an embedded generalized Prym--Tyurin scheme of generalized exponent $a e$; i.e., \eqref{E:blochmurre} holds with $m=a e$.  

Conversely, if $(Z,\xi,C,\iota_Z,\sigma)$ is an embedded generalized Prym--Tyurin scheme of exponent $m$ such that  $(Z,\xi,C,\iota_Z)$ is an embedded Prym--Tyurin scheme of exponent $e$, then $e\mid m$ and $\sigma = 1-\frac{m}{e}N_Z$.   
     \end{lem}

\begin{proof}
  On one hand, suppose $(Z,\xi,C,\iota_Z)$ is an embedded
  Prym--Tyurin scheme of exponent $e$, and let $a$ be a positive integer.  
  Consider the endomorphism
  $\sigma = 1-a N_Z$.  Since $N_Z|_Y=0$ and
  $Y=\operatorname{Im}([e]_X-N_Z)$, we have that $N_Z([e]_X-N_Z)=0$. 
  It follows that $a N_Z([a e]_X-a N_Z)=0$, 
  and thus
  $\sigma$ satisfies \eqref{E:blochmurre} with $m=a e$.  
  
 On the other hand, suppose
 $(Z,\xi,C,\iota_Z,\sigma)$ is an embedded generalized Prym--Tyurin scheme of exponent $m$ such that  $(Z,\xi,C,\iota_Z)$ is an embedded Prym--Tyurin scheme of exponent $e$.  Then, as observed above, due to \Cref{BL:T532}, we have $\frac{1}{m}(1-\sigma) = \frac{1}{e}N_Z$. Thus $1-\sigma = \frac{m}{e}N_Z$.  Since the left hand side is an endomorphism and $N_Z$ is primitive (\Cref{BL:C534}), we have that $e\mid m$, and $\sigma = 1-\frac{m}{e}N_Z$.  
                                \end{proof}

A \emph{generalized Prym--Tyurin scheme of exponent $m$} over $S$ 
 is a principally polarized abelian scheme $(Z,\xi)$ over $S$ such that there exists  a tuple $(Z,\xi, C,\iota_Z,\sigma)$ that is an embedded generalized Prym--Tyurin scheme of exponent $m$.  From the lemma above, this is the same as a Prym--Tyurin scheme of exponent $e$ for some $e\mid m$.  
    
\begin{rem}[Welters] \label{R:welters1}
                                 In \cite[Prop.~1.1]{welters87}, Welters gives an equivalent
formulation of the condition \eqref{E:blochmurre}; namely, the data of
a symmetric endomorphism $\sigma$ satisfying \eqref{E:blochmurre} is
equivalent to the data of an abelian subscheme $\iota \colon Z\hookrightarrow
\operatorname{Pic}^0_{C/S}$ such that $\K(\iota^*\lambda_C)\subseteq
Z[m]$.  Again, in this situation $Z=\operatorname{Im}(1-\sigma)$.
Welters' proof is given over an algebraically closed field of
characteristic $0$, but the proof as written holds over any $S$ which
satisfies \Cref{alwaysepi}, if one replaces connected components of kernels with largest  abelian subschemes  and images of abelian varieties with images of abelian schemes.  
   Note that the condition that the data give a generalized Prym--Tyurin scheme is that   $\K(\iota^*\lambda_C)=Z[e]\subseteq
Z[m]$ for some $e$.  
\end{rem}

\subsubsection{Prym scheme}\label{S:Pscheme}
An \emph{embedded Prym scheme of exponent $e$} over $S$  is a tuple $$(Z,f\colon C\to C', \iota_Z)$$ such that
$Z$ is an abelian scheme over $S$,  $C$ is a smooth proper  curve over $S$, $
\iota_Z\colon Z\hookrightarrow \operatorname{Pic}^0_{C/S}$ is a closed subabelian $S$-scheme, 
 $f\colon C\to C'$ is a finite, thus flat (see \S
 \ref{SS:mapsbetweencurves}), $S$-morphism of smooth proper $S$-curves,  and 
setting $Y:= f^*\operatorname{Pic}^0_{C'/S}\subseteq
\operatorname{Pic}^0_{C/S}$ (Lemma \ref{L:f*picflat}), we have $e_Y=e$ and, identifying $Z$ with $\iota_Z(Z)$, we have that $Z$ is the complement of $Y$:
\begin{equation}
Z:= \operatorname{Im}([e]_X-N_Y).
\end{equation}
       
A \emph{Prym scheme of exponent $e$} over $S$ is an abelian scheme $Z/S$ such that there exists a tuple $(Z,f\colon C\to C',\iota_Z)$ that is an embedded Prym scheme of exponent $e$.

Let $f\colon C \to C'$ be a finite $S$-morphism of proper $S$-curves.  We will see in \Cref{L:f*picflat} that $Y:= f^*\operatorname{Pic}^0_{C'/S}$ is an abelian scheme.  Setting $e=e_Y$, we have the  associated Prym scheme of exponent $e$, which is defined to be the complement of $Y$:
\begin{equation}\label{E:P(C/C')}
P(C/C'):= \operatorname{Im}([e]_X-N_Y)=(\ker N_Y)^{\abvar}
= \left({\operatorname{Pic}^0_{C/S}}/{f^*\operatorname{Pic}^0_{C'/S}}\right)^{\widehat {\ }}.
\end{equation}
(The second equality follows form \Cref{L:Zisimage}, while the third is our definition of the Prym.)
In 
\Cref{C:PrymC/C'Nm} we give yet another equivalent formulation.  If we have  $\lambda_C|_{P(C/C')}=e\xi$ for some positive integer  $e$ and some principal polarization $\xi$ on $P(C/C')$, then we have also a  Prym--Tyurin scheme of exponent $e$
$$
(P(C/C'),\xi),
$$
which we  call the \emph{Prym--Tyurin Prym scheme of exponent $e$} associated to the cover $C/C'$.   In particular, $(P(C/C'), \xi, C, \iota)$, where $\iota\colon P(C/C')\hookrightarrow \operatorname{Pic}^0_{C/S}$ is the natural inclusion, is an embedded Prym--Tyurin scheme of exponent $e$.

\begin{rem}[Curves with automorphisms]\label{R:C/sigma}
Let $C/S$ be a smooth projective curve over $S$, which admits a finite \'etale group of $S$-automorphisms $G\subseteq \operatorname{Aut}_S(C)$.  The quotient $C':= C/G$ exists as an algebraic space over $S$, being the quotient of a scheme by an \'etale equivalence relation, and the geometric fibers, being quotients of smooth projective curves by finite groups, are smooth projective curves.  Provided the fibers have genus $g'\ge 2$, then $C'$ is a projective scheme over $S$, via the relatively ample relative canonical bundle.   In this case one has a  morphism $f\colon C\to C'$ as in the discussion above, and one can consider the associated Prym scheme $P(C/C')$.  
In the situation where $G$ is generated by a single automorphism
$\sigma$ of order $n$ (see \S\ref{S:ppavAut}), we note that
$Y=\operatorname{Im}(1+\sigma^*+(\sigma^*)^{2}+\cdots
+(\sigma^*)^{n-1})$.  Indeed, by descent,
$Y=f^*\operatorname{Pic}^0_{C'/S}$ parameterizes exactly those line
bundles invariant under the action of $G$.
 (Note that in this case, the endomorphism $\tau$ of \S\ref{S:ppavAut} is automatically epi-abelian.)
Consequently, in addition to the descriptions of $P(C/C')$ in \eqref{E:P(C/C')}, one has $P(C/C')=\operatorname{Im}((n-1)-\sigma^*-(\sigma^*)^2-\cdots- (\sigma^*)^{n-1})$ \eqref{E:1-sigma}.   
 In this case, since $\lambda_C$ is an indecomposable polarization, and $g'\ge 2$ implies that $Y\ne 0$ (see e.g., \Cref{L:pullpush} for a reminder),  one has that $2\le e\mid n$, and therefore, if $n$ is prime one has $e=n$ (see \S\ref{S:ppavAut}).
See also \Cref{C:PrymC/C'Nm} for another description of $P(C/C')$, and  \S\ref{S:ppavAut} for a discussion of the polarization type.  
\end{rem}

\section{Abel maps for curves}
\label{S:abelmap}

Let $C/S$ be a smooth proper curve over $S$,
    with canonically 
principally polarized Jacobian $(\operatorname{Pic}^0_{C/S},\lambda_C)$. 
Recall from \S \ref{S:rudiments} that
 there is an Albanese torsor of
$C/S$: 
$$\xymatrix{\alpha^{(1)}\colon  C \ar@{^(->}[r] &  \pic^{(1)}_{C/S}}$$ with $\pic^{(1)}_{C/S}$ a torsor under $\pic^0_{C/S}$, which is initial for $S$-morphisms from $C$ into torsors under abelian schemes over $S$.  If $C/S$ admits a section $P\colon S\to C$, 
then there is a pointed Albanese scheme
$$\xymatrix{\alpha_P\colon C\ar@{^(->}[r] &  \operatorname{Pic}^0_{C/S}}$$ which is initial for $S$-morphisms from $C$ into abelian schemes over $S$ taking $P$ to the zero section.  
 
             Recall  that there is a canonical isomorphism \eqref{E:PicTors}:
\begin{equation}\label{E:muC}
\xymatrix{
 \widehat{\pic^0_{C/S}}  \ar[r]^<>(0.5)\zeta_<>(0.5)\sim  &\pic^0_{{\pic^{(1)}_{C/S}}/S}.
}
\end{equation}

\subsection{Basic properties of the Abel map}

\begin{lem}
  \label{L:BLsec11.3}
  There is an equality of morphisms
    \[
      \xymatrix@C+2pc{\pic^0_{C/S}
\ar@<.5ex>[r]^{-\lambda_C} \ar@<-.5ex>[r]_{({\alpha^{(1)}}^*\zeta)\inv} & \widehat{\pic^{0}_{C/S}}.
}
    \]
  In particular, $ \xymatrix{ \widehat{\pic^0_{C/S}} \ar[r]^<>(0.5){{\alpha^{(1)}}^*\zeta} & \pic^0_{C/S}}$ is an isomorphism.  
                   \end{lem}

\begin{proof}
Using rigidity (e.g., \cite[Prop.~6.1]{GIT}), it suffices to check equality of the morphisms along a single fiber over $S$.  
  Having restricted to a fiber, it  suffices to verify the claims after passage to the algebraic
  closure, at which point the result is standard; see
  \cite[Lem.~III.6.9]{milneAV}.  (See also {\cite[Lem.~11.3.1,
  Prop.~11.3.5]{BL}} for the case $S = \spec \cx$.)
   \end{proof}

The universal property of the Albanese torsor means that if $\beta\colon C \to
T$ is an $S$-morphism to a torsor under an abelian scheme $Z/S$, then there is a uniquely determined $S$-morphism 
$\til \beta\colon \pic^{(1)}_{C/S} \to T$ which makes the following diagram commute:
\begin{equation}\label{E:til-beta}
  \xymatrix{
    C \ar[r]^\beta \ar[d]_{\alpha^{(1)}} & T \\
    \pic^{(1)}_{C/S}\ar@{-->}[ur]_{\til\beta}
  }
\end{equation}
Pullback of line bundles by $\beta$ induces an $S$-morphism
    $\pic_{T/S} \to \pic_{C/S}$ which
restricts to give the morphism of abelian schemes $\beta^*\colon \pic^0_{T/S} \to \pic^0_{C/S}$.  Moreover, pullback of line bundles by
$\til\beta$ induces a morphism
${\til\beta}^* \colon  \pic^0_{T/S} \to \pic^0_{{\pic^{(1)}_{C/S}}/S}$.
 Letting $\zeta_Z\colon \widehat Z\to \operatorname{Pic}^0_{T/S}$ be the canonical isomorphism \eqref{E:PicTors}, we  denote by $\widehat {\tilde \beta}$ the composition
\begin{equation}\label{E:til-b-hat}
\widehat {\tilde \beta}:=  \zeta^{-1} \tilde \beta^* \zeta_Z\colon  \widehat Z \longrightarrow \widehat{\operatorname{Pic}^0_{C/S}}.
\end{equation}

\begin{lem}
  \label{L:BLcor11.4.2}
  There is an equality of morphisms $$\widehat {\til\beta} =
  -\lambda_C\beta^*\zeta_Z;$$  
 i.e., the following diagram commutes:
  \begin{equation}\label{E:BLC1142}
    \xymatrix{
      \widehat Z \ar[r]^{\zeta_Z}  &\operatorname{Pic}^0_{T/S}\ar[r]^{\beta^*}\ar[d]_{{\til\beta}^*}& \pic^0_{C/S}      \ar[d]^{-\lambda_C} \\
     & \pic^0_{{\pic^{(1)}_{C/S}}/S}\ar[r]_{\sim}^{\zeta^{-1}}& \widehat{\pic^{0}_{C/S}} 
    }
 \end{equation}
\end{lem}

\begin{proof}
The commutativity of \eqref{E:BLC1142} follows immediately from \eqref{E:til-beta} and \Cref{L:BLsec11.3}. 
Indeed, using from  \Cref{L:BLsec11.3} that $-\lambda_C= ({\alpha^{(1)}}^* \zeta)^{-1}$, it is equivalent to show that ${\alpha^{(1)}}^* {\tilde \beta}^* \zeta_Z=\beta^*\zeta_Z$. As $\zeta_Z$ is an isomorphism, it suffices to show 
${\alpha^{(1)}}^* {\tilde \beta}^*=\beta^*$; but this follows from the
commutativity of \eqref{E:til-b-hat}.  (The case $S = \spec \cx$ is {\cite[Cor.~11.4.2]{BL}}.)
     \end{proof}

\subsection{Covers of curves}
\label{SS:mapsbetweencurves}
Let $C$ and $C'$ be smooth proper curves over  $S$, and
let $f\colon C \to C'$ be a finite $S$-morphism of degree $d$.  Then $f$ is a flat
$S$-morphism.  Indeed, since $S$ is by hypothesis locally Noetherian,
it suffices to verify that for each $s\in S$, $f_s: C_s \to C'_s$ is
flat (e.g., \cite[Cor.~14.27]{GW}); but this is immediate, since $f_s$
is a finite morphism of smooth curves (e.g.,\cite[Prop.~14.14]{GW}). 

On one hand,
pullback gives an $S$-morphism
\[
  \xymatrix{
    \pic^0_{C'/S} \ar[r]^{f^*} & \pic^0_{C/S}.
  }\]
On the other hand, the universal property of the Albanese torsor means that
there is a unique morphism $f_*^{(1)}$ that makes the following
diagram commute:
\begin{equation}\label{E:fCC'}
  \xymatrix{
    C \ar[d]^f \ar[r]^<>(0.5){\alpha^{(1)}} & \pic^{(1)}_{C/S} \ar@{-->}[d]^{f_*^{(1)}} \\
C' \ar[r]^<>(0.5){\alpha^{(1)}} & \pic^{(1)}_{C'/S}      
}
\end{equation}
and moreover, $f_*^{(1)}$ is a morphism of torsors over a homomorphism $$\xymatrix{f_*\colon 
\pic^0_{C/S} \ar[r]& \pic^0_{C'/S}.}$$

\begin{rem}
The homomorphism $f_*$ is denoted by $N_f$ in \cite[\S
11.4]{BL}, and agrees with 
 the norm map associated to $f$ (e.g., \cite[\S12.6]{GW}). 
 When $S=\operatorname{Spec}k$ for an algebraically closed field $k$, then on $k$-points $f_*$ is the restriction of the map $\pic_{C/S}(k)
\to \pic_{C'/S}(k)$ given by $\sum a_PP \mapsto \sum a_P f(P)$; i.e., 
the morphism induced by proper push forward on codimension-$1$ cycle classes. 
\end{rem}

The homomorphisms  $f_*$ and $f^*$ are essentially dual:
\begin{lem}
  \label{L:f-up-lw*}
     We have
  \[
    \widehat{f_*} \lambda_{C'} = \lambda_C f^*;
  \]
     i.e., the following diagram commutes
  $$
 \xymatrix{
 \operatorname{Pic}^0_{C'/S} \ar[d]_{f^*} \ar[r]^{\lambda_{C'}}_\sim&\widehat{ \operatorname{Pic}^0_{C'/S}}  \ar[d]^{\widehat{f_*}} \\
\operatorname{Pic}^0_{C/S} \ar[r]^{\lambda_C}_\sim& \widehat{ \operatorname{Pic}^0_{C/S}}.
 } 
 $$
 In particular, $\ker(\widehat{f_*})$ and $\ker(f^*)$ are isomorphic
 group schemes.  \end{lem}

\begin{proof} (If $S = \spec \cx$, this is {\cite[p.~331,
      eq(2)]{BL}}.)  Consider the diagram \eqref{E:fCC'} and apply
  $\operatorname{Pic}^0_{(-)}$ and the canonical isomorphism $\zeta$
  \eqref{E:muC}.  Then using \Cref{L:BLsec11.3}, this reduces to
  showing that pull back of line bundles along
  $f^{(1)}_*\colon \operatorname{Pic}^{(1)}_{C/S}\to
  \operatorname{Pic}^{(1)}_{C'/S}$ corresponds to pull back of line
  bundles along $f_*\colon \operatorname{Pic}^0_{C/S}\to
  \operatorname{Pic}^0_{C'/S}$. By rigidity it suffices to check
  equality of the morphisms along a geometric fiber;   over an
  algebraically closed field this is standard (e.g.,
  \cite[p.328]{mumford74}).          \end{proof}

We further detail the relation between $f_*$ and $f^*$.

\begin{lem}
  \label{L:pullpush}
  We have
  \[
    f_*f^* = [d]_{\pic^0_{C'/S}}.
  \]
\end{lem}

\begin{proof} Again, by rigidity we may check along a geometric fiber.  
  This is then standard, say viewing $f_*$ and $f^*$ as being induced by proper push forward and flat pull back, respectively,  in the Chow group;  i.e., one looks at the effect on divisors, remembering that $f$ is a
  finite flat morphism of degree $d$.
\end{proof}

We can track the effect on polarizations.

\begin{lem}
  \label{L:BL12.3.1}
  We have  $(f^*)^*(\lambda_C)=d\lambda_{C'}$. 
  \end{lem}

\begin{proof}[Proof]
 As in \cite[Lem.~12.3.1]{BL}, observe that 
$$
d\lambda_{C'}=\lambda_{C'}\circ  [d]_{\operatorname{Pic}^0_{C'/K}}= \lambda_{C'}\circ f_*\circ f^* = \widehat {f^*}\circ \lambda_C\circ f^* =: (f^*)^*(\lambda_C)
$$
where the second equality is \Cref{L:pullpush}, and the third equality comes from applying $\widehat{(-)}$ to \Cref{L:f-up-lw*}. 
\end{proof}

In \Cref{L:newBLprop11.4.3} below, we calculate $\ker(f^*)$ when $f$ is a finite morphism of smooth projective curves over a field, and show that there is an
abelian cover $C^\abelian \to C'$ which is initial among all abelian
covers of $C'$ through which $f$ factors.
Under a separability hypothesis, we can use this to show that even over an arbitrary base scheme $S$, the image of $f^*$ is an abelian scheme.  (In all applications known to the authors, even if the inseparable degree of $f_s$ is nonconstant -- as in, for example, the case of the multiplication-by-$p$ morphism for a family of elliptic curves with both supersingular and ordinary fibers -- it is clear that $\ker(f^*)$ is still a finite flat group scheme, and thus the image of $\pic^0_{C'/S}$ is an abelian scheme.)

\begin{lem}
  \label{L:f*picflat}
  Let $f\colon C \to C'$ be a finite morphism of smooth projective curves over $S$ such that either $f$ is fiberwise separable or \Cref{alwaysepi} holds.  Then $\ker(f^*\colon \pic^0_{C'/S} \to \pic^0_{C/S})$ is a finite flat group scheme, and so the image $f^*\pic^0_{C'/S}$ is an abelian subscheme of $\pic^0_{C/S}$.
\end{lem}

\begin{proof}
  Since the claim is tautologically true if \Cref{alwaysepi} holds, we assume $f$ fiberwise separable and proceed.  By \Cref{L:maxabcover}, $f$ factors as $C \stackrel{j}\to C^\abelian \stackrel{h}\to C'$ where 
  $C^\abelian$ is a smooth proper curves over $S$; $j$ and $h$ are finite; $h$ is an abelian \'etale 
  morphism; and for each $s\in S$, $(C^\abelian)_s$ is initial among all
  abelian \'etale covers through which $C_s \to C'_s$ factors.

By \Cref{L:newBLprop11.4.3}(d),
  $\ker(j^*)$ is fiberwise trivial, and thus trivial, while
  \Cref{L:newBLprop11.4.3}(a) shows that $\ker(h^*)$ is fiberwise the
  covering group of the morphism $C^\abelian \to C'$.
\end{proof}

In the previous argument, we used the fact that
the formation of $C^\abelian$ makes sense in families:

\begin{lem}
  \label{L:maxabcover}
  Let $f\colon C \to C'$ be a finite fiberwise separable morphism of
  smooth proper curves over $S$.  Then $f$ factors through an abelian
  \'etale cover $C^\abelian \to C'$ such that, for each $s\in S$,
  $(C^\abelian)_s \iso (C_s)^\abelian$.\footnote{If the fibers of $C^\abelian$ have genus $1$, then it is possible that $C^\abelian$ is not a scheme, but merely an algebraic space.  All our techniques and results still apply, and we will elide this distinction.}
\end{lem}

\begin{proof}  
  By fppf descent,   we may and do assume that $S$ is
  the spectrum of an Artinian local ring, with closed point $0$.
  Recall that, since the fundamental group is insensitive to nilpotents, if
  $D/S$ is a smooth proper curve, then an
  \'etale cover of $D_0$ extends uniquely to an \'etale cover of $D$.

  It now suffices to show that if there is an abelian \'etale cover
  $g_0:C^a_0 \to C'_0$ through which $f_0$ factors, and if $g: C^a
  \to C'$ is the unique extension of $C^a_0$ to an \'etale cover of
  $C'$, then $f$ factors through $g$.

  To see this, let $D_0 = C^a_0 \times_{C'_0} C_0$ be the fiber
  product
  \[
  \xymatrix{
    D_0 \ar[r] \ar[d] & C_0 \ar[d]^{f_0} \\ C^a_0 \ar[r]^{g_0} &C'_0.
  }\]
  It is an \'etale cover of $C_0$.

  On one hand, let $D$ be the extension of $D_0$ as an \'etale cover
  of $C$.  On the other hand, $C^a\times_{C'}C$ is also an \'etale
  cover of $C$.  Since their special fibers coincide, these covers of
  $C$ are isomorphic;
  \[
  D \iso C^a \times_{C'}C.
  \]
  However, we can analyze $D$ in greater detail.  Let $G$ be the
  covering group of $g_0$, and let $r =
  \deg(g_0) = \abs G$.  Like any torsor, the $G$-torsor $C^a_0 \to C'_0$ trivializes
  itself, so that $C^a_0\times_{C'_0}C^a_0 \iso \sqcup_{\gamma\in G} C_0^a$,
  and so
  \[
    D_0 = C^a_0 \times_{C'_0} C_0 \, \iso \,
    C^a_0 \times_{C'_0} C^a_0 \times_{C^a_0}  C_0
    \, \iso \, \sqcup_{\gamma \in G}C^a_0 \times_{C^a_0} C_0
    \, \iso \, \sqcup_{\gamma\in G} C_0.
    \]
                      Moreover $D$,
  the unique extension of $D_0$ to a cover of $C$, is visibly the disjoint
  union of $r$ copies of $C$.  We thus find that $C^a \times_{C'}C$ is
  isomorphic to a disjoint union of $r$ copies of $C$, and thus $f$
  factors through $g$.
\end{proof}

In Lemma \ref{L:pullpush}, we computed $f_*f^*$.  In a complementary fashion, we have

\begin{lem}
  \label{L:pushpull} Let $Y:= f^*\operatorname{Pic}^0_{C'/S}\subseteq
  \operatorname{Pic}^0_{C/S}$.  Then  $e_Y|d$ and 
  \[
    f^*f_* = \frac{d}{e_Y} N_Y.\]
   \end{lem}

\begin{proof}
 The proof of the stated equality in \cite[Prop.~12.3.2]{BL} holds in this setting, as well.  Indeed, if we factor $f^*$ as an isogeny $j$ and an inclusion $\iota_Y$, then using \Cref{L:BL12.3.1} 
  we obtain the diagram:
$$
\xymatrix{
\operatorname{Pic}^0_{C'/S} \ar[r]^<>(0.5)j \ar[d]_{d \lambda_{C'}} \ar@/^1.5pc/[rr]^{f^*}& Y \ar@{^(->}[r]^<>(0.5){\iota_Y} \ar[d]_{\iota_Y^*\lambda_C}& \operatorname{Pic}^0_{C/S} \ar[d]^{\lambda_C}\\
\widehat{\operatorname{Pic}^0_{C'/S}}& \widehat Y \ar[l]_<>(0.5){\widehat j}& \widehat{\operatorname{Pic}^0_{C/S}} \ar[l]_<>(0.5){\widehat{\iota_Y}}  \ar@/^1.5pc/[ll]^{\widehat{f^*}=\lambda_{C'} f_*\lambda_C^{-1}}\\
}
$$
where the identification of $\widehat{f^*}$ comes from \Cref{L:f-up-lw*}. 
Therefore, we have
$$
f^*\circ f_* =   f^* \circ  \lambda_{C'}^{-1} \circ  (\lambda_{C'} \circ  f_*\circ  \lambda_{C}^{-1})\circ \lambda_{C}= (\iota_Y\circ  j) \circ \lambda_{C'}^{-1}\circ  (\widehat j\circ  \widehat {\iota_Y})\circ  \lambda_{C}.
$$
Now from the left hand square we have  in $\operatorname{End}_{\mathbb Q}(Y)$ that 
$
\iota_Y^*\lambda_C= \widehat j^{-1}\circ  d \lambda_{C'}\circ  j^{-1}
$
so that $\lambda_{C'}^{-1}=(\frac{1}{d} \widehat j\circ  (\iota_Y^*\lambda_C)\circ  j)^{-1}= d j^{-1}\circ  (\iota_Y^*\lambda_C)^{-1}\circ \widehat j^{-1}$, and we have in $\operatorname{End}_{\mathbb Q}(\operatorname{Pic}^0_{C/S})$ that   
$$
f^*\circ f_* = d \iota_Y \circ (\iota_Y^*\lambda_C)^{-1} \circ \widehat{\iota_Y}\circ \lambda_C = d \frac{1}{e_Y} N_Y.
$$
We conclude that $e_Y|d$ as $N_Y$ is primitive (\Cref{BL:C534}).
    \end{proof}

\begin{cor}\label{C:PrymC/C'Nm}  We have
$$
P(C/C')= (\ker f_*)^{\abvar}.
$$
 \end{cor}

\begin{proof}
From the definition \eqref{E:P(C/C')} and \Cref{L:pushpull} we have 
$$
P(C/C')= (\ker N_Y)^{\abvar}= (\ker \frac{d}{e_Y}N_Y)^{\abvar}=(\ker (f^*f_*))^{\abvar}
$$
Clearly we have $(\ker f_*)^{\abvar}\subseteq (\ker (f^*f_*))^{\abvar}$.  From \Cref{L:pullpush} we have that $\ker f^*$ is finite over $S$, being a subscheme of the finite group scheme $\operatorname{Pic}^0_{C'/S}[d]$. Therefore, the inclusion  $(\ker f_*)^{\abvar}\subseteq (\ker (f^*f_*))^{\abvar}$ is an equality, and we are done.  
\end{proof}

\begin{rem}\label{R:PrymC/C'Nm}
If $S=\operatorname{Spec}K$ is the spectrum of a field $K$, we have $P(C/C')=((\ker f_*)^\circ)_{\operatorname{red}}$, and if moreover, $\operatorname{char}(K)=0$, then $P(C/C')=(\ker f_*)^\circ$ (\Cref{R:LAS}).  
\end{rem}

\section{Cycles and endomorphisms over fields}
\label{S:cyclefield}

The main goal of this section is to review the construction and basic
properties of the endomorphism $\delta(\alpha,\beta)$ of an abelian
variety $X$ over a field $K$ induced by cycle classes $\alpha$ and
$\beta$ of complementary dimension on $X$.  Since the construction
uses intersection theory, we prefer to work over a field $K$; one can
do intersection theory over slightly more general bases, but for our
purposes  
 we get more general results
working over a field and then spreading, and so it is more convenient
to simply do the constructions over a field.

The endomorphism $\delta(\alpha,\beta)$ was constructed
\cite{matsusaka59}, and described and analyzed in the language of
Weil's foundations for algebraic geometry. Thus, it was originally
discovered as an object over an arbitrary, ``sufficiently large''
field.  It is studied in \cite{BL} in the special case of
complex abelian varieties.  In the present treatment, we rely on a
modern reformulation of Samuel's notion of a regular homomorphism for
the initial construction of $\delta(\alpha,\beta)$.  Once the
existence of this endomorphism has been secured, it turns out that the
analysis of its properties in \cite[\S 11.6]{BL} is equally valid in the
algebraic category.

\subsection{Equivalences of algebraic cycles on abelian varieties}
\label{SS:EqRel}
We recall here the agreement of certain notions of equivalence between cycles on an abelian variety $X$ over an algebraically closed field $k$. We refer the reader to \cite[\S1.2 and App.~A]{MNPmot} for references.

To fix terminology and notation, we fix a cohomology theory $\mathcal H^\bullet$ with coefficient ring $R_{\mathcal H}$.  This can be any Weil cohomology theory, but we will also work with 
several cohomology theories with ``integral'' coefficients, which give rise to standard Weil cohomology theories.  For instance, for $\ell\ne \operatorname{char}(k)$,  we will consider $\ell$-adic cohomology with $\mathbb Z_\ell$-coefficients.  In positive characteristic, we will also consider crystalline cohomology with coefficients in the Witt ring $W(k)$, and for $k=\mathbb C$, we also consider Betti cohomology with $\mathbb Z$-coefficients.  Note that tensoring with the field of fractions of the coefficient ring of any of these cohomology theories gives rise to a Weil cohomology theory.  
Unless a particular cohomology theory is specified, if we  say a cycle or cycle class is homologically trivial, 
we mean its class in $\mathcal H^\bullet$ is trivial for all of the applicable cohomology theories.

  Now, for a divisor $D$ on $X$, we have that
  \begin{itemize}
\item    \emph{$D$ is algebraically trivial $\iff$  $D$ is homologically trivial $\iff$ $D$ is numerically trivial.}
\end{itemize}

  Indeed, on any smooth projective variety over an algebraically
  closed field, for cycles of any dimension one has the forward
  implications.   Moreover, for divisors, Matsusaka's
  theorem (e.g.,  \cite[App.~A, Thm., p.13]{MNPmot}) 
         states that numerical triviality of a divisor implies that some multiple is homologically trivial, and homological triviality of a divisor implies that some multiple is algebraically trivial.   On an abelian variety, however, the cohomology is free of finite rank over the ring of coefficients, and the N\'eron--Severi group is a free $\mathbb Z$-module of finite rank, and so numerical (resp.~homological) triviality of $D$ implies homological (resp.~algebraic) triviality of $D$.  We also recall that $D$ is algebraically trivial if and only if $\mathcal O_{X}(D)\in \operatorname{Pic}^0_{X/k}(k)$.  
  
  Similarly, for a $1$-cycle $\Gamma$ on $X$, we have 
    \begin{itemize}
\item    \emph{$\Gamma$ is algebraically trivial $\implies $ $\Gamma$ is 
  homologically trivial $\iff$  $\Gamma$ is numerically
  trivial.}
\end{itemize}

  As mentioned above, the forward implication holds for cycles of any dimension on any
  smooth projective variety over an algebraically closed field.  The
  reverse implication follows from \cite[Cor.~10]{sebastianSmash},
  which states that any numerically trivial $1$-cycle on a variety
  dominated by a product of curves has a positive integer multiple
  that is homologically trivial.  
            Since every abelian variety is
  dominated by the Jacobian of a smooth projective curve (see the
  proof of \Cref{C:PPAV=PT} for a reminder), and the
  Jacobian of a smooth projective curve of genus $g$ is dominated by the product of
  $g$ copies of the curve,
  Sebastian's result applies to abelian varieties.  We thus conclude
  that some positive multiple of $\Gamma$ is homologically
  trivial. However, again, since the cohomology of an abelian variety
  is free of finite rank over the ring of coefficients, this implies
  that $\Gamma$ itself is homologically trivial.

\subsection{Regular homomorphisms and traces of endomorphisms}
\label{S:regular}
Let $Z/K$ be a smooth geometrically irreducible projective variety, and let $A/K$ be an abelian variety. 
Denote by  $\A^i(Z) \subseteq \chow^i(Z)$ denote the group of algebraically
trivial cycle classes.
A (Galois-equivariant) regular homomorphism (in codimension $i$) is an $\aut(\bar K/K)$-equivariant group homomorphism
\[
  \xymatrix{\A^i(Z_{\bar K}) \ar[r]^\phi & A(\bar K)}
\]
such that, for each pair $((T,t_0),Z)$ with $(T,t_0)$ a pointed smooth variety over $K$, and each $\Xi \in \chow^i(T\times Z)$, the map of points
\[
  \xymatrix@R=.5em{
    T(\bar K) \ar[r] & A(\bar K)\\
    t \ar@{|->}[r] & \Phi(\Xi_t-\Xi_{t_0})
  }
\]
is induced by a $K$-morphism of varieties
\[
  \xymatrix{
    T \ar[r]^{\psi_\Xi} & A.
  }
\]
In particular, if $(Z,z_0)$ is a pointed $K$-variety of dimension $d$,
then the Albanese map is a universal regular homomorphism in
codimension $d$. We refer the reader to \cite{ACMVfunctorial} for more details on regular homomorphisms.

\medskip 
We now turn our attention to recalling Weil's trace formula \eqref{E:WeilTr}.
For this we further specialize the discussion above to the case where $Z=C$ is a pointed smooth projective curve over $K$, and recall the
Albanese property of the Picard variety (\S \ref{S:AlbCurve}).  In this situation, we recover the
relation between correspondences on curves and endomorphisms of the
Jacobian which Weil constructed (\cite{weil-vaca}; see  \cite[\S 3.2.2]{kahnlectures} and  \cite[\S 3]{schollclassical} for modern treatments): there is a canonical surjection 
 \[
  \xymatrix@R=.5em{
    \chow^1(C\times_K C) \ar@{->>}[r] & \operatorname{End}(\pic^0_{C/K})
    \\
    D \ar@{|->}[r] & \gamma_D,
  }
\]
with kernel 
 the subgroup of divisors generated by horizontal and
vertical fibers of the product $C\times_K C$.  Below, we will need to
use Weil's formulation of the Lefschetz formula for
self-correspondences on $C$; it is expressed in terms of the trace of
the corresponding endomorphism on the Jacobian.

Suppose $X$ is a $g$-dimensional abelian variety and $f \in \End(X)$.  Recall that
there is a unique monic polynomial $P_f\in \integ[T]$ of degree
$2g$, the characteristic polynomial of $f$, such that for each
integer $r$ we have $P_f(r) = \deg(f - [r]_X)$.  If $X$ is a complex
abelian variety, this is the characteristic polynomial of the action
of $f$ on $H_1(X,\rat)$; more generally, $P_f(T)$ is the
characteristic polynomial of the linear operator induced by $f$ on
$T_\ell X$ for any $\ell\not = \operatorname{char}(K)$.  The trace of
$f$, $\tr(f)$, can be read off from $P_f(T)$; indeed, 
$P_f(T) = T^{2g} - \tr(f) T^{2g-1} + \cdots$.

For $D \in \chow^1(C\times_KC)$, Weil was able to calculate the trace of the associated endomorphism using intersection theory on $C\times_K C$.
More precisely for  $D \in \chow^1(C\times_K C)$, if one defines the indices $d_1(D)$ and
$d_2(D)$ by $\operatorname{pr}_i(D) = d_i(D)[C]$
(e.g., \cite[Ex.~16.1.4]{fulton}, \cite[Def.~2.19]{kahnlectures}),  
then (e.g., \cite[Ex.~6.45]{kahnlectures} \cite[Thm.~VI.3.6]{langabvars} \cite[\S 2.II]{weil48})
\begin{equation}\label{E:WeilTr}
   \tr(\gamma_D) = d_1(D) + d_2(D) - (\Delta_C \cdot D),
\end{equation}
where $\Delta_C \in \chow^1(C\times_K C)$ is the class of the diagonal.

\subsection{Definition of $\delta$ via intersection theory and regular homomorphisms} \label{S:Def-delta}
Let $X/K$ be an abelian variety over a field $K$.  Let $\alpha\in \operatorname{CH}^a(X)$ and $\beta\in \operatorname{CH}^b(X)$ by cycle classes of complementary (co)dimension,  so that $a+b=\dim X$.  
On the product $X\times_K X$, with projections $\pi_1,\pi_2$, and addition map $\mu:X\times_K X \to X$, we can consider $\pi_2^*\alpha\in \operatorname{CH}^{a}(X\times_K X)$, and $(\operatorname{Id}_X\times \mu)_*\pi_2^*\beta \in \operatorname{CH}^{b}(X\times_K X)$, and the intersection product
\begin{equation}\label{E:G(a,b)}
\Gamma(\alpha,\beta):= (\pi_2^*\alpha)\cdot  ((\operatorname{Id}_X\times \mu)_*\pi_2^*\beta)\in \operatorname{CH}^{\dim X}(X\times_K X).
\end{equation}
Since for an abelian variety the identity map is an Albanese
map,
 the fact that the Albanese is a regular homomorphism implies that the  cycle class $\Gamma(\alpha,\beta)$ induces  a morphism of varieties over $K$:
\[
  \xymatrix{
    \hat \delta(\alpha,\beta)\colon X \ar[r] &X.
  }\]
We adjust $\hat\delta$ to construct a $K$-\emph{homomorphism}, i.e., a morphism of abelian varieties:
$$
\xymatrix{ \delta(\alpha,\beta)=\hat\delta(\alpha,\beta)-\hat \delta(\alpha,\beta)(0):X \ar[r]& X.}$$  If  $V, W\in Z^\bullet (X)$
are cycles with cycle classes $[V]$, $[W]$, then we define $\delta(V,W)=\delta([V],[W])$.  
   
\begin{rem}\label{R:dbarK}
As regular homomorphisms and intersection products are stable under base change of field \cite[Thm.~1]{ACMVfunctorial}, so is $\delta$.  In particular,  $\delta(\alpha_{\bar K},\beta_{\bar K}) = (\delta(\alpha,\beta))_{\bar K}$.  
\end{rem}

From the definition it is clear that $\delta(-,-)$ is additive in each entry.  We also have: 
  \begin{pro}\label{P:del=0}
 In the notation above, suppose that  $\alpha_{\bar K}$ or $\beta_{\bar K}$ is one of the following:
 \begin{enumerate}
 \item A torsion cycle class,
 \item A homologically trivial cycle class,
 \item A numerically trivial $1$-cycle class or divisor class.
\end{enumerate} 
Then $\delta(\alpha,\beta)=0$.   
   \end{pro}

  \begin{proof} (See {\cite[Prop.~5.4.3]{BL}} for the case $S = \spec
    \cx$.)  
  A morphism of abelian varieties is zero if and only if the morphism is zero after base change to the algebraic closure; thus from \Cref{R:dbarK}, we may assume we are working over an algebraically closed field.  

 (1) If for some positive integer $n$ we have $\delta(n\alpha,\beta)=0$ or $\delta(\alpha,n\beta)=0$, then $\delta(\alpha,\beta)=0$,  due to the additivity of $\delta(-,-)$ in each entry, together with the fact that $\operatorname{End}(X)$ is a free $\mathbb Z$-module of finite rank (e.g., \S  \ref{S:isog}).  
 
    (2) Note first that if $\alpha$ is algebraically trivial, the assertion follows from the definitions and rigidity of endomorphisms of abelian varieties (e.g., \cite[Prop.~6.1]{GIT}).  For the case where $\alpha$ is only assumed to be homologically trivial, we use a recent result \cite[Prop.~8.1]{ACMVdiagonal}, which shows that for fields of any characteristic, regular homomorphisms depend only on the cohomology class of the correspondence inducing the regular homomorphism (this is well-known over $\mathbb C$).  Then, if $\alpha$ is homologically trivial, one has that $\pi_2^*\alpha$ is homologically trivial, and, since intersection product in Chow is taken to cup product in cohomology, one has from the definition of $\Gamma(\alpha,\beta)$ in \eqref{E:G(a,b)} that $\Gamma(\alpha,\beta)$ is homologically trivial. Therefore, \cite[Prop.~8.1]{ACMVdiagonal} implies that $\hat\delta(\alpha,\beta)$ is the zero map, and consequently,  $\delta(\alpha,\beta)=0$, as well.  The same argument works for $\beta$, as well.  
  
(3) This follows from (2) since numerically trivial divisors and  $1$-cycles on abelian varieties are homologically trivial (\S \ref{SS:EqRel}).  
  \end{proof}

\subsection{Geometric construction of $\delta$ over an algebraically closed field}
\label{S:deltapoint} Over an algebraically closed field we can give a geometric description of $\delta$.  
Let $X$ be an abelian variety over an algebraically closed field $k$.  Suppose $V, W\subseteq X$ are integral  cycles of complementary dimension, so that $\dim V + \dim W = \dim
X$.  
    
If the intersection $V\cap W$ is nonempty of pure dimension $0$, and the intersection is  $V.W=\sum _{i=1}^n r_ix_i$, then we define
\begin{equation}\label{E:S(V,W)}
S(V,W)=
r_1x_1+\cdots +r_nx_n
\end{equation}
where the sum above is addition in the abelian variety $X$.   Note that $S(-,-)$ is clearly additive in each entry (under the hypotheses on the intersections of the cycles), and symmetric.  

For $P\in X(k)$, let $t_P\colon X \to X$
   be the corresponding translation
map.
 Recall from the Moving Lemma (\cite[Thm.~2, Lem.~1]{kleiman74}) that there is a nonempty open subset $U\subseteq X$ such that, if $P\in X$,
then $V\cap t_P^*W$ is empty or of pure of dimension zero.
If for a general $P\in U$ the intersection is of pure dimension $0$, then we define for $P\in U$
$$
\hat \delta(V,W)(P)=S(V,t_P^*W)
$$
giving a morphism $\hat \delta(V,W)\colon U\to X$, which extends to give a morphism $\hat \delta(V,W)\colon X\to X$.  We define $\delta(V,W)=\hat\delta(V,W)-\hat \delta(V,W)(0)$.  
Otherwise, if for general $P\in U$ the intersection $V\cap t_P^*W$ is empty, we set $\delta(V,W)=0$.    Clearly $\hat \delta$ and $\delta$ agree with the definition in the previous section.  For clarity we note that when $V\cap t_P^*W$ and $V\cap W$ are of pure dimension $0$ we have 
\begin{equation}\label{E:d(V,W)S}
\delta(V,W)(P)= S(V,t_P^*W)-S(V,W).
\end{equation}

  \subsection{Formal properties of $\delta$ 
   }\label{S:Frml-d}
    
Here we review some of the formal properties of $\delta$ established
in \cite[\S 1]{matsusaka59}, following the treatment in \cite[\S5.4]{BL}.

\begin{pro}\label{P:BL544-6-7}
 In the notation above: 
\begin{alphabetize}
\item \label{L:BL544}
 For algebraic cycles $V,W$ on $X$ with $\dim V+\dim W=\dim X$,
 we have 
  \[
  \delta(V,W) + \delta(W,V) = -(V \cdot W)[1]_X.
  \]
  
  \item \label{L:BL546}
   For algebraic cycles $V_0, \dots, V_r$ on $X$ with $\sum
  \dim V_i = r\dim X$, we have
  \[
  \delta(V_0, V_1\cdot \ldots \cdot V_r) = \sum_{i=1}^r \delta(V_0 \cdot V_1
  \cdot \ldots \cdot \breve{V}_i \cdot \ldots  \cdot V_r, V_i)
  .\]
  
  \item  \label{P:BL547} 
   \[
  \delta(D^r,D^{g-r}) = -\frac{g-r}{g}(D^g)[1]_X.
  \]
\end{alphabetize}
\end{pro}

\begin{proof}
(a)  It suffices to prove the equality after base change to the algebraic closure.  Then this is a straight forward argument using \eqref{E:d(V,W)S} and \eqref{E:S(V,W)} (see \cite[Lem.~5.4.4]{BL}).

(b) Again, it suffices to prove the equality after base change to the algebraic closure.  Then this is a straight forward argument using \eqref{E:d(V,W)S} and \eqref{E:S(V,W)} (see \cite[Lem.~5.4.6]{BL}).

(c) This follows formally from (a) and (b); see \cite[Prop.~5.4.7]{BL}.
\end{proof}

\subsection{Interaction of $\delta$ with the Abel map}
\label{S:abelanddelta}
Here we review some of the results from \cite[\S11.6]{BL}.
 We start with a  $K$-morphism $\beta: C\to Z$ from a smooth projective genus $g$ curve $C$ over $K$ to an abelian variety $Z$ over $K$,   and we will use the notation from \eqref{E:til-beta} and \eqref{E:til-b-hat}.

\begin{pro}
  \label{P:BL1161}
For  a divisor $D$ on $Z$,
     \begin{equation}
    \label{E:deltabetaL}
  \delta(\beta_*[C],D)  =\widehat{\widehat{\til\beta}} \circ \lambda_C\inv \circ  \widehat{\til\beta} \circ 
  \phi_{D},
\end{equation}
\[
  \xymatrix{
    Z \ar[r]^{\phi_D} & \widehat Z \ar[r]^<>(0.5){\widehat{\til\beta}} &
    \widehat{\pic^0_{C/K}} \ar[r]^{\lambda_C\inv} & \pic^0_{C/K}
    \ar[r]^{\widehat{\widehat{\til\beta}}} & Z  }\]
               \end{pro}

\begin{proof} It suffices to check equality of the morphisms over the algebraic closure.  In this case, using the definition of $\delta$ given in \S \ref{S:deltapoint}, the proof follows formally as in \cite[Prop.~11.6.1]{BL}.
   \end{proof}

\begin{pro}
\  \label{P:BL1162}
  With notation as above,
  \begin{alphabetize}
  \item $\tr(\delta(\beta_*[C],D)) = -2(\beta_*[C]\cdot D)$
  \item $\tr(\delta(D,\beta_*[C]) = -(2g-2)(\beta_*[C]\cdot D)$
  \end{alphabetize}
\end{pro}

\begin{proof}
  As usual, it suffices to prove this over an algebraically closed
  field.  The proof in \cite[Prop.~11.6.2]{BL} proceeds without
  incident, provided one uses the intrinsic trace reviewed in 
  \S \ref{S:regular}, and then invokes Weil's trace formula \eqref{E:WeilTr}  in the analysis of a certain self-correspondence
  on $C$. 
        \end{proof}

\begin{teo}[{\cite[\S 2]{matsusaka59}; see also \cite[Thm.~11.6.4]{BL}}]\label{T:BL1164a} 
Let $Z$ be an abelian variety over a field $K$, let $D$ be a  divisor on $Z$,  and let  $\Gamma$ be an algebraic $1$-cycle on $Z$. 
  \begin{alphabetize}
\item \label{T:BL1164aa}  If $D$ is non-degenerate (i.e., $\phi_D$ is an isogeny), then   $\delta(\Gamma,D) = 0$ if and only if  $\Gamma_{\bar K}$ is numerically equivalent to
  zero.

\item \label{T:BL1164ab}  If $\Gamma$ is an integral curve  and non-degenerate (i.e., generates $Z$), then  $\delta(\Gamma,D) = 0$ if and only if   $D_{\bar K}$ is algebraically equivalent to zero.
   \end{alphabetize}

    \end{teo}

\begin{proof}
(a) We have already shown one direction (\Cref{P:del=0}), so assume that $\delta(\Gamma,D)=0$.
 It suffices to show the result after base change to the algebraic closure (see \Cref{R:dbarK}).  
  The argument of \cite[Thm.~11.6.4(a)]{BL}, which relies only on
  \Cref{P:BL1161} and \Cref{P:BL1162}, is valid in the algebraic
  setting, too.  (In that argument, the curve $C$ ranges among the
  desingularizations of the various components of $\Gamma$.)
  
(b)  
We have already shown one direction (\Cref{P:del=0}), so assume that $\delta(\Gamma,D)=0$.  Again,  it suffices to show the result after base change to the algebraic closure. The argument of  \cite[Thm.~11.6.4(b)]{BL} holds here, as well.
\end{proof}

\section{Prym varieties over fields}
\label{S:prymfield}

In this section we work over a field $K$.  We show that every principally polarized abelian variety is a Prym--Tyurin variety of some exponent, and we  classify 
  Prym--Tyurin Prym varieties of all exponents.

\subsection{Prym--Tyurin varieties and Welters' criterion}
\label{SS:ptfield}

Recall from \S \ref{S:PTsch} that a Prym--Tyurin variety of exponent $e$ over a field $K$ 
 is a principally polarized abelian variety $(Z,\xi)/K$ together with a smooth projective curve $C/K$ and an injective $K$-homomorphism  $\iota_Z\colon Z\hookrightarrow \operatorname{Pic}^0_{C/K}$ of abelian varieties such that $\iota_Z^*\lambda_C=e\xi$, where $\lambda_C$ is the canonical principal polarization on $\operatorname{Pic}^0_{C/K}$.   
 The condition $\iota_Z^*\lambda_C=e\xi$ is equivalent to the condition that $\K(\iota_Z^*\lambda_C) = Z[e] :=  \ker [e]_Z$ (\Cref{R:Z[e]-PT}); if there exist divisors $\Theta$ and $\Xi$ such that $\phi_\Theta=\lambda_C$ and $\phi_\Xi= \xi$, e.g., if $K=\bar K$, then this condition is also equivalent to the condition
$$
\iota_Z^*\Theta \equiv e\Xi.
$$
The equivalence above is taken to be homological or numerical;   these notions all agree for divisors on abelian varieties, as they are equivalent to homological or numerical equivalence over the algebraic closure (\S \ref{SS:EqRel}).

The main tool in classifying Prym--Tyurin varieties is the following:

\begin{lem} \label{L:BL-12.2.3}
  Let $\beta\colon C \to T$ be a $K$-morphism from a smooth projective curve 
  to a torsor under a principally polarized abelian variety $(Z,\xi)/K$ of dimension $g_Z$, and let $\zeta_Z\colon \widehat Z\stackrel{\sim}{\to}\operatorname{Pic}^0_{T/K}$ be the canonical isomorphism \eqref{E:PicTors}. 
Fix an identification of $T_{\bar K}$ with $Z_{\bar K}$, and fix 
 divisors $\Theta_{\bar K}$ and $\Xi_{\bar K}$ such that $\phi_{\Theta_{\bar K}}=(\lambda_C)_{\bar K}$ and $\phi_{\Xi_{\bar K}}= \xi_{\bar K}$.

   Then the
  following are equivalent:
  
  \begin{alphabetize}
  \item \label{L:BL-12.2.3a} $(\beta^*\circ \zeta_Z \circ \xi)^*\lambda_C= e\xi$.
  
  \item \label{L:BL-12.2.3b} The following diagram commutes:
 \begin{equation}\label{E:BL-12.2.3}
      \xymatrix@R=1.5em{
Z \ar[r]^{e\xi} \ar[d]_{\xi} & \widehat Z\\
\widehat Z \ar[r]^{(\beta^*\zeta_Z)^*\lambda_C}  \ar[d]_{\beta^*\zeta_Z}& Z \ar[u]_{\xi=\widehat \xi} \\
\operatorname{Pic}^0_{C/K} \ar[r]_{\lambda_C}& \widehat {\operatorname{Pic}^0_{C/K}} \ar[u]_{\widehat{\beta^* \zeta_Z}}
}
\end{equation}

\item \label{L:BL-12.2.3c} $\delta(\beta_{\bar K*}[C_{\bar K}],\Xi_{\bar K}) = -[e]_{Z_{\bar K}}$.

\item  \label{L:BL-12.2.3d} $\beta_{\bar K*}[C_{\bar K}] \equiv e\frac{\displaystyle [\Xi_{\bar K}]^{g_Z-1}}{(g_Z-1)!} $, where equivalence is taken to be homological or numerical; these notions agree for $1$-cycles on abelian varieties over algebraically closed fields (\S \ref{SS:EqRel}).
\end{alphabetize}
\end{lem}

\begin{proof}  (See also {\cite[Lem.~12.2.3]{BL}}.)

(a)$\iff$(b). This is clear from the definition of a restriction of a polarization.

(b)$\iff$(c).  The diagam \eqref{E:BL-12.2.3} commutes if and only if it commutes over $\bar K$.  Therefore, we will work over $\bar K$.  Recalling that we have identified $T_{\bar K}$ with $Z_{\bar K}$, so that  $\beta_{\bar K}\colon  C_{\bar K}\to Z_{\bar K}$, and $(\zeta_{Z})_{\bar K}$ is the identity, 
we always  have:
\begin{align*}
 \delta(\beta_{\bar K*}[C_{\bar K}],\Xi_{\bar K}) &=  \widehat{\widehat{\til\beta}}_{\bar K}\lambda_{C_{\bar K}}\inv
    \widehat{\til\beta}_{\bar K}\xi_{\bar K} &\text{(\Cref{P:BL1161})} \\
  & =  - (\widehat {\beta_{\bar K}^*} \lambda_{C_{\bar K}})\lambda_{C_{\bar K}}\inv (\lambda_{C_{\bar K}} \beta_{\bar K}^*)
    \xi_{\bar K} &\text{(\Cref{L:BLcor11.4.2})}\\
                         &= - \widehat {\beta^*_{\bar K}}\lambda_{C_{\bar K}} \beta_{\bar K}^*\xi_{\bar K}
 \intertext{as well as the fact that }
  -\xi^{-1}e\xi
  & = -[e]_Z,
    \end{align*}
so that (b) and (c) are equivalent.
 
(c)$\iff$(d) By \Cref{P:BL544-6-7}\ref{P:BL547}
\[
  \delta(\frac{e}{(g_Z-1)!}  \Xi_{\bar K}^{g_Z-1},\Xi_{\bar K})  = -[e]_{Z_{\bar K}}.
\]
Now use  \Cref{T:BL1164a}\ref{T:BL1164aa} to see that, since $\Xi_{\bar K}$ is ample, $\delta(\Gamma_1,\Xi_{\bar K}) =
\delta(\Gamma_2,\Xi_{\bar K})$ if and only if $\Gamma_1\equiv\Gamma_2$.
\end{proof}

\begin{rem}[Welters]\label{R:welters2}
     For an abelian variety $A$ over a field $K$, denote by $e_{A,e}\colon A[e]\times A[e]\to \mmu_{e,K}$ the Weil pairing on $e$-torsion.  
In the setting of \Cref{L:BL-12.2.3}, set $W=\xi \beta^*\zeta_Z(Z)$ to be the image of $Z$ in $\operatorname{Pic}^0_{C/K}$, so that we have a factorization $\beta^*\circ \zeta_Z\circ  \xi\colon   Z\twoheadrightarrow W \hookrightarrow \operatorname{Pic}^0_{C/K}$, and let $\xi|_{\widehat W}\colon \widehat W\to  W$ be the pull-back of the principal polarization $\xi$ along  $\widehat W\to \widehat Z\stackrel{\zeta^{-1}}{\to} Z$.   If $K=\bar K$ is an algebraically closed field of characteristic $0$,
 \cite[Prop.~1.17]{welters87} shows that \Cref{L:BL-12.2.3}\ref{L:BL-12.2.3a}, \ref{L:BL-12.2.3b}, \ref{L:BL-12.2.3d} are equivalent to  
   the condition that $\ker \lambda_C|_{W}\subseteq W[e]$ (see \Cref{R:welters1}) together with the data of a maximal isotropic 
  subgroup $\ker \lambda_C|_{W}\subseteq H\subseteq W[e]$  with respect to the Weil pairing
   $e_e\colon \operatorname{Pic}^0_{C/K}[e]\times\operatorname{Pic}^0_{C/K}[e]\to   \mmu_e$.   Welters' arguments hold over an arbitrary algebraically closed field provided
one works with $\ker(-)^{\abvar}$  instead of
        $\ker(-)^\circ$ in \cite[(1.4) and (1.13)]{welters87}, and  in the calculation of the Weil pairing at the bottom of
          \cite[p.90]{welters87}, one works with the group schemes, rather than just $\bar K$ points.
                          \end{rem}

From \Cref{L:BL-12.2.3} we obtain what is known as Welters'
Criterion \cite{welters87}.  Welters works over an algebraically closed field of characteristic $0$, but his arguments hold more generally (see \Cref{R:welters2})
   over an arbitrary algebraically closed field. 
    Our presentation follows that of \cite{BL}, which in turn uses strategies from 
\cite[\S 2]{matsusaka59}.

\begin{teo}[{Welters' Criterion \cite[(1.18--1.19)]{welters87}, \cite[Crit.~12.2.2]{BL}}]\label{T:Welters}
 Let $(Z,\xi)/K$ be a principally polarized abelian variety of dimension $g_Z$ and let $C/K$ be a smooth projective curve. 
 Fix  divisors $\Theta_{\bar K}$ and $\Xi_{\bar K}$ such that $\phi_{\Theta_{\bar K}}=(\lambda_C)_{\bar K}$ and $\phi_{\Xi_{\bar K}}= \xi_{\bar K}$.
 
\begin{alphabetize}

\item  Suppose there is a $K$-morphism $\beta\colon C\to T$  to a torsor under $Z$ over $K$ such that 
\begin{enumerate}[label=(\roman*)]
\item the composition $\xymatrix{\widehat Z \ar[r]^<>(0.5){\zeta_Z}_<>(0.5)\sim& \operatorname{Pic}^0_{T/K} \ar[r]^{\beta^*}&  \operatorname{Pic}^0_{C/K}}$ is an injective $K$-homomorphism of  abelian vareties, where $\zeta_Z$ is the canonical isomorphism \eqref{E:PicTors}, and 
\item after identifying $T_{\bar K}$ and $Z_{\bar K}$, we have $\beta_{\bar K*}[C_{\bar K}]\equiv e \frac{\displaystyle [\Xi_{\bar K}]^{g_Z-1}}{(g_Z-1)!} $.
\end{enumerate}
Then the inclusion $\iota_Z:= \beta^*\zeta_Z\xi:Z\hookrightarrow \operatorname{Pic}^0_{C/K}$ makes $(Z,\xi)$ a Prym--Tyurin variety of exponent $e$; i.e., $\iota_Z^*\lambda = e\xi$, so that $(Z,\xi,C,\iota_Z)$ is an embedded Prym--Tyurin variety of exponent $e$.

\item Conversely, suppose there is an  inclusion $\iota_Z\colon Z\hookrightarrow \operatorname{Pic}^0_{C/K}$ making $(Z,\xi)$ a Prym--Tyurin variety of exponent $e$; i.e., $(Z,\xi,C,\iota_Z)$ is an embedded Prym--Tyurin variety of exponent $e$. Let $Y$ be the complement of $Z$.  Then, under the isomorphism 
$
\xymatrix@C=2em{
Z \ar[r]^<>(0.5){-\xi}_\sim& \widehat Z \ar[r]_<>(0.5)\sim& \operatorname{Pic}^0_{C/K}/Y,
}
$
(see \eqref{E:BL1213-2})  $T:= \operatorname{Pic}^{(1)}_{C/K}/Y$ is a torsor under $Z$, and the composition $\xymatrix{\beta \colon C\ar[r]^<>(0.5){\alpha^{(1)}}& \operatorname{Pic}^{(1)}_{C/K}\to T}$ of the Abel map with the quotient map satisfies conditions (i) and (ii) above.  
\end{alphabetize}

Moreover, these constructions are inverse to one another, up to canonical isomorphisms.
          \end{teo}

\begin{proof} It suffices to check the assertions  of (a) and (b) after base change to the algebraic closure, so we will assume that $K=\bar K$, and we fix a point $P\in C(K)$.  
(a) follows immediately from \Cref{L:BL-12.2.3}.
 
For (b),  
 suppose conversely that   there is an  inclusion $\iota_Z\colon Z\hookrightarrow \operatorname{Pic}^0_{C/K}$ making $(Z,\xi)$ a Prym--Tyurin variety of exponent $e$.  First consider the map $\xymatrix{\beta \colon C\ar[r]^<>(0.5){\alpha^{(1)}}& \operatorname{Pic}^{(1)}_{C/K}\to T}$.  Again using $\beta(P)$ to identify $T$ with $Z$, unwinding the definitions, and using the short exact sequence \eqref{E:BL1213-2}, we have 
 $\beta=  -\xi^{-1} \widehat \iota_Z \lambda_C \alpha_P\colon  C\to Z$. 
   This implies that  $\beta^* =  -\alpha_P^* \lambda_C \iota_Z \xi^{-1}=  \iota_Z \xi^{-1}$,  where here we are using \Cref{L:BLsec11.3}.  Consequently,  $\beta^*\colon \widehat Z\to \operatorname{Pic}^0_{C/K}$ is an injective $K$-homomorphism giving (i),
   where we are using that  $\zeta_Z$ is the identity after we have identified $T=Z$. 
   It also follows from the identity $\beta^*= \iota_Z \xi^{-1}$ that  $\iota_Z=\beta^*\xi$, which equals  $\beta^*\zeta_Z\xi$, again since $\zeta_Z$ is the identity after we have identified $T=Z$.   Therefore, $(\beta^*\zeta_Z\xi)^*\lambda_C= \iota_Z^*\lambda_C=e\xi$, so  we can employ \Cref{L:BL-12.2.3}(a), and we obtain (ii). 
   
The equality $\iota_Z=\beta^*\zeta_Z\xi$ also shows one direction of the assertion that the two constructions in the theorem are inverses to one another.  In the other direction, given $\beta\colon C\to T$ satisfying (i) and (ii), one observes that $\beta$ factors through the Abel map $\alpha^{(1)}\colon C\to \operatorname{Pic}^{(1)}_{C/K}$ using the universal property of the Albanese torsor.  This gives an isomorphism of $T$ with $\operatorname{Pic}^{(1)}_{C/K}/Y$. 
           \end{proof}

As a consequence, we get  the Matsusaka Criterion:

\begin{cor}\label{C:WMR}
Let $(Z,\xi)/K$ be a principally polarized abelian variety of dimension $g_Z>0$ and let $C/K$ be a smooth projective curve. 
 Fix  divisors $\Theta_{\bar K}$ and $\Xi_{\bar K}$ such that $\phi_{\Theta_{\bar K}}=(\lambda_C)_{\bar K}$ and $\phi_{\Xi_{\bar K}}= \xi_{\bar K}$.
 Suppose there is a $K$-morphism $\beta\colon C\to T$  to a torsor under $Z$ over $K$ such that 
after  identifying $T_{\bar K}$ and $Z_{\bar K}$, we have 
\begin{equation}\label{E:WMR}
\beta_{\bar K*}[C_{\bar K}]\equiv \frac{[\Xi_{\bar K}]^{g_Z-1}}{(g_Z-1)!}.
\end{equation}
Then the composition  $\iota_Z:= \beta^*\zeta_Z\xi\colon Z\to\operatorname{Pic}^0_{C/K}$ is an isomorphism, and there is a unique isomorphism $\gamma\colon \operatorname{Pic}^{(1)}_{C/K}\to T$ such that $\beta = \alpha^{(1)}\gamma$.  In other words, up to an isomorphism of torsors, $\beta$ is the Abel map.  
\end{cor}

\begin{rem}
The equality in \eqref{E:WMR} does not depend on the choice of identification of $T_{\bar K}$ with $Z_{\bar K}$, as any two choices will differ by translation.
\end{rem}

\begin{proof}
We use \Cref{L:BL-12.2.3}(d) to conclude that diagram \eqref{E:BL-12.2.3} commutes with $e=1$.  Therefore, since $\xi$ is an isomorphism, a diagram chase implies that $\beta^*\zeta_Z$ is injective, so that $\beta^*\zeta_Z\xi$ is injective, as well.  Setting $\iota_Z=\beta^*\zeta_Z\xi$,  \Cref{T:Welters} then implies that $(Z,\xi, C,\iota_Z)$ is an embedded Prym--Tyurin variety of exponent $1$.  We then conclude using \Cref{C:e=1-ppav} that, since $\dim Z>0$, and $\lambda_C$ is indecomposable, we have $\iota_Z(Z)=\operatorname{Pic}^0_{C/K}$.  
The final assertion follows from the universal property of the
Albanese torsor.  (See also \cite[Rem.~12.2.5]{BL} for the case $K = \cx$.)
\end{proof}

\begin{rem}[Matsusaka--Ran Criterion]\label{R:MR}
In fact, using \cite[Thm.~p.329]{collino}, one can replace the condition \eqref{E:WMR} in \Cref{C:WMR} with the weaker condition that $\beta_{\bar K*}[C_{\bar K}].[\Xi_{\bar K}]=g_Z$.  This would be the natural translation of the Matsusaka--Ran Criterion \cite[Thm.~p.329]{collino} to the case of non-closed fields.  
Note also that in \cite[Thm.~p.329]{collino},  disjoint unions of curves are also considered; one can formulate a corresponding version of the Mastsusaka--Ran Criterion over non-closed fields, as well. 
\end{rem}

\begin{cor}\label{C:PPAV=PT}
Let $(Z,\xi)/K$ be a principally polarized abelian variety of
dimension $g$ over a field.  Then $(Z,\xi)$ is a Prym--Tyurin variety
of exponent $n^{g-1}(g-1)!$ for infinitely many $n$, and for
\emph{all} $n\ge 3$ if $\operatorname{char}(K) = 0$.
\end{cor}

\begin{proof}
  (See also {\cite[Cor.~12.2.4]{BL}} when $K = \cx$.)
  Let $n\ge 3$ be an integer. (During the course of the proof we will
  encounter two places where, if $K$ has positive characteristic, then
   $n$ must
  be replaced by a suitable multiple.)
 According to a theorem of Lefschetz
  $n\Xi$ is very ample (e.g., \cite[\S 17]{mumfordAV}), and thus
  defines an embedding in projective space.  If $K$ is infinite, then
  using Bertini's theorem we can take the intersection of $g-1$
  hyperplanes to obtain an embedding $\beta\colon  C\hookrightarrow Z$ of a
  smooth projective curve over $K$. (If $K$ is finite, it may be
  necessary to take a higher multiple of $\Xi$.  That this suffices is
  due to Gabber and to Poonen \cite{gabberGAFA,poonen}; indeed,
  proving that any abelian variety is a quotient of a Jacobian was the
  impetus for those works.) Then the image of $C$ generates $Z$
  (e.g,. \cite[Proof of Thm.~10.1]{milneAV}). By construction we obtain that
$$
\beta_*[C]\equiv [n\Xi]^{g-1}= n^{g-1}[\Xi]^{g-1},
$$
so that from Welters' Criterion (\Cref{T:Welters}), all we have to show is that $\beta^*\colon \widehat Z \to \operatorname{Pic}^0_{C/K}$ is injective.
We may verify this injectivity over the algebraic closure of $K$.  So we may assume that there is a $K$-point $P\in C(K)$, and we consider the diagram 
\begin{equation*}\label{E:til-beta'}
  \xymatrix{
    C \ar[r]^\beta \ar[d]_{\alpha^{(1)}} \ar@/_2pc/[dd]_{\alpha_P} & Z \\
    \pic^{(1)}_{C/K}\ar@{->}[ur]^{\til\beta} \ar[d]& \\
    \pic^{0}_{C/K} \ar@{-->}[ruu]_{\beta_P}& 
  }
\end{equation*}
where $\beta_P$ is defined as in the diagram. From
\Cref{L:BLcor11.4.2} we find that $\widehat {\beta_P}=
-\phi_{\Theta}\beta^*$.   Thus it suffices to show that $\widehat
{\beta_P}$ is injective.  Because $C$ generates $Z$, $\beta_P$ is
surjective; by \Cref{C:hatf}, it now suffices to verify that
$\ker(\beta_P)$ is an abelian variety.  

If $K$ has characteristic $0$, then a hyperplane section of a normal
variety induces a surjection of \'etale fundamental groups \cite[Cor.~XII.3.5]{sga2}, and thus of 
fundamental group schemes.  In positive characteristic, the failure of
Kodaira vanishing can present an obstacle to surjectivity on
fundamental group schemes, which can be rectified by passing to a
suitable multiple of the hyperplane section
\cite[Thm.~1.1]{biswasholla07}.  Thus, possibly after replacing
$n\Xi$ with some multiple $mn\Xi$, if $\operatorname{char}(K)>0$,  we
may and do assume that $\beta$ induces a surjection
$\beta_*\colon \pi_1^\nori(C_{\bar K}) \twoheadrightarrow \pi_1^\nori(Z_{\bar
  K})$.   Now, $\alpha_{P,*}\colon  \pi_1^\nori(C_{\bar K}) \to \pi_1^\nori(\pic^0_{C/\bar
  K})$ is the abelianization map, and in particular surjective
\cite[Cor.~3.8]{antei11}.  Therefore, $\beta_{P,*}\colon \pi_1^\nori(\pic^0_{C/\bar
  K})\to \pi_1^\nori(Z_{\bar K})$ is surjective, and so
(\Cref{L:nori}) $\ker(\beta_P)$ is an abelian variety.
\end{proof}

\begin{rem}
  \label{R:exponentbound}
In positive characteristic, it is possible to specify a large supply of
suitable numbers $n$ in the statement of \Cref{C:PPAV=PT}.  Indeed,
over a finite field, \cite{bruceli} provides effective bounds for the
degrees of hypersurfaces which meet $Z$ in a smooth subvariety; and
over an arbitrary field of positive characteristic,
\cite[Thm.~3.5]{biswasholla07} provides effective bounds which
similarly ensure that the induced map on fundamental group schemes is
surjective.

In each case, the bounds are in terms of the absolute size of the
degree, rather than divisbility properties.  Consequently, given any
$n\ge 3$, there exists some $r$ such that $(Z,\xi)$ is a Prym--Tyurin
variety of exponent $n^r(g-1)!$.
 \end{rem}

\subsection{Prym--Tyurin Prym varieties}
\label{SS:prym}
Recall from \S \ref{S:Pscheme}  that to a finite morphism $f\colon C\to C'$
of smooth projective curves over a field $K$ we have associated a Prym variety, $P(C/C')$, defined to be the complement of $Y:= f^*\operatorname{Pic}^0_{C'/K}$, and that $P(C/C')$  comes with an inclusion $\iota \colon P(C/C')\hookrightarrow \operatorname{Pic}^0_{C/K}$.   We say that $P(C/C')$ is a Prym--Tyurin Prym variety of exponent $e$ if $\iota^*\lambda_C=e\xi$ for the positive integer $e$ and some principal polarization $\xi$ on $P(C/C')$.   
       Recall that \eqref{E:P(C/C')} and \Cref{C:PrymC/C'Nm} give alternative descriptions of $P(C/C')$, and that in the case that $C/C'$ is a cyclic Galois cover, \Cref{R:C/sigma} gives yet another description.

As it turns out, there are a number of restrictions on the exponent  and the type of cover that can give rise to a Prym--Tyurin Prym variety.  
           Over an algebraically closed field of  characteristic not equal to $2$, Mumford \cite{mumford74} classified those
degree $2$ covers $f$ which give rise to a Prym--Tyurin Prym variety.  We follow the exposition of
\cite[Thm.~12.3.3]{BL}, and extend it to non-closed fields of all characteristics.  Note
that there is a slight oversight in \cite{BL}, which is subsequently
corrected in \cite{langeortega11} (the case missing in \cite{BL} corresponds  to \ref{prym:etaletriple} in \Cref{T:classifyprym}, below).

\subsubsection{Inseparable covers} \label{S:purelyinsep}

    We want to collect some facts about inseparable covers and Frobenius morphisms.   To this end, in this section assume that
$\operatorname{char}(K)=p>0$.

Let $T$ be any scheme of characteristic $p$. The $p^{th}$ power map on
$\mathcal O_T$ defines the absolute Frobenius morphism $\absfrob_T: T
\to T$.  If $X/K$ is any scheme, set
$X^{(p/K)} = X^{(p)} = X\times_{\spec K,\absfrob_{\spec K}}\spec
K$. The absolute Frobenius morphism $\operatorname{fr}_X$ of $X$ is a morphism over $\operatorname{fr}_K$, and so, by the definition of the fiber product, factors through a
canonically-defined relative Frobenius $K$-morphism $F_{X/K}$:
       \[
  \xymatrix{
X \ar[r]^{F_{X/K}} \ar@/^2pc/[rr]^{\absfrob_X} \ar[rd] &  X^{(p)} \ar[r] \ar[d] & X \ar[d] \\
&    \spec K \ar[r]^{\absfrob_K} & \spec K
  }
\]
The relative Frobenius is functorial.  In particular, a morphism $\alpha:X \to Y$ of $K$-schemes induces a morphism $\alpha^{(p)}$ such that the following diagram commutes:
\[
  \xymatrix{X \ar[r]^\alpha \ar[d]^{F_X} & Y \ar[d]^{F_Y}\\
    X^{(p)} \ar[r]^{\alpha^{(p)}} & Y^{(p)}}\]
If $r\ge 2$, we inductively define
$X^{(p^r)} = (X^{(p^{r-1})})^{(p)}$, and with a slight abuse of notation write $F^{\circ r}_X$ for $F\circ_{X^{(p^{r-1})}} \circ \cdots \circ F_{X}\colon X \to X^{(p^r)}$.

In the special case where $X$ is actually a commutative group scheme,
one can use this group structure to canonically define a so-called Verschiebung
morphism $V_{X/K}: X^{(p)} \to X$, which has the property that  $V_{X/K} \circ F_{X/K}
= [p]_{X/K}$, the multiplication-by-$p$ map.  This is worked out in generality
in \cite[VIIA.4.3]{sga3-1}, and in the special case we need in \cite[\S
5.2]{EvdGM};
   see also \cite{odaderham}.    If $X$ is an abelian
variety, then there is a canonical isomorphism $\widehat{X^{(p)}} \iso
(\widehat X)^{(p)}$ \cite[Prop.~2.1]{odaderham}, and Frobenius and Verschiebung are dual in the
sense that $\widehat{V_{X/K}} = F_{\widehat X/K}$ and
$\widehat{F_{X/K}} = V_{\widehat X/K}$ \cite[Cor.~2.2]{odaderham}.

Now consider a cover $f\colon C \to C'$ of curves over $K$.
Then $f$ factors uniquely as \cite[\href{https://stacks.math.columbia.edu/tag/0CD@}{Prop.~0CD2}]{stacks-project}
\begin{equation}
\label{E:insep-cov}\xymatrix{
  C \ar@/^2pc/[rr]^f\ar[r]_-{i=F^r_{C/K}} & C^{(p^r)} = C'' \ar[r]_-j & C' \\
}
\end{equation}
where $i$ is purely inseparable and $j$ is separable.
                
         As the genus of $C$ is equal to that of $C^{(p^r)}$ (e.g, 
 \cite[\href{https://stacks.math.columbia.edu/tag/0CD0}{Lem.~0CD0}]{stacks-project}), from \Cref{L:pullpush} we have that $(F^r_{C/K})^*:\operatorname{Pic}^0_{C^{(p^r)}/K}\to \operatorname{Pic}^0_{C/K}$ is an isogeny, 
 so that 
$i^*\colon \pic^0_{C''/K} \to \pic^0_{C/K}$ is an isogeny, and thus $P(C/C'')$ is trivial.

For later use, we will also want to recall from \cite[Prop.~2.1]{odaderham} the effect of the relative Frobenius 
$F_{C/K}$
 on Picard schemes.  There is a canonical identification
$\pic^0_{C^{(p)}/K} \iso (\pic^0_{C/K})^{(p)}$; $(F_{C/K})_*:\operatorname{Pic}^0_{C/K}\to \pic^0_{C^{(p)}/K}$ and $(F_{C/K})^*: \pic^0_{C^{(p)}/K}\to \operatorname{Pic}^0_{C/K}$ are then 
canonically identified with, respectively, the relative Frobenius $F$ and Verscheibung $V$ 
maps of $\pic^0_{C/K}$; this together with Lemma \ref{L:pullpush} becomes the
factorization $[p] = F\circ V$:
\[
\xymatrix@R=.5em{
\operatorname{Pic}^0_{C^{(p)}/K} \ar[r]^{F_{C/K}^*} \ar@{=}[d] \ar@/^2pc/[rr]^{[p]}& \operatorname{Pic}^0_{C/K} \ar[r]^{F_{C/K,*}} \ar@{=}[d]&\operatorname{Pic}^0_{C^{(p)}/K} \ar@{=}[d]\\
(\operatorname{Pic}^0_{C/K})^{(p)}\ar[r]_<>(0.5)V \ar@/_2pc/[rr]_{[p]}& \operatorname{Pic}^0_{C/K} \ar[r]_<>(0.5)F& (\operatorname{Pic}^0_{C/K})^{(p)}
}
\]
         
\begin{pro}[Inseparable covers]\label{P:insep}
  Let $f\colon C \to C'$ be a finite morphism of smooth projective curves over a field
  $K$, factored as in \eqref{E:insep-cov}.   Then we have an
  isomorphism of Prym varieties $P(C/C')^{(p^r)} = P(C^{(p^r)}/C')$.
     Moreover, $P(C/C')$ is a Prym--Tyurin Prym variety of exponent $e$, i.e., $\lambda_C|_{P(C/C')}= e\xi$ for a principal polarization $\xi$ on $P(C/C')$,   if and only if the Prym variety $P(C''/C')$ associated to the separable cover $C''/C'$ is a Prym--Tyurin Prym variety of exponent $e$.
\end{pro}

\begin{proof}
By considering the factorization \eqref{E:insep-cov}, and by induction
on the degree of inseparability of the morphism $f$, it suffices to
prove that given morphisms of curves
\[
  \xymatrix{ C \ar[r]^{F_C} & C^{(p)} \ar[r]^h &C'}
\]
with $h$ finite, we have $P(C/C')$ is a Prym--Tyurin Prym variety of
exponent $e$ if and only if $P(C^{(p)}/C')$ is a Prym--Tyurin Prym
variety of exponent $e$.

Let $X = \pic^0_C$; then $\pic^0_{C^{(p)}}$ is canonically isomorphic
to $X^{(p)}$ \cite[Prop.~2.1]{odaderham}.
Note that $P(C^{(p)}/C')$ is the complement (with respect to the
polarization  $\lambda_{C^{(p)}}$) of $h^*\pic^0_{C'}$ in
$X^{(p)}$.  The equality of morphisms  $V_X = (F_C)^*$ implies that 
$P(C/C')$ is the complement in $X$ (with
respect to $\lambda_C$) of $(h
F_C)^*\pic^0_{C'} = V_X(h^*\pic^0_{C'})$.

Similarly, the equality of morphisms $V_X = (F_C)^*$ and  $F_X = (F_C)_*$, together with
the duality $\widehat{F_X} = V_{\widehat X}$
and the calculation of Lemma \ref{L:f-up-lw*}, shows that the upper-right hand corner in
the following diagram commutes.  (The bottom row is obtained from the
middle row by the functoriality of Frobenius.)
\[
  \xymatrix{
    h^*\pic^0_{C'} \ar[d]^{V_X} \ar@{^(->}[r] & X^{(p)} \ar[d]^{V_X} \ar[r]^{\lambda_{C^{(p)}}} &
      \widehat X^{(p)}\ar[d]^{V_{\widehat X}}\\
      V_X(h^*\pic^0_{C'}) \ar[d]^F \ar@{^(->}[r] & X \ar[d]^F
      \ar[r]^{\lambda_C} & \widehat X \ar[d]^F\\
      (V_X(h^*\pic^0_{C'}))^{(p)} \ar[r] & X^{(p)}
      \ar[r]^{(\lambda_C)^{(p)}} & \widehat X^{(p)}
    }
  \]
The composition of arrows in each column is the multiplication-by-$p$
map.  Consequently, considering the right two columns, we have $[p]\circ \lambda_{C^{(p)}} = (\lambda_C)^{(p)} \circ [p]$, so that by the freeness of the endomorphism ring as a $\mathbb Z$-module, we have   $\lambda_{C^{(p)}} = (\lambda_C)^{(p)}$; and if we
let $Y = V_X(h^*\pic^0_{C'})$, then $h^*\pic^0_{C'} = Y^{(p)}$.
 We therefore have a commutative diagram with
exact rows
\[
  \xymatrix{
    0 \ar[r] & Y \ar[r]\ar[d]^F & X \ar[r] \ar[d]^F& \widehat{P(C/C')} \ar[d]^F
    \ar[r] & 0 \\
    0 \ar[r] & Y^{(p)} \ar[r] & X^{(p)} \ar[r] &
    \widehat{P(C/C')}^{(p)} \ar[r] &0
  }
\]
and in particular, if we  set $Z = P(C/C')$, then
\[
  P(C^{(p)}/C') =   P(C/C')^{(p)} = Z^{(p)}.
\]
Moreover,
\[
  Y^{(p)}\times_{X^{(p)}} Z^{(p)} = (Y\times_XZ)^{(p)}.
\]
For any integer $e$, an easy calculation shows $([e]_Z)^{(p)} =
[e]_{Z^{(p)}}$, and thus $(Z[e])^{(p)} = Z^{(p)}[e]$.  Consequently, by \Cref{L:kerpullbackpolarization}
  and  \Cref{R:Z[e]-PT},  
     we may conclude that 
  $P(C^{(p)}/C')$ is Prym--Tyurin Prym variety of exponent $e$ if and only if $Y^{(p)}\times_{X^{(p)}} Z^{(p)}= Z^{(p)}[e]$, if and only if $Y\times_XZ= Z[e]$, if and only if  $P(C/C')$ is a
  Prym--Tyurin Prym variety of exponent $e$.  
\end{proof}

\subsubsection{Kernel of the pull-back morphism for line bundles}

Next, we recall the relation between covers of curves and the kernel of the pull-back morphism on line bundles.

Say that a finite morphism of curves $C \to C'$ is an abelian
cover if $C$ is a torsor over $C'$ under some finite commutative group
scheme. 

\begin{lem}
  \label{L:newBLprop11.4.3}
  Let $f:C \to C'$ be a finite morphism of curves over $K$, and let $f^*$ be the induced morphism $f^*: \pic^0_{C'/K} \to \pic^0_{C/K}$. 
  \begin{alphabetize}
  \item Suppose $C$ is a torsor under the finite commutative group scheme $G$.  If $C$ admits a $K$-rational point, then $\ker(f^*) \iso G^\vee$; in general, $\ker(f^*)^\vee$ is isomorphic to a twist of $G$.

\item Suppose $f:C \to C'$ is purely inseparable of degree
  $p^r$. After choosing an isomorphism $C' = C^{(p^r)}$, we have
  $\ker(f^*) =  \ker(V^{\circ r}_{\pic^0_C})$.  Moreover, $f$ is
  abelian if and only if $C$ and $C'$ have genus one.

  \item There is an abelian cover $C^\abelian \to C'$, which is initial among all abelian covers of $C'$ through which $C \to C'$ factors, and sits in the following diagram:
     \begin{equation}\label{E:MaxAbCovFact}
 \xymatrix@C=1em@R=1em{
 C \ar[rr]^f \ar[rd]_{f^\nonabelian}&& C'\\
 &C^\abelian \ar[ru]_{f^\abelian}&
 }
     \end{equation}
If $C$ admits a $K$-rational point, then the cover $C^\abelian \to C'$
is a torsor under $\ker(f^*)^\vee$; in general, $C^\abelian$ is a torsor under some twist of $\ker(f^*)^\vee$.
If $K$ is perfect or if $f$ is separable, then  formation of $C^\abelian$ is compatible with algebraic
extension of the base field.

\item If $f^*$ is an inclusion, then $C^\abelian = C'$.  If $f$ is
  separable and if $C^\abelian = C'$, then  $f^*$ is an inclusion.
  
\item Suppose $f$ is separable.  Then $f$  is an abelian cover if and only if $\deg(f) = \deg(f^*)$.
  \end{alphabetize}
\end{lem}

\begin{proof}
Suppose that $C$ admits as $K$-point, $P$, and let $P' = f(P)$.
  The pointed Abel map $\alpha_P\colon  (C,P) \to
(\pic^0_{C/K},0)$ induces a map $\pi_1(\alpha_P)\colon  \pi_1^\nori(C)
  \to
\pi_1^\nori(\pic^0_{C/K})$ of fundamental group schemes, which gives
an isomorphism on the abelianization of the fundamental group scheme
of $C$ \cite[Cor.~3.8]{antei11}:
\[
  \xymatrix{
    \pi_1^\nori(C) \ar@{->>}[r] & \pi_1^\nori(C)^\abelian
    \ar[r]^-\sim & \pi_1^\nori(\pic^0_{C/K}).
  }\]
Now consider the finite morphism $f\colon (C,P) \to (C',P')$ of pointed curves.  In the following diagram, the existence of the top row is built into the theory of Nori's fundamental group; the descent to the middle row is because, by hypothesis, $G$ is abelian; and the commutativity of the bottom rectangle follows from  \eqref{E:fCC'} and the chosen trivializations of
$\pic^{(1)}_{C/K}$ and $\pic^{(1)}_{C'/K}$:
\[
  \xymatrix{
0 \ar[r]&     \pi_1^\nori(C) \ar@{->>}[d] \ar[r]^{\pi_1(f)} & \pi_1^\nori(C')
    \ar@{->>}[d] \ar@{->>}[r] & G \ar[r] \ar[d]^\sim& 0\\
&    \pi_1^\nori(C)^\abelian \ar[d]^\sim \ar[r]^{\pi_1(f)^\abelian} &
    \pi_1^\nori(C')^\abelian \ar[d]^\sim \ar[r] & G \ar[r]\ar[d]^\sim & 0\\
&    \pi_1^\nori(\pic^0_{C/K}) \ar[r]^{\pi_1(f_*)} &
    \pi_1^\nori(\pic^0_{C'/K})\ar[r] & G \ar[r] & 0
  }
\]
In particular, $\coker(\pi_1(f_*)) \iso G$. By Lemma \ref{L:kerhatf}, $\ker(\widehat{f_*}) \iso G^\vee$.  Now (a) follows because $\ker(\widehat{f_*}) \iso \ker(f^*)$ (Lemma \ref{L:f-up-lw*}).  If $C$ admits no $K$-point, then $C$ is still a torsor over $C'$ under a group scheme which becomes isomorphic to $\ker(f^*)^\vee$ after a finite extension of $K$.

For (b), suppose $f:C \to C'$ is a purely inseparable abelian cover of
degree $p^r$.  The description of $\ker(f^*)$ follows from our
discussion of Frobenius above.  For the remaining claim, on one hand, since $f$ is purely inseparable, $C$ and
$C'$ have the same genus.  On the other hand, because $C$ is a torsor
over $C'$, its Euler characteristic $\chi(C)$ satisfies $\chi(C) =
\deg(f)\chi(C')$.  Therefore, $\chi(C) = \chi(C') = 0$, and the common
genus of $C$ and $C'$ is one.

For (c), suppose $C_1$ and $C_2$ are torsors over $C'$ under
respective finite commutative group schemes $G_1$ and $G_2$ through
which $f$ factors.  Then $C_{12}:=  C_1\times_{C'} C_2$ is a torsor
under the finite commutative group scheme $G_1\times G_2$, and $f$
factors through $C_{12}$.  Since $C \to C'$ is finite, we may find an
initial such cover, and produce the factorization
\eqref{E:MaxAbCovFact}. The description of the covering group of
$C^\abelian \to C'$ is given in (a).

We investigate the compatibility of the formation of $C^\abelian$ with
algebraic 
field extensions $L/K$. If $f$ is separable, the claimed compatibility
is classical; indeed, $C^\abelian$ is the curve whose function field
$K(C^\abelian)$ is the maximal unramified abelian extension of $K(C')$
inside $K(C)$.  Otherwise, suppose that $L/K$ is finite and
separable; we may assume $L/K$ is actually Galois.
Then the collection 
of abelian covers of $C'_L$ through which $f_L$ factors is stable
under $\gal(L/K)$, and thus $(C_L)^\abelian$ descends to
$K$.

For (d), suppose $f^*$ is an inclusion.
Let $C'' \to C'$ be a torsor over $C'$ under the finite
commutative group scheme $G$.  If $C \to C'$ factors through $C''$,
then $f^*$ factors through $g^*$.  Since $\ker(g^*) \subseteq
\ker(f^*) = \st{1}$, the trivial group scheme, we find that $C''$ is a
torsor under $\ker(g_*)^\vee \iso \st{1}$.

For the converse, suppose $f$ is separable.  Then the
usual theory of the \'etale fundamental group shows that \'etale
(Galois) subcovers of $C'$ correspond to (normal) subgroups of
$\pi_1^\et(C')$ which contain the image of $\pi_1^\et(C)$, and \'etale
abelian subcovers of $C'$ correspond to subgroups of $\pi_1^\ab(C')$
which contain the image of $\pi_1^\ab(C)$.  Consequently, $C^\abelian$
is a torsor over $C'$ under $\ker(f^*)^\vee$.  In particular,
$C^\abelian = C'$ if and only if $f^*$ is an inclusion.

Finally, we address (e); suppose $f$ is separable.  Part (a) implies that if $f$ is
abelian with covering group the \'etale group scheme $G$, then
$\deg(f) =\ord(G) = \ord(G^\vee) = \deg(f^*)$.  Conversely, if
$\deg(f) = \deg(f^*)$, then $\deg(f^\nonabelian) = 1$, and thus $C \to
C'$ is an abelian cover.
\end{proof}

\begin{rem}
  \label{R:defprank}
Lemma \ref{L:newBLprop11.4.3}, via its reliance on \cite{antei11}, encodes (and is built upon) the starting point
of geometric class
field theory: for a fixed curve $C'$, there is a bijection
between cyclic \'etale covers $C \to C'$ of degree $N$ and multiplicative
cyclic subgroups of $\pic^0_{C'/K}$ of degree $N$, i.e., sub-group schemes of
$\pic^0_{C'/K}$ which are isomorphic, over $\bar K$, to
$(\integ/N)^\vee = \mmu_N$.  (In fact, this is a special case of the
general fact that, for a pointed smooth proper scheme
$X$ over $S$ and a finite commutative group scheme $G/S$, there is a
canonical bijection $H_\bullet^1(X,G) \iso \hom(G^\vee,
\pic^\tau_{X/S})$,
  where the left-hand side denotes the subgroup of
torsors which are pointed over the given section of $X$, and
$\pic^\tau_{X/S}$ denote the torsion component of the Picard scheme
\cite[Prop.~3.2]{antei11}.)  In particular, if $C'_{\bar K}$ admits
such a cover, then there exists an inclusion $\mmu_{N,\bar K}
\hookrightarrow \pic^0_{C'/\bar K}$.

If $N$ is invertible in $K$, this is always possible; but if $p =
\operatorname{char}(K)>0$, then the existence of such an inclusion is
a nontrivial constraint on $C'$.

Indeed, recall that if $A/K$ is an abelian variety over a field, then
there exists an integer $f$, $0 \le f \le \dim A$, such that the
\'etale quotient $A[p]^{\et}$ has order $p^f$; this is the $p$-rank of $A$.  Equivalently,
$A[p](\bar K) \iso (\integ/p)^f$.  Since $A[p]^\vee \iso A[p]$, the
$p$-rank may also be computed as $\dim_{\ff_p} \hom(\mmu_{p,\bar K},
A[p]_{\bar K})$.  The abelian variety is called ordinary if $f=\dim A$,
i.e., its $p$-rank is as large as possible.
The $p$-rank of a curve is by definition the
$p$-rank of its Jacobian, and a curve is called ordinary if its
Jacobian is.

Thus, $C'$ admits an \'etale cyclic $p$-cover (over $\bar K$) if and
only if its $p$-rank is positive.  We return to this point in \Cref{R:prymprank}.
\end{rem}

\subsubsection{Classification of Prym--Tyurin Prym varieties}

We now classify all Prym--Tyurin Prym varieties over a field.  The cases where the genus of the base curve is $0$, or the genus of the base curve is positive and equal to the genus of the cover (i.e., the case $g(C)=g(C')=1$ and $f$ is \'etale), or  the case where the cover has degree $1$, are trivial, giving rise to Prym varieties of exponent $1$ (see \Cref{C:e=1-ppav}), and so we exclude those cases below.  

\begin{teo}[{Classification of Prym--Tyurin Prym varieties \cite[Thm.~12.3.3]{BL}}]
  \label{T:classifyprym}
Let $f\colon C \to C'$ be a finite  
morphism of smooth projective curves over a field $K$ with respective genera $g>g'\ge 1$, let 
$$C \stackrel{i}{\to} C''
\stackrel{j}{\to} C'$$ be the unique, up to isomorphism,  factorization of $f$ 
 with $j$ separable and $i$ purely inseparable, let $d$ be the degree of $j$, and let $e$ be a positive integer.

 The Prym variety $P(C/C')$ is a Prym--Tyurin variety of exponent $e$;  i.e., $\lambda_C|_{P(C/C')}= e\xi$ for a principal polarization $\xi$ on $P(C/C')$, if and only if $C''/C'$ is one of the following types:

\begin{alphabetize}

\item\label{prym:etaledouble} $d=2$ and $j$ is \'etale, in which case $e=2$ and $\dim P(C/C')=g'-1$;

\item \label{prym:ram}  $d=2$ and the ramification divisor of $j$ has degree $2$, in which case $e=2$ and $\dim P(C/C')=g'$;

      \item\label{prym:etaletriple} $d=3$, $j$ is \'etale and noncyclic, and $g'=2$, in which case $e=3$ and $\dim P(C/C')=g'=2$;

  \item\label{prym:g2} $g = 2$ and $g'=1$, in which case $e=\deg {j^{\operatorname{n-ab}}}$ and $\dim P(C/C')=g'=1$, where ${j^{\operatorname{n-ab}}}$ is the morphism from the factorization of $j$ in \eqref{E:MaxAbCovFact} via the maximal abelian cover defined in \Cref{L:newBLprop11.4.3}.
  \end{alphabetize}
\end{teo}

\begin{rem}[Degree $2$ covers] \label{R:deg2Pr}
In cases (a) and (b) the cover $j$ is cyclic of degree $2$, and so is Galois. Consequently, there is an involution $\sigma\colon C''\to C''$ with $C'=C''/\langle \sigma \rangle$, and $P(C''/C')=\operatorname{Im}(1-\sigma^*)$ (\Cref{R:C/sigma}).   
\end{rem}

\begin{rem}[Degree $2$ ramification]
  \label{R:wildram}
   In case (b) there are two possibilities. Since $\deg(j) = 2$, it is
  Galois; and if $P$ is a ramification point of $j$, then its
  ramification index is therefore exactly two. Either
  $\operatorname{char}(K)\not = 2$, and $j$ is (necessarily tamely)
  ramified at exactly two geometric points; or
  $\operatorname{char}(K) = 2$, $j$ is (necessarily wildly) ramified
  at a single point, and the lower ramification filtration is trivial
  at step $2$.  We call the latter case weakly wildly ramified;  any such cover looks formally
  locally like
  $K\powser x \hookrightarrow K\powser x
  [y]/(y^2-y-x)$.
        We will
  never need this explicit description, and refer the reader to
  \cite[Ch.~IV]{serreLF} for a reminder on ramification filtrations.
                  \end{rem}

\begin{rem}[$p$-ranks and Pryms]\label{R:prymprank}   Recall (\Cref{R:defprank}) that if $C'/K$ is a smooth projective curve
over a field of characteristic $p>0$, then $C'$ geometrically admits a cyclic
\'etale cover of degree $p$ if and only if its $p$-rank is positive.
We briefly explore the interaction between the $p$-rank  of $C'$ and the
classification in \Cref{T:classifyprym}.
    \begin{alphabetize}
\item Suppose that $K$ has characteristic $2$.  

\begin{enumerate}[label=(\roman*)]
\item  If $f\colon C\to C'$ satisfies the hypotheses of
  \Cref{T:classifyprym} case \ref{prym:etaledouble}, we have seen that $C'$
  has positive $2$-rank (\Cref{R:defprank}).  
  
  \item  If $f$ satisfies the hypotheses of 
  \Cref{T:classifyprym} case \ref{prym:etaletriple}, this also forces
  $C'$ to have positive $2$-rank.  Indeed, we assume that $C = C''$, i.e., that $C \to C'$ is separable.  Let
$\widetilde C \to C'$ be the Galois closure of $C$ over $C'$.  It is
\'etale over $C'$ and its Galois group, being a non-cyclic transitive
subgroup of $S_3$, is  $\operatorname{Aut}(\widetilde
C/C') \iso S_3$. 
   Let $D$ be the quotient of $\widetilde C$ by the
unique, thus normal, subgroup of order $3$.  We have a diagram of \'etale covers of curves
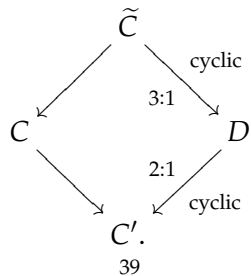
\[
  \xymatrix{
    &\widetilde C \ar[dr]_{3:1}^{\text{cyclic}} \ar[dl] \\
    C\ar[dr] && D \ar[dl]_{2:1}^{\text{cyclic}} \\
    & C'.
  }
\]
In characteristic $p=2$, the existence of the \'etale double cover $D
\to C'$ once again forces $C'$ to have positive $2$-rank.
\end{enumerate}

\item Note that in characteristic $p=3$,
 if $f$ satisfies the hypotheses of 
  \Cref{T:classifyprym} case \ref{prym:etaletriple}, then one can apply the same analysis as in the previous case, and 
 the existence of a cyclic \'etale triple
cover $\widetilde C \to D$ forces $D$ to have positive $p$-rank.
However, in characteristic $3$ this is automatic
\cite[(7.1)]{fabervandergeer04}: an \'etale double cover of a curve of
genus $2$ has positive $3$-rank.

\item 
In characteristic $p\ge 5$, we make a side remark regarding  \Cref{T:classifyprym} case \ref{prym:etaledouble} that there
exist \'etale double covers of curves of genus $2$ which have $p$-rank
zero \cite[Prop.~6.1]{ozmanpries19}.
\end{alphabetize}
\end{rem}

\begin{proof}[Proof of \Cref{T:classifyprym}]
  As promised, we follow the argument of \cite[Thm.~12.3.3]{BL}.
  First, from \Cref{P:insep}, it suffices to prove the case that $f=j$.  
We set $Z=P(C/C')$ for brevity.

We start by showing that if $Z$ is a Prym--Tyurin Prym variety of exponent $e$, then $f$ is
in one of the cases delineated in the statement of the theorem. 
First, since $\dim \operatorname{Pic}^0_{C/K}=g>g-g'=\dim Z>0$ we may assume that $e\ge 2$ (\Cref{C:e=1-ppav}).  
By hypothesis, $\iota_Z^*\lambda_C =  e\xi$ for some principal polarization $\xi$ on $Z$, which, as we noted at the beginning of \S \ref{SS:ptfield} (equivalently, \Cref{R:Z[e]-PT}) is equivalent to  $\K(\iota_Z^*\Theta) =  Z[e]$; by 
\Cref{L:kerpullbackpolarization} (see also \Cref{R:Ki*Theta}), we then have $\K(\iota_Y^*\lambda_C)
\iso Z[e]$.  In particular, $\dim Y \ge \dim Z$ (since $\K(\iota_Y^*\lambda_C)\subseteq Y[e]$, and so $ e^{2d_Z} = \deg [e]_Z= \deg \K(\iota_Y^*\lambda_C) \le \deg [e]_Y= e^{2d_Y}$), 
 i.e., $g' \ge g-g'$, and so 
 \begin{equation}\label{E:g,ge,2g'}
 2g' \ge g.
 \end{equation}
In our setting the Riemann--Hurwitz formula reads
$
  2g-2 = d (2g'-2)+\delta,
$
where $\delta$ is the degree of the ramification divisor, which combined with \eqref{E:g,ge,2g'} gives 
 \begin{equation}\label{E:RiemHurIn}
2g'-1\ge g-1 =d (g'-1)+\frac{\delta}{2}\ge d(g'-1)
 \end{equation}
 and we find that 
 \begin{equation}\label{E:bsc-est}
  d-1 \ge (d-2)g'.
\end{equation}

\begin{itemize}

  \item \emph{Suppose $d \ge 3$ and $g'\ge 3$}. Then one easily derives a numerical
contradiction in \eqref{E:bsc-est}.  

\item \emph{Suppose $d\ge 3$  and $g'=2$}.  Then from \eqref{E:bsc-est} we must have $d=3$, giving equality in \eqref{E:RiemHurIn} 
so that $\delta=0$ and $g=4$.  In particular, $f$ is \'etale, and
$\dim Y = \dim Z = 2$.  By Lemma \ref{L:pushpull}, $e|d$; since $e>1$,
we have $e=3$.  By hypothesis $Z$ is a Prym--Tyurin variety, and thus
$\K(\iota_Z^*\lambda_C)= Z[3] \iso \K(\iota_Y^*\lambda_C)$.  Therefore, since $\dim Y= \dim Z$, we have 
$\K(\iota_Y^*\lambda_C) = Y[3]$.
\begin{itemize}
  \item Suppose $f$ were cyclic; we will derive a contradiction.  Then $f^*$ is not injective (Lemma
\ref{L:newBLprop11.4.3}), so $f^*$ factors as $\iota_Y h$ with $h$ a
nontrivial isogeny.
    On one hand, this implies that $\deg ((f^*)^*\lambda_C)= \deg(
(\iota_Yh)^*\lambda_C) = \deg (\widehat h) \deg (\iota_Y^*\lambda_C)\deg(h) > \deg
\iota_Y^*\lambda_C$, and therefore the orders of their kernels satisfy
$\abs{\K((f^*)^*\lambda_C)} = \abs{\K((\iota_Yh)^*\lambda_C}>\abs{\K(\iota_Y^*\lambda_C)} =\abs{Y[3]} = \abs{\pic^0_{C'/K}[3]}$.
  On the other hand, 
\Cref{L:BL12.3.1} implies that $\K((f^*)^*\lambda_C) =
\pic^0_{C'/K}[3]$, contradicting the previous degree estimate.  

\item If $f$ is \emph{not} cyclic, then we are in case
  \ref{prym:etaletriple} of the theorem.
\end{itemize}

\item \emph{Suppose  $d \ge 2$ and $g' = 1$}.  Then by \eqref{E:g,ge,2g'} we have $g=2$.  So we are in case \ref{prym:g2}; the only thing left to show is the statement about the exponent. As we have seen, the exponents of $Z$ and $Y$ are the same (\Cref{R:Ki*Theta}), so it suffices to find the exponent of $Y$. For this, consider the factorization of \eqref{E:MaxAbCovFact}:
\begin{equation}\label{E:MaxAbCovFact-?}
 \xymatrix@C=1em@R=1em{
 C \ar[rr]^f \ar[rd]_{f^{\operatorname{n-ab}}}&& C'\\
 &C^{\abelian} \ar[ru]_{f^{\abelian}}
 }
\end{equation}
 By assumption $f$ is separable, and consequently so is
 $f^{\abelian}$.  Now, $C^{\abelian}$ is a torsor over $C'$ under a
 finite flat commutative group scheme; since $f^{\abelian}$ is
 separable, this group scheme is reduced, and so $f^\abelian$ is \'etale.
       From Riemann--Hurwitz, we conclude that $g(C^{\abelian})=1$, so that $(f^{\abelian})^*$ is an \'etale isogeny of elliptic curves, and 
 the image of $f^*$ is equal to the image of $(f^{\operatorname{n-ab}})^*$.  We then use  
\Cref{L:BL12.3.1} which says $((f^{\operatorname{n-ab}})^*)^*\lambda_C = \deg (f^{\operatorname{n-ab}})\lambda_{C^{\abelian}}$.  

\item \emph{Suppose $d=2$ and $g' \ge 2$}, then one shows that $\delta \le 2$.
Parity shows that $\delta\in \st{0,2}$, and thus we are in  case  
\ref{prym:etaledouble}  or \ref{prym:ram}; recall that it follows immediately from \eqref{L:pushpull} that $e=2$. 

\end{itemize}

Now, conversely, we show that in any of the cases \ref{prym:etaledouble}--\ref{prym:g2} of the theorem,
$\K(\iota_Z^*\lambda_C) = Z[e]$, and thus that  $P(C/C')$ is a Prym--Tryurin Prym.   Note that once we have established this, then the discussion above about the exponent applies in all of the cases.

In case \ref{prym:g2}, $Z$ is an elliptic curve, and so any polarization is
a multiple of the principal polarization. 

In case \ref{prym:ram}, since
$\delta=2$ we have by Riemann--Hurwitz that $g = 2g'$.  Moreover, 
$f$ is totally ramified at its ramification points, and thus the only
(abelian) torsor through which $C \to C'$ factors is the trivial
torsor $C'$ itself.
 Therefore, $f^*$ is injective (Lemma
\ref{L:newBLprop11.4.3}), $Y = \pic^0_{C'/K}$, and $\iota_Y=f^*$.  This means  $\iota_Y^*\lambda_C = (f^*)^*\lambda_C$, and by
 \Cref{L:BL12.3.1} we find that this is $2\lambda_{C'}$ so that  $\K(\iota_Y^*\lambda_C) =
\pic^0_{C'/K}[2]$.  Since $\dim Y = \dim Z$, by 
\Cref{L:kerpullbackpolarization} (see also \Cref{R:Ki*Theta}) we have $\K(\iota_Z^*\lambda_C)  = Z[2]$,
and $Z$ is a Prym--Tyurin variety.

Case \ref{prym:etaletriple} is similar; we have $g'=2$, so that, by Riemann--Hurwitz,  $g=4$, which implies 
that $\dim Y = \dim Z = 2$.  Since $\deg f = 3$ is prime, $f\colon C \to C'$
does not factor through a nontrivial cover of curves.  Because $f$
itself is not cyclic, $f^*$ is thus injective (Lemma
\ref{L:newBLprop11.4.3}), $Y = \pic^0_{C'/K}$, and $\iota_Y=f^*$, and so we have, again  by  \Cref{L:BL12.3.1},  that $\K(\iota_Y^*\Theta) =
\pic^0_{C'/K}[3]$.  Consequently, again by \Cref{L:kerpullbackpolarization} (see also \Cref{R:Ki*Theta}),  $\K(\iota_Z^*\lambda_C) = Z[3]$, and $Z$ is a Prym--Tyurin variety.

Finally, in case \ref{prym:etaledouble}, we have that $f$ is an \'etale double cover. 
Note that from Riemann--Hurwitz we have $g=2g'-1$.  
 From  Lemma
\ref{L:newBLprop11.4.3}
there is a factorization $f^* = \iota_Y \circ h$, where $h\colon \pic^0_{C'/K} \to Y$ is
an isogeny of degree $2$.  Since $\K(h^*\iota_Y^*\lambda_C) =
\pic^0_{C'/K}[2]$ (\Cref{L:BL12.3.1}), we have that $\K(\iota_Y^*\lambda_C)$ is a sub-group scheme
of $Y[2]$ of index $4$.  
From \Cref{R:Ki*Theta} we have $\K(\iota_Y^*\lambda_C)\cong \K(\iota_Z^*\lambda_C)$, and the containment $\K(\iota_Y^*\lambda_C)\subseteq Y[2]$ implies $\K(\iota_Z^*\lambda_C)\subseteq Z[2]$. Since $\dim Z=g-g'=g'-1$, 
 after considering the degrees of these group schemes, we see that the only
possible sub-group scheme of $Z[2]$ isomorphic to this group scheme is
$Z[2]$ itself.  Therefore, $\K(\iota_Z^*\lambda_C) = Z[2]$, and
$Z$ is a Prym--Tyurin variety of exponent $2$.
\end{proof}

As explained in \Cref{R:C/sigma}, 
Prym varieties often arise in nature from a curve with automorphisms.
  Using Theorem \ref{T:classifyprym}, it is not hard to
classify those group actions
giving rise to Prym--Tyurin Prym varieties.  
   In particular, we extract  the following converse to \Cref{R:deg2Pr},  which
is entirely classical away from characteristic two, but seems less
well-known in the case of even characteristic.

\begin{cor}[Prym--Tyurin Prym varieties associated to involutions]\label{C:PTVinv}
  Let $C$ be a smooth projective curve of genus $g$ over a field $K$, and let
  $\sigma\colon C \to C$ be a nontrivial separable involution.  Let $C' = C/\left<
    \sigma\right>$ be the quotient curve.  Then the Prym variety 
    $P(C/C')$ is a Prym--Tyurin variety 
      if and only if either:
\begin{enumerate}[label=(\alph*)]
  \item $\iota$ acts without fixed points, or
  \item
    \begin{enumerate}[label=(\roman*)]
      \item $\operatorname{char}(K)\not = 2$ and $\sigma$ has exactly two fixed points; or
        \item $\operatorname{char}(K)=2$, $\sigma$ has exactly one fixed point, and the action of
          $\sigma$ is weakly wildly ramified there, in the sense of 
      \Cref{R:wildram}, that at the  unique fixed
      point, $C$ looks formally locally like $ K\powser x[y]/(y^2-y-x)$, with local
      involution $y\mapsto y+1$.
        \end{enumerate}
      \end{enumerate}
In both cases the exponent is $e=2$.  In case (a) one has $\dim P(C/C')=g'-1=(g-1)/2$, and in case (b) one has $\dim P(C/C')=g'=g/2$.
\end{cor}

    \begin{proof}
      This is an immediate consequence of \Cref{T:classifyprym} and   \Cref{R:wildram}.
         \end{proof}

Consider a cyclic $p$-cover $C \to C'$ of curves in characteristic
$p$.  The Deuring--Shafarevich formula allows one to compute the
$p$-rank $r(C)$ in terms of the $p$-rank $r(C')$, and thus that of
$P(C/C')$.  (See \cite{shiomi11} for a precise statement and the history of this problem, and \cite{caisulmer23} for a recent and far-reaching generalization.) The $p$-rank is a partial invariant of the $p$-torsion
group scheme of an abelian variety. In the special
case of a Prym--Tyurin Prym variety, we can actually compute the
isomorphism class of $P(C/C')[p]$:

\begin{lem}
  \label{L:prymbt1}
  Let $f:C \to C'$ be a finite separable morphism of degree $d$ of smooth curves
  over a field $K$ of characteristic $p$.  Let $Y = f^* \pic^0_{C'/K}$, and let $Z =
  P(C/C')$ be the corresponding Prym variety. 
  \begin{alphabetize}
    \item Suppose $d=p=2$ and $f$ is \'etale. Then $Y[2]_{\bar K} \iso
      \pic^0_{C'/K}[2]_{\bar K}$, and there is an exact sequence of
      group schemes
      \[
        \xymatrix{0 \ar[r] & Z[2] \ar[r] & Y[2] \ar[r] & H \ar[r] & 0}\]
      where $H_{K^{\mathrm{perf}}} \iso (\mmu_2 \oplus \integ/2)$.
\item Suppose $d=p=2$ and the ramification divisor of $f$ has degree
  $2$.  Then $Z[2] \iso \pic^0_{C'/K}[2]$.
\item Suppose $d=p=3$, $f$ is \'etale and noncyclic, and $C'$ has genus $2$.  Then $Z[3]
  \iso \pic^0_{C'/K}[3]$.
\end{alphabetize}
\end{lem}

\begin{proof}
In cases (b) and (c), we have already seen in the proofs of the
corresponding cases of Theorem \ref{T:classifyprym} that $Y \iso
\pic^0_{C'}$, and that $Z[p] = \K(\iota_Z^* \lambda_C) =
\K(\iota_Y^*\lambda_C) = Y[p]$.

Thus, we now consider case (a).  Let $r(Y)$ and $r(Z)$ denote the
$p$-ranks of $Y$ and $Z$, respectively, and let $r(C)$ and $r(C')$ be
the respective $p$-ranks of $\pic^0_{C/K}$ and $\pic^0_{C'/K}$. For
later use, we recall that the Deuring--Shafarevich formula (e.g.,
\cite[Thm.~1.1]{shiomi11}) implies that $r(C) = 2 r(C')-1$.  Since the $p$-rank is an
isogeny invariant, and additive for direct sums of abelian varieties,
we find that $r(Z) = r(Y)-1 = r(C')-1$.

We have $N := \ker(f^*) \iso \mmu_2$, a group scheme of exponent two; then
\[
  Y[2] \iso \left(\frac{\pic^0_{C'}[4]}{N}\right)[2].
\]
Over $\bar K$, we have $\pic^0_{C'}[4]_{\bar K} \iso T  \oplus G$
where $T \iso (\mmu_4)^{\oplus r(C')}$ and $\hom(\mmu_2, G) = (0)$.
Therefore, $N\subset T$ and we have
\[
  (\pic^0_{C'}[4]_{\bar K}/N)[2] \iso (T/N)[2]
  \oplus G[2] \iso \pic^0_{C'}[2]_{\bar K},
\]
where the last isomorphism follows from the fact that $T/N$
is again a finite multiplicative group scheme of rank $r(C')$.

We have seen above that $\K(\iota_Z^*\lambda_{C'}) = Z[2] =
Z\times_{\pic^0_{C/K}} Y \subset Y[2]$, and so there is an exact
sequence of $2$-torsion group schemes
      \[
        \xymatrix{0 \ar[r] & Z[2] \ar[r] & Y[2] \ar[r] & H \ar[r] & 0.}\]
Rank conditions show that $H$ is a self-dual $2$-torsion group scheme
of order $2^{2 \cdot(\dim Y - \dim Z)} = 4$.   Because $r(Y) = r(Z) -1$, we have $\dim_{\ff_2}
H^\et = 1$.  Since $H$ is self-dual, $H$ is an extension of $\integ/2$
by $\mmu_2$; in particular, $H_{K^{\mathrm{perf}}} \iso \mmu_2 \oplus
\integ/2$.
\end{proof}

\begin{rem}
The analogous claim of \ref{L:prymbt1} for situation Theorem
\ref{T:classifyprym} is false in every positive characteristic.
Indeed, given a field $K$ of characteristic $p>0$, choose elliptic
curves $E_0$ and $E_1$ over $K$ with $E_0$ supersingular and $E_1$
ordinary, i.e., with $r(E_i) = i$; note that if $E'_i$ is isogenous to
$E_i$ for $i = 0, 1$, then $E'_0[p]_{\bar K}\not \iso E'_1[p]_{\bar
  K}$.  Let $A/K$ be a principally polarized abelian surface over $K$
which is isogenous to $E_0 \times E_1$, but is not \emph{isomorphic}
to a product of elliptic curves.  Then $A$ is the Jacobian of a smooth
projective curve, $C$.  Moreover, $C$ admits a nontrivial map to
$E_0$, but $P(C/E_0)$ is isogenous to $E_1$.
\end{rem}

    \section{Prym schemes}
    \label{S:prymscheme}

While the first part of this paper treated abelian schemes over
arbitrary bases, the reader will have noticed our retreat in \S
\ref{S:cyclefield} and \S \ref{S:prymfield} to objects over fields.  
We now return to working over a connected locally Noetherian scheme $S$.  For ease of notation, given a polarized abelian variety $(Z,\xi)$ over a field $K$, we will write $[\xi_{\bar K}]$ for the class $[\Xi_{\bar K}]$ for any ample divisor $\Xi_{\bar K}$ on $Z_{\bar K}$ such that $\xi_{\bar K}=\phi_{\Xi_{\bar K}}$.

\subsection{Prym--Tyurin schemes and a relative Welters' criterion}
\label{SS:pt-sch}

The starting point is a relative version of Welters' Criterion:

\begin{teo}[{Relative Welters' Criterion}]\label{T:WeltersS}
 Let $(Z,\xi)$ be a principally polarized abelian scheme of dimension
 $g_Z$ over $S$ and let $C/S$ be a smooth proper
 curve.

\begin{alphabetize}

\item  Suppose there is an $S$-morphism $\beta\colon C\to T$  to a torsor under $Z$ over $S$ such that 
\begin{enumerate}[label=(\roman*)]
\item the composition $\xymatrix{\widehat Z \ar[r]^<>(0.5){\zeta_Z}_<>(0.5)\sim& \operatorname{Pic}^0_{T/S} \ar[r]^{\beta^*}&  \operatorname{Pic}^0_{C/S}}$ is an injective $S$-homomorphism of  abelian vareties, where $\zeta_Z$ is the canonical isomorphism \eqref{E:PicTors}, and
\item for every geometric point $\bar s$ of $S$ (equivalently, for a single geometric point $\bar s$ of $S$), after identifying
  $T_{\bar s}$ and $Z_{\bar s}$, we have $\beta_{\bar s*}[C_{\bar
    s}]\equiv e \frac{\displaystyle [\xi_{\bar s}]^{g_Z-1}}{(g_Z-1)!}
  $.
   \end{enumerate}
Then the inclusion $\iota_Z:= \beta^*\zeta_Z\xi\colon Z\hookrightarrow \operatorname{Pic}^0_{C/S}$ makes $(Z,\xi)$ a Prym--Tyurin scheme of exponent $e$; i.e., $\iota_Z^*\lambda = e\xi$, so that $(Z,\xi,C,\iota_Z)$ is an embedded Prym--Tyurin scheme of exponent $e$.

\item Conversely, suppose there is an  inclusion $\iota_Z\colon Z\hookrightarrow \operatorname{Pic}^0_{C/S}$ making $(Z,\xi)$ a Prym--Tyurin scheme of exponent $e$; i.e., $(Z,\xi,C,\iota_Z)$ is an embedded Prym--Tyurin scheme of exponent $e$. Let $Y$ be the complement of $Z$.  Then, under the isomorphism 
$
\xymatrix@C=2em{
Z \ar[r]^<>(0.5){-\xi}_\sim& \widehat Z \ar[r]_<>(0.5)\sim& \operatorname{Pic}^0_{C/S}/Y,
}
$
(see \eqref{E:BL1213-2})
  $T=\operatorname{Pic}^{(1)}_{C/S}/Y$ is a torsor under $Z$, and the composition $\xymatrix{\beta \colon C\ar[r]^<>(0.5){\alpha^{(1)}}& \operatorname{Pic}^{(1)}_{C/S}\to T}$ of the Abel map with the quotient map satisfies conditions (i) and (ii) above.  
\end{alphabetize}

Moreover, these constructions are inverse to one another, up to canonical isomorphisms.
\end{teo}

\begin{proof}
It suffices to establish the claim on geometric fibers, and so this reduces to \Cref{T:Welters}.
\end{proof}

As a consequence, we get  the Matsusaka Criterion:

\begin{cor}[{Relative Matsusaka Criterion}]\label{C:WMR-S}
Let $(Z,\xi)$ be a principally polarized abelian scheme of dimension
$g_Z>0$ over $S$ and let $C/S$ be a smooth proper curve. 
 Suppose there is an $S$-morphism $\beta\colon C\to T$  to a torsor under $Z$ over $S$ such that for every geometric point $\bar s$ of $S$ (equivalently, for a single geometric point $\bar s$ of $S$), 
after  identifying $T_{\bar s}$ and $Z_{\bar s}$, we have 
\begin{equation}\label{E:WMR-S}
\beta_{\bar s*}[C_{\bar s}]\equiv \frac{[\xi_{\bar s}]^{g_Z-1}}{(g_Z-1)!}.
\end{equation}
Then the composition  $\iota_Z:= \beta^*\zeta_Z\xi\colon Z\to\operatorname{Pic}^0_{C/S}$ is an isomorphism, and there is a unique isomorphism $\gamma\colon \operatorname{Pic}^{(1)}_{C/S}\to T$ such that $\beta = \alpha^{(1)}\gamma$.  In other words, up to an isomorphism of torsors, $\beta$ is the Abel map.  
\end{cor}

\begin{proof}
It suffices to check the assertions on geometric fibers, and so this
reduces to \Cref{C:WMR}. \end{proof}

\begin{rem}[Relative Matsusaka--Ran Criterion]\label{R:MR-S}
In fact, using \cite[Thm.~p.329]{collino}, one can replace the condition \eqref{E:WMR-S} in \Cref{C:WMR-S} with the weaker condition that $\beta_{\bar s*}[C_{\bar s}].[\Xi_{\bar s}]=g_Z$.  
  \end{rem}

\subsection{Prym--Tyurin Prym schemes}
\label{SS:prymS}

We now prove our main theorem.

\begin{teo}[Classification of Prym--Tyurin Prym schemes]  \label{T:classifyprymS}
Let $f\colon C \to C'$ be a finite  
morphism of smooth proper curves over $S$ with geometrically connected fibers of respective genera $g>g'\ge 1$, and suppose that $f$ is fiberwise separable.  Let $e$ be a positive integer. 
The
  following are equivalent.
  \begin{enumerate}[label=(\alph*)]
  
  \item \label{T:ClPr-a} $P(C/C')$ is a Prym--Tyurin Prym scheme of exponent $e$; i.e., $\lambda_C|_{P(C/C')}= e\xi$ for a principal polarization $\xi$ on $P(C/C')$.
  
  \item \label{T:ClPr-as} For every point $s$ in $S$, we have that $P(C_s/C'_s)$ is a Prym--Tyurin Prym variety of exponent $e$.
  
    \item\label{T:ClPr-abars}  For every geometric point $\bar s$ in $S$, we have that $P(C_{\bar s}/C'_{\bar s})$ is a Prym--Tyurin Prym variety of exponent $e$.

    \item\label{T:ClPr-bs} For every point  $s$ of $S$,
  $f_s$    is one of the four types delineated
      in Theorem \ref{T:classifyprym} with corresponding exponent $e$.

    \item\label{T:ClPr-bbars}  For every geometric point  $\bar s$ of $S$,
                    $f_{\bar s}$ is one of the four types delineated
      in Theorem \ref{T:classifyprym} with corresponding exponent $e$.

  \item\label{T:ClPr-as'} For one  point $s$ of $S$, we have that $P(C_s/C'_s)$ is a Prym--Tyurin Prym variety of exponent $e$.

    \item\label{T:ClPr-abars'}  For one  geometric point $\bar s$ of $S$, we have that $P(C_{\bar s}/C'_{\bar s})$ is a Prym--Tyurin Prym variety of exponent $e$.
  
      \item\label{T:ClPr-bs'}  For one point  $s$ of $S$,  we have that $f_s$   is one of the four types delineated
      in Theorem \ref{T:classifyprym} with corresponding exponent $e$.
      
      \item\label{T:ClPr-bbars'}  For one geometric point  $\bar s$ of $S$,  we have that $f_{\bar s}$   is one of the four types delineated
      in Theorem \ref{T:classifyprym} with corresponding exponent $e$.

\end{enumerate}
      \end{teo}

      \begin{proof}
The equivalence of \ref{T:ClPr-a}, \ref{T:ClPr-as}, and \ref{T:ClPr-abars} (resp.~ \ref{T:ClPr-as'} and \ref{T:ClPr-abars'}) follows from \Cref{L:eXi-X[e]}.  The equivalence of \ref{T:ClPr-as} and \ref{T:ClPr-bs} (resp.~ \ref{T:ClPr-as'} and \ref{T:ClPr-bs'}) is \Cref{T:classifyprym}, as is the equivalence of \ref{T:ClPr-abars} and \ref{T:ClPr-bbars} (resp.~\ref{T:ClPr-abars'} and \ref{T:ClPr-bbars'}).  
In other words, we have \ref{T:ClPr-a}$\iff$\ref{T:ClPr-as}$\iff$\ref{T:ClPr-abars}$\iff$\ref{T:ClPr-bs}$\iff$\ref{T:ClPr-bbars} and \ref{T:ClPr-as'}$\iff$\ref{T:ClPr-abars'}$\iff$\ref{T:ClPr-bs'}$\iff$\ref{T:ClPr-bbars'}.   

We complete the proof by showing the equivalence of \ref{T:ClPr-bbars} and \ref{T:ClPr-bbars'}.  The degree of the finite flat morphism $f$ is constant on geometric fibers.  
         As a consequence,  
 Riemann--Hurwitz gives that the degrees of
 the ramification divisors are also constant.  All of the
 conditions in \Cref{T:classifyprym}, with the exception of the condition for a noncyclic cover in 
 \Cref{T:classifyprym}\ref{prym:etaletriple}, are stated in terms of
 the genera of the curves, the degrees of the covers, and the degrees
 of the ramification divisors.  
 The equivalence of \ref{T:ClPr-bbars} and \ref{T:ClPr-bbars'} then follows in all cases of  \Cref{T:classifyprym} once we recall that the degree (over $C'$) of the maximal \'etale abelian subcover of $C'' \to C'$ is constant on $S$  (\Cref{L:maxabcover}).
\end{proof}

\begin{rem}   In fact, the proof of Theorem \ref{T:classifyprymS} shows that if,
  for some point $s\in S$, $j_s$ is of type (t) for one of the four
  types delineated in Theorem \ref{T:classifyprym}, then the same is
  true for \emph{every} point $s'\in S$.  Note that the good reduction
  of an \'etale cover stays \'etale, while in mixed characteristic, it
  is possible for a degree $2$ cover ramified in exactly two geometric
  points to specialize to a degree $2$ cover weakly wildly ramified at a single
  geometric point.
\end{rem}

We now turn to Prym--Tyurin Prym schemes associated to involutions.  
 Note that we restrict the genus of $C$ to be $g\ge 3$ to ensure that the genus of the associated quotient curve is $g'\ge 2$; in this way we avoid the issue of whether the associated quotient $S$-curve exists as a scheme, and not only as an algebraic space (\Cref{R:C/sigma}).  This excludes case \ref{prym:g2}  from \Cref{T:classifyprym} (which is the same as case \ref{prym:ram} if  $g=2$), and consequently, we are excluding case \ref{prym:ram}  from \Cref{C:PTVinv} when $g=2$.  
   
\begin{cor}[Prym--Tyurin Prym schemes associated to involutions]\label{C:PTVinvS}
  Let $C$ be a smooth projective curve of genus $g\ge 3$ over  $S$, and let
  $\sigma\colon C \to C$ be a nontrivial involution such that for every geometric point $\bar s$ of $S$ we have that $\sigma_{\bar s}$ is separable.  Let $C' = C/\left<
    \sigma\right>$ be the quotient curve.  Then the Prym scheme  
    $P(C/C')$ is a Prym--Tyurin scheme 
      if and only if for every geometric point $\bar s$ of $S$ we have that $\sigma_{\bar s}$
   is one of the two types delineated
      in \Cref{C:PTVinv}.  In both cases the exponent is $e=2$; in case (a) one has $\dim_S P(C/C')=g'-1=(g-1)/2$, and in case (b) one has $\dim_S P(C/C')=g'=g/2$.     
                   \end{cor}

    \begin{proof}
      This is an immediate consequence of \Cref{C:PTVinv}.
         \end{proof}

\section{Applications}
\label{S:applications}

In this section, we would like to briefly indicate how the tools
developed here allow for easy extensions of some results about abelian
varieties.  We restrict ourselves here to giving an impressionistic
summary of the available results, and hope that the interested reader
will follow the references for more details.  (Better yet, we hope
that the interested reader will develop entirely new applications we have not 
anticipated!)

\subsection{Prym--Tyurin varieties via Hecke algebras}

We have seen that a principally polarized abelian variety is a
Prym--Tyurin variety of arbitrarily large exponent; consequently, it is
of some interest to produce Prym--Tyurin varieties of small exponent.

One method, developed by Carocca et.~al., is as follows.  Consider a
$G$-Galois cover $f:D \to C$ of smooth projective curves over $\cx$.
If $H\subseteq G$ is a subgroup, then the Hecke algebra
$\rat[H\backslash G / H]$ acts on the Jacobian of the intermediate
curve $E = D/H$. To a rational irreducible representation of $G$ one
associates an idempotent of the Hecke algebra. The main result of
\cite{caroccaetalhecke09} is a sufficient condition for the image of
this idempotent to be a Prym--Tyurin variety.  The criterion is in
terms of the restriction of this data to $H$ and the ramification data
of the covering map $f$; and part of the formula there is an explicit
calculation of the Prym--Tyurin exponent which arises.

It turns out that their main theorem is valid for a cover of curves
over an arbitrary field $K$, provided the action of $G$ is defined
over $K$ and the cover is tamely ramified.  The tameness condition is
necessary for the genus calculation in
\cite[Prop.~4.8]{caroccaetalhecke09}, and for the description of
intermediate covers in \cite{rojas07}.

More generally, it seems likely to us that many constructions
involving decompositions of abelian varieties pass from the complex
numbers to arbitrary fields without serious difficulties.

\subsection{Point counts on Pryms}

Consider an abelian variety $X$ over a finite field $\ff_q$; then Weil
gives bounds on $\#X(\ff_q)$.  If $X$ is actually the Jacobian of a
curve $C$, then information about the curve (such as gonality) can be
exploited to give bounds on $\#X(\ff_q)$ which are often strictly stronger
than the usual Weil bounds \cite{lachaudmartindeschamps}.  In the case
where $q$ is \emph{odd} and $X$ is a classical Prym, i.e., a
Prym--Tyurin variety attached to an \'etale double cover of curves,
Perret \cite{perret06} and then Aubry and Haloui \cite{aubryhaloui16} used a similar strategy to bound the
number of points.  The only reason we can surmise for restricting to
odd characteristic is perhaps some uncertainly concerning the status of
the Prym construction in characteristic two, which might be traced to
the original restrictions in Mumford's paper \cite{mumford74}.  We
are happy to report that the main results of
\cite{aubryhaloui16,perret06} are valid over fields of even
characteristic as well.

More generally, constructions involving the Prym variety of an \'etale
double cover of curves can be used without fear in even
characteristic.

\subsection{Moduli spaces}

In \cite{fabervandergeer04}, Faber and van der Geer use the classical Prym
construction to study complete subvarieties of $\mathcal M_{g',k}$, the
moduli space of curves of genus ${g'}\ge 2$ over a field $k$ of
characteristic $p>2$.  To this end, they work with $\mathcal R_{g'}$,
the moduli space of \'etale double covers $C \to C'$, where $C'$ has
genus $g'$; this stack is naturally  a smooth Deligne--Mumford stack over $\operatorname{Spec}\mathbb Z[\frac{1}{2}]$ (e.g., \cite[\S (6.5.1)]{beauvilleschottky}, \cite[\S 5]{DM69}).
Faber and van der Geer work extensively with the Prym map $\mathcal R_{g'} \to
\mathcal A_{g'-1}$, which sends such a cover to the principally polarized Prym variety $(P(C/C'),\xi)$ \cite[\S (6.2)]{beauvilleschottky}. 
The construction of the Prym map is made via \cite[Cor~VI$_B$4.4]{sga3-1} using the universal family, and (implicitly) the fact that away from characteristic two, the norm map has reduced, thus smooth, kernel.
   The Prym map for a family of covers
$C/C'$ over a general base $S$ is then defined via the composition $S\to \mathcal R_{g'}\to \mathcal A_{g'-1}$.  
A result such as ours,
involving complements in arbitrary families, provides perhaps a more direct construction of the Prym map.
  Moreover, we can actually define the moduli space $\mathcal R_{g'}$, 
 and
the morphism $\mathcal R_{g'} \to \mathcal A_{g'-1}$, over $\operatorname{Spec}\integ$.

More precisely, 
define a category fibered in groupoids (CFG) $\mathcal R_{g'}$ over the \'etale site of schemes (over $\mathbb Z$) in the following way.  
Over a scheme $S$, we define $\mathcal R_{g'}(S)$ to be the category
of pairs $(C/S,\sigma)$ where $C/S$ is a smooth projective curve of genus $2g'+1$,  and $\sigma:C\to C$ is a nontrivial $S$-involution of $C$ such that
   every geometric fiber of the quotient map $C\to
C':= C/\langle \sigma\rangle$ is an \'etale double cover.
           Morphisms in $\mathcal R_{g'}$ are defined via pull-back in the obvious way.  There are natural forgetful functors of CFGs 
$F_{2g'+1}:\mathcal R_{g'}\to \mathcal M_{2g'+1}$ and $F_{g'}: \mathcal R_{g'}\to \mathcal M_{g'}$.

\begin{pro}\label{P:RgZZ}
  The CFG  $\mathcal R_{g'}$ is a Deligne--Mumford stack over  the \'etale site of schemes (over $\mathbb Z$), and each fiber of $\mathcal R_{g'} \to \spec(\integ)$ is geometrically irreducible.
\end{pro}

\begin{proof}
  Beauville's proof (\cite[(6.5.1)]{beauvilleschottky}) (where one
  allows stable curves, but restricts to schemes over $\integ[1/2]$)
  is also valid in this context, too.  
  The key point is to show that the forgetful functor
  $F_{2g'+1}:\mathcal R_{g'}\to \mathcal M_{2g'+1}$ is representable,
  as it is a standard result that this  implies  that $\mathcal
  R_{g'}$ is a DM stack (see, e.g., \cite{CMW18} for a detailed
  exposition).  Concretely, let $S$ be a scheme and $C/S$ a relative
  curve of genus $2g'+1$, determined by a morphism $S \to \mathcal
  M_{2g'+1}$.  We must show that the fibered product $\mathcal
  R_{g'}\times_{\mathcal M_{2g'+1}} S$ is representable by a scheme.
  Insofar as this fibered product represents the functor of a special
  kind of automorphism of $C$, we may proceed as follows.  There is a chain
  of functors $\aut_S(C) \supset \Inv_S(C)  \supset
  \Inv_S^\circ(S) \supset \Inv_S^{\text{\'et}}(C)\iso \mathcal
  R_{g'}\times_{\mathcal M_{2g'+1}} S$ defined as, respectively, the
  automorphisms of $C$; the involutions of $C$; the fiberwise
  nontrivial involutions of $C$; and the fiberwise \'etale nontrivial
  involutions of $C$.  We now show successively that each of these
  functors is representable (by a scheme).

  Since $2g'+1 \ge 5 > 1$, $C$ is projective over $S$ and so the functor $\aut_S(C)$ is representable.  The
  composition $\operatorname{Aut}_S(C)\to
  \operatorname{Aut}_S(C)\times_S \operatorname{Aut}_S(C)\to
  \operatorname{Aut}_S(C)$ of the diagonal with composition of
  automorphisms shows, via pull-back of the identity, that
  $\operatorname{Inv}_S(C)$ is representable by a closed subscheme of
  $\operatorname{Aut}_S(C)$.
A similar argument shows that the sub-functor $\operatorname{Inv}^\circ_S(C)$ of non-trivial involutions is representable by an open subscheme of $\operatorname{Inv}_S(C)$.  Finally, since the \'etale locus is open, 
the sub-functor $\operatorname{Inv}^{\text{\'et}}_S(C)$ of non-trivial
involutions where the quotient map gives an \'etale double cover  is
representable by an open subscheme of
$\operatorname{Inv}^\circ_S(C)$.  Therefore, $\mathcal R_{g'}$ is a DM stack.

  For an alternative approach to showing that $\mathcal R_{g'}$ is a
  DM stack, one can directly show that the forgetful morphism $F_{g'}:\mathcal R_{g'}\to \mathcal M_{g'}$ is representable.  For this,  one can appeal to \cite[Prop.~3.2]{antei11}, which is
  the relative version of a special case of \Cref{L:newBLprop11.4.3}: if
  $\varpi: C' \to S$ is a smooth proper curve, then there is a
  bijection between nontrivial \'etale double covers of $C'$ and nonzero
  sections of $R^1\varpi_*(\mmu_{2,C'})$;
      and these sections are
  canonically identified with nonzero sections of $\pic^0_{C'/S}[2]$.

  This perspective allows us to quickly prove geometric irreducibility
  of $\mathcal R_{g'}$ in characteristic two. All other
  characteristics are settled in
  \cite[p.181]{beauvilleschottky}; see also \cite[Thm.~5.13]{DM69}.  Let $\mathcal
  M_{g',\ff_2}^\ord$ denote the ordinary locus in $\mathcal
  M_{g',\ff_2}$ (i.e., the locus where the $2$-torsion is of full
  rank); it is open and dense (e.g., \cite[Thm.~2.3]{fabervandergeer04}). 
     Let $\mathcal R_{g',\ff_2}^\ord =
  \mathcal R_{g',\ff_2}\times_{\mathcal M_{g',\ff_2}} \mathcal
  M_{g',\ff_2}^\ord$ be the open substack parametrizing \'etale double
  covers $C \to C'$ where $C'$ is ordinary.  Then $\mathcal
  R_{g',\ff_2}^\ord$ is again dense in $\mathcal R_{g',\ff_2}$, and so it suffices to show that
  $\mathcal R_{g',\ff_2}^\ord$ is geometrically irreducible.  To see
  this density, 
  it suffices to show that if $K$ is any algebraically closed field of
  characteristic two, and if $C'/K$ is a curve of genus $2g'+1$ and
  and $\chi$ is a section of $\pic^0_{C'/K}[2]^\et$, then the data
  $(C',\beta)$ deforms to the data of an ordinary curve and a point of
  order two on its Picard scheme.  Using the density of $\mathcal
  M_{2g'+1}^\ord$, we may find an equicharacteristic discrete
  valuation ring $S$ and a curve $\widetilde C'/S$ with ordinary
  generic fiber and with special fiber isomorphic to $C'$. Let $G$ be the
  finite flat group scheme $\pic^0_{\widetilde C'/S}[2]$.  After a
  suitable fppf base extension $T \to S$ (in fact, $G \to S$ itself
  will do) there is a section of $G_T \to T$ which specializes to $\chi$.

  Now let $\mathcal I_{g',\ff_2}$ be the moduli stack of data
  $(C' \to S, \alpha\colon  \pic^0_{C'/S}[2]^\et \stackrel{\sim}{\to}
  (\integ/2)_S^{g'})$.  On one hand, a monodromy calculation
  \cite[Thm.~2.1]{ekedahlmonodromy} shows that $\mathcal I_{g',\ff_2}$ is
  geometrically irreducible.  On the other hand, the forgetful functor
  $\mathcal I_{g',\ff_2} \to \mathcal M_{g'}^\ord$ admits a factorization through
  $\mathcal R_{g',\ff_2}^\ord$, say by letting the isomorphism $\alpha$
  determine the section $\alpha\inv([1, 0, \dots, 0])$; and since any
  section of $\pic^0_{C'/S}[2]^\et$ can, after fppf base extension
  $T\to S$, be completed to a basis of sections, this morphism is surjective.  Therefore,
  $\mathcal R_{g',\ff_2}^\ord$  is geometrically irreducible; by
  density, so is $\mathcal R_{g',\ff_2}$.
\end{proof}

For clarity, we summarize our discussion with the following statement.  (Note 
that outside of characteristic $2$, this is proven in
\cite{beauvilleschottky}; there the Prym map is defined \emph{a
priori} over reduced versal spaces, and then \emph{a posteriori} over
arbitrary bases via the induced morphisms to versal spaces.)

\begin{pro}
There is a Prym morphism of stacks over $\operatorname{Spec} \mathbb Z$
\[
\xymatrix{P:\mathcal R_{g'}\ar[r] &  \mathcal A_{g'-1}}\]
$$
(C\to C') \mapsto P(C/C')
$$
\end{pro}

\begin{proof}
The construction of $\mathcal R_{g'}$ is \Cref{P:RgZZ}; the
existence of the Prym morphism is extracted from our main theorem in \Cref{C:PTVinvS}.
\end{proof}

\subsection{Group schemes are summands of the torsion of Jacobians}

The titular claim of \cite{priesulmer22} is that every $BT_1$ group
scheme appears in a Jacobian.  We can quickly recover much of
that abstract statement, and extend it, as follows.

\begin{lem}
  Let $K$ be a field of characteristic $p>0$, and let $Z$ be a
  principally polarized
  abelian variety with $g :=  \dim Z < p$.  Then there exists a smooth
  projective curve $C/K$ such that, for each $m \ge 1$, $Z[p^m]$ is a
  summand of $\pic^0_{C/K}[p^m]$ as a principally-quasipolarized
  finite group scheme.
\end{lem}

(In fact, when $K = \bar K$ and $m=1$, \cite[Thm.~1.1]{priesulmer22}
is much more precise, in that the authors show that for $C$ one may take a quotient of a Fermat curve whose degree is explicitly bounded in
terms of $g$.)

\begin{proof}
Fix some $n\ge 3$ with $p\nmid n$.  By Corollary \ref{C:PPAV=PT} and
Remark \ref{R:exponentbound}, there exists a smooth projective curve
$C/K$ such that $Z$ is a sub-abelian variety of $\pic^0_{C/K}$
of exponent $e= n^r(g-1)!$ for some $r$, and $\Theta_C\rest Z= e\Xi$.  Let $Y$ be the complement of
$Z$.  Since $\gcd(p,n^r(g-1)!) = 1$, the isogeny $\mu\colon Y \oplus Z \to \pic^0_{C/K}$ induces,
for each $m$, an isomorphism $Y[p^m] \oplus Z[p^m]
\stackrel{\sim}{\to} \pic^0_{C/K}[p^m]$.
\end{proof}

 \bibliographystyle{amsalpha}
 \bibliography{DCG}

\end{document}